\documentclass[11pt]{amsart}
\usepackage{graphicx} % Required for inserting images
\usepackage{longtable}
\usepackage{lscape}
\usepackage{amsmath, amsfonts, amssymb, amsthm, dsfont, mathrsfs, wasysym, graphics, graphicx, url, hyperref, hypcap, makecell,
xargs, multicol, pdflscape, multirow, hvfloat, array, ae, aecompl, pifont, mathtools, a4wide, float, blkarray, overpic, nicefrac, stmaryrd, anyfontsize, yfonts, array, tabularx, fontawesome, paralist}

\usepackage{tikz,tikz-cd}
\usepackage{csvsimple}
\usepackage{svg}
\usepackage{fullpage}
\usepackage[margin=1in]{geometry}
\usepackage{booktabs}
\usepackage{comment}

\usepackage[noabbrev,capitalise]{cleveref}

\numberwithin{equation}{section}
\theoremstyle{plain}
\newtheorem{thmUniv}{Theorem}

\newtheorem{theorem}{Theorem}[section]

\newtheorem{lemma}[theorem]{Lemma}
\newtheorem{conjecture}{Conjecture}

\newtheorem{corollary}[theorem]{Corollary}
\newtheorem{proposition}[theorem]{Proposition}

\newtheorem{formulation}{Formulation}
\theoremstyle{definition}
\newtheorem{definition}[theorem]{Definition}
\newtheorem{remark}[theorem]{Remark}

\newtheorem{example}[theorem]{Example}
\newtheorem{question}[theorem]{Question}

\AddToHook{env/theorem/begin}{\crefalias{section}{theorem}}
\AddToHook{env/proposition/begin}{\crefalias{theorem}{proposition}}
\AddToHook{env/corollary/begin}{\crefalias{theorem}{corollary}}
\AddToHook{env/lemma/begin}{\crefalias{theorem}{lemma}}
\AddToHook{env/conjecture/begin}{\crefalias{theorem}{conjecture}}
\AddToHook{env/definition/begin}{\crefalias{theorem}{definition}}
\AddToHook{env/example/begin}{\crefalias{theorem}{example}}
\AddToHook{env/remark/begin}{\crefalias{theorem}{remark}}
\AddToHook{env/observation/begin}{\crefalias{theorem}{observation}}
\AddToHook{env/question/begin}{\crefalias{theorem}{question}}
\AddToHook{env/notation/begin}{\crefalias{theorem}{notation}}
\AddToHook{env/assumption/begin}{\crefalias{theorem}{assumption}}
\AddToHook{env/convention/begin}{\crefalias{theorem}{convention}}
\newcommand{\good}{$\c{AO}$-Hamiltonian}
\newcommand{\V}{N}
\newcommand{\Arcs}{F}
\newcommand{\node}{node} 
\newcommand{\nodes}{nodes} 
\newcommand{\arc}{arc} 
\newcommand{\arcs}{arcs} 
\newcommand{\combdir}{\wedge}

\usepackage[T1]{fontenc}
\usepackage{bbm}
\usepackage[shortlabels, inline]{enumitem}
\usepackage{pdflscape}
\usepackage[normalem]{ulem}
\usepackage{marginnote}

\hypersetup{colorlinks=true, citecolor=darkblue, linkcolor=darkblue}

\newcommand{\R}{\mathbb{R}} % reals
\newcommand{\N}{\mathbb{N}} % naturals
\newcommand{\Z}{\mathbb{Z}} % integers
\newcommand{\C}{\mathbb{C}} % complex
\renewcommand{\c}[1]{{\mathcal{#1}}} % call letters
\renewcommand{\b}[1]{{\boldsymbol{#1}}} % bold letter
\newcommand{\Bool}[1][n]{\c Bool_{#1}}
\newcommand{\cube}[1][n]{Q_{#1}}
\newcommand{\WeakOrder}[1][n]{\c Weak_{#1}}

\renewcommand{\emptyset}{\varnothing} % prettier emptyset
\renewcommand{\epsilon}{\varepsilon} % prettier epsilon

\newcommand{\ssm}{\smallsetminus} % small set minus

\newcommand{\one}{{1\!\!1}} % the all one vector
\newcommandx{\ones}[1][1=n]{\one_{#1}} % the all one vector of length n
\newcommand{\eqdef}{\coloneqq}%{\mbox{\,\raisebox{0.2ex}{\scriptsize\ensuremath{\mathrm:}}\ensuremath{=}\,}} % :=
\DeclareMathOperator{\cons}{cons} % nbr facets

\DeclareMathOperator{\rank}{rk}

\DeclareMathOperator{\Aut}{Aut}

\newcommand{\ie}{\textit{i.e.}~} % id est
\newcommand{\eg}{\textit{e.g.}~} % exempli gratia
\newcommand{\aka}{\textit{a.k.a.}~} % also known as
\definecolor{darkblue}{rgb}{0,0,0.7} % darkblue color
\definecolor{green}{RGB}{57,181,74} % green color
\definecolor{violet}{RGB}{147,39,143} % violet color
\newcommand{\darkblue}{\color{darkblue}} % darkblue command
\newcommand{\defn}[1]{\textsl{\darkblue #1}} % emphasis of a definition
\newcommand{\mathdefn}[1]{{#1}} % emphasis of a definition
\newcommand{\symmdiff}{\triangle}

\newcommand{\pol}[1][P]{\mathsf{#1}}%polytope

\newcommand{\polZ}{\pol[Z]}

\newcommand{\AO}[1][G]{\c{AO}(#1)}
\newcommandx{\AOD}[1][1 = D]{\c{AO}(#1)}
\newcommandx{\AODk}[2][1 = D, 2 = k]{\c{AO}_{#2}(#1)}
\newcommand{\CO}[1][G]{\c{CO}(#1)}
\newcommandx{\COD}[1][1 = D]{\c{CO}(#1)}
\newcommandx{\CODk}[2][1 = D, 2 = k]{\c{CO}_{#2}(#1)}
\newcommandx{\M}[2][1=\b A, 2 = H']{\AO[{#2|#1}]}

\newcommandx{\ZG}[1][1=G]{\polZ_{#1}} %Graphical zonotope
\newcommandx{\ZGKnm}[2][1=n,2=m]{\polZ_{#1, #2}}

\newcommand{\yes}{\textcolor{green}{\textbf{\ding{51}}}}
\newcommand{\unknown}{\textcolor{red}{\textbf{$\leftarrow$}}}

\makeatletter
\newcommand{\sectionnotoc}[1]{%
  {\let\addcontentsline\@gobblethree
   \section*{#1}}%
}

\newcommand{\subsectionnotoc}[1]{%
  {\let\addcontentsline\@gobblethree
   \subsection{#1}}%
}
\makeatother

\newcommand{\circled}[1]{%
  \tikz[baseline=(n.base)]{
    \node (n) {$#1$};
    \ifnum#1<10
      \draw (n) circle (5.25pt);
    \else
      \draw (n) circle (6.5pt);
    \fi
  }%
}

\usepackage{todonotes}

\newcommand{\OEIS}[1]{\href{https://oeis.org/#1}{\cite[#1]{OEIS}}} % oeis

\AtEndDocument{\vfill{\footnotesize%
\textsc{Universit\"at Kassel, Germany} \\
Leonie M\"uhlherr: \texttt{leonie.muehlherr@mathematik.uni-kassel.de} 

\textsc{Universit\"at Osnabrück, Germany} \\
Germain Poullot: \texttt{germain.poullot@uni-osnabrueck.de}
}}

\title{Hamiltonicity of graphs of acyclic orientations\\and acyclic polynomials}
\author{Leonie M\"uhlherr and Germain Poullot}
\date{\today}

\usepackage[style=alphabetic,doi=false,isbn=false,url=false,eprint=true,
backend=biber,giveninits=true,maxalphanames=5,maxcitenames=3,maxbibnames=8,hyperref]{biblatex}

\begin{document}

\begin{abstract}
We study the graph $\AO$ of acyclic orientations of a graph $G$.
Two acyclic orientations are adjacent in this graph if they disagree on the orientation of a single arc.
In particular, we focus on the Hamiltonicity of the graphs $\AO$.

Using two methods of pattern lacing which generalize the zig-zag method of Brenner, Cardinal, McConville, Merino and M\"utze, we characterize which multipaths are {\good}.
Moreover, we give a criterion for the gluing of a multipath on a given graph to preserve $\c{AO}$-Hamiltonicity.
Building towards an inductive certification of $\c{AO}$-Hamiltonicity via the (open) ear decomposition of 2-connected graphs, we propose three ways of gluing several multipaths to a given graph.

In addition, we define the acyclic polynomials to encapsulate both the number of acyclic orientations of a graph and the ``parity problem'' proposed by Savage, Squire and West:
if $-1$ is not a root of the acyclic polynomial of $G$, then $G$ is not {\good}.
We explore numerous properties of the acyclic polynomials, proving that they are not instances of the famous Tutte--Whitney polynomials, but that they too exhibit a partial deletion-contraction phenomenon.
\end{abstract}

\maketitle

\vspace{-1.15cm}
\tableofcontents
\vspace{-1.33cm}

% \listoffigures

% \clearpage
\sectionnotoc{Acknowledgements}

The authors deeply thank Martina Juhnke for welcoming LM many times in Osnabr\"uck, as well as Mieke Fink for asking about what became \Cref{prop:SumAcyclicPolynomials}, and Torben Donzelmann and Thiago Holleben for motivating us to write \Cref{thm:ChordalGraphsAreGammaAlternating}.
Furthermore, the authors are grateful to Jean Cardinal, Torsten M\"utze, Vincent Pilaud and Francesco Verciani for many interesting comments, discussions, and literature recommendations.
We acknowledge the CRM conference ``Combinatorial Geometries and Geometric Combinatorics'' and the PAGCAP workshop ``Beyond Permutahedra and Associahedra'', in late 2025, where this project started.

LM was supported by German Science Foundation grant 522790373.

\sectionnotoc{Use of AI declaration}

The only use we made of AI (namely ChatGPT using GPT-5.5) was to correct mistakes in English writing, to search the literature and to improve some minor explanations of the existing literature (\eg replacing ``the number of maximal cones at a given distance from a base maximal cone in the dual graph of a hyperplane arrangement'' by ``the Whitney numbers''), to improve the readability of the table in Appendix~\ref{app:GraphExamples}, and to parallelize our implementation of the acyclic signature in order to obtain the results of \Cref{exm:ClassesOfNonAOHgraphs}.
All statements, proofs, concepts, notations, figures, references to the literature, etc., were made by us, as humans, using non-AI-powered tools (as far as we are aware).
Notably, computations of Hamiltonian cycles and acyclic polynomials were run using SageMath~\cite{Sage} and our own pieces of code, not relying on AI-powered computations, except for the parallelization required in \Cref{exm:ClassesOfNonAOHgraphs} for graphs of more than $22$ edges.

% \newpage
\vspace{0.25cm}
\section{Introduction}\label{sec:Intro}

The motivation for this article spans across various mathematical disciplines.
At its core, the question we want to answer is for which graphs the acyclic orientations can be enumerated in a nice way.
One can make this statement more precise by formulating it via Gray codes, which are enumerations where consecutive objects differ only by a small change.
The original and most common example is the enumeration of binary strings, see \cite{Knuht2011-ArtOfProgramming} for more background.
Acyclic orientations of a graph can be encoded as binary strings, each bit coding for an edge, and $0$,~$1$ representing its two possible orientations.
Yet, in this representation, it is hard to check whether a binary string encodes an acyclic orientation or not.
We are interested in the following problem.

\begin{formulation}
Classify the graphs $G$ for which the list of acyclic orientations admits a Gray code. 
\end{formulation}

The reader unfamiliar with this version of the problem may take a look at \Cref{tab:CubeGraph4,tab:PermutahedralGraph4} for some typical examples. 
In this article, the main perspective on this question is the graph-theoretical one: 
We study the graph of acyclic orientations $\AO$ whose nodes are the acyclic orientations of a graph $G$, and where two orientations are neighboring if and only if they differ in flipping the orientation of exactly one arc.
We can restate our problem:

\begin{formulation}\label{for1}
Classify the graphs $G$ whose graph of acyclic orientations $\AO$ is Hamiltonian. 
\end{formulation}

This is the research angle of Savage, Squire and West in their seminal work \cite{SavageSquireWest93-GrayCodesAcyclicOrientations}.
They achieved a classification for some notable example classes, such as cycles, wheel-graphs, ladder-graphs, and chordal graphs (see \cite{SavageSquireWest93-GrayCodesAcyclicOrientations} or here below \Cref{exm:KnownAOhamiltonicity}).
This quest for efficient ways to produce all the acyclic orientations of a graph was then pursued by numerous authors,
see \eg Squire \cite{Squire1998-GeneratingAcyclicOrientations} and Barbosa--Szwarcfiter \cite{BarbosaSzwarcfiter1999-GeneratingAcyclicOrientations} for generating problems.
Most prominently, Pruesse and Ruskey \cite{PruesseRuskey1995-PrismOfAcylicOrientationGraph} showed that the prism over the graph of acyclic orientations of any graph is Hamiltonian.
Furthermore, Pilaud \cite{Pilaud2024-AcyclicReorientations} studied the properties of the poset obtained from the graph of acyclic orientations of $G$ by orienting it away from a base (acyclic) orientation $D$.
Remarkably, he determined under which conditions on $G$ and $D$ it is a lattice.
In particular, the rank function of this poset gives the coefficients of the acyclic polynomials that we study in \Cref{sec:NonHamiltonicity}.
We also refer to Savage's survey \cite{Savege97-SurveyGrayCodes} for a presentation of the famous Steinhaus--Johnson--Trotter algorithm \cite{Johnson63-GenerationOfPermutations} which produces a Hamiltonian cycle in the graph of acyclic orientations for the complete graph.
The interested reader may consult \cite[Figure~5]{BrennerCardinalMcConvilleMerinoMutze25-CombinatorialGenerationV} for a comprehensive survey on algorithms and Hamiltonicity of graphs of acyclic orientations of combinatorial families beyond graphs.

% , Pruesse--Ruskey \cite{PruesseRuskey1995-PrismOfAcylicOrientationGraph} for fundamental properties of the graph of acyclic orientations, Pilaud \cite{Pilaud2024-AcyclicReorientations} for the characterization of the graphs of acyclic orientations which turn into lattices once chosen a base orientation, or Savage's survey \cite{Savege97-SurveyGrayCodes} for a presentation of the famous Steinhaus--Johnson--Trotter algorithm \cite{Johnson63-GenerationOfPermutations} which produces a Hamiltonian cycle in the graph of acyclic orientations for the complete graph.

The results of Savage, Squire and West strongly suggest that parity considerations (such as ``Is $C_n$ a cycle with $n$ even or odd?'') play a key role in determining whether $\AO$ is Hamiltonian or not.
In particular, they establish a simple yet powerful criterion for deciding that $\AO$ is not Hamiltonian:
fixing any orientation $D$ of $G$, if $\AO$ is Hamiltonian, then the number of acyclic orientations of $G$ disagreeing with $D$ on the orientation of an \emph{even} number of arcs is equal to the number of acyclic orientations of $G$ disagreeing with $D$ on an \emph{odd} number of arcs.
When a graph does not satisfy this condition, Savage, Squire and West \cite{SavageSquireWest93-GrayCodesAcyclicOrientations} speak of a \emph{parity problem} (see also \Cref{lem:ParityArgument}).

As a tribute to their foundational work, we extend their methods in \Cref{ssec:Prelude} to study Cartesian products of graphs, and prove:

\begin{thmUniv}[{\Cref{cor:CubePrismsEtcAreNotAOH}}]\label{thmZ}
The following graphs are not {\good}:
\begin{compactenum}
\item Ladder graphs $L_n \eqdef P_n \times K_2$ for any $n$ even
\item Book graphs $B_n \eqdef K_{1, n} \times K_2$ for any $n$ odd
\item Prism graphs $\mathrm{Pr}_n \eqdef C_n\times K_2$ for any $n$ even
\item Toroidal grid graphs $T_{4, m} \eqdef C_4\times C_m$ for any $m$ even
\item Cube graphs $Q_n = K_2 \times \dots\times K_2$, with $n$ copies of $K_2$, for any $n\geq 2$,
\item More generally, the graph $T\times Q_n$ for any tree $T$ where either $n \geq 2$ or $T$ has an even number of nodes,
\item Thick wheel graphs $\mathrm{TW}_n \eqdef W_n \times K_2$ for any $n$ odd, where $W_n$ is $C_n$ together with a node adjacent to all nodes of $C_n$,
\item Any $(T,\, n_1, \dots, n_s)$-tree of cycles for any tree $T$ and any $n_1, \dots, n_s\geq 1$, if there is some $i\in [d]$ such that $n_i$ is even and if $\sum_{i=1}^s (n_i-1)$ is odd.
\end{compactenum}
\end{thmUniv}

This result motivates us to define the rank with respect to $D$ of an acyclic orientation $A$, denoted $\rank_D A$, as the number of arcs on the orientation of which $A$ and $D$ disagree.
We encapsulate the numbers of acyclic orientations of each rank in a polynomial $\Psi(D; t) \eqdef \sum_{A\in \AO} t^{\rank_D A}$ called the \emph{acyclic polynomial} (\Cref{sec:NonHamiltonicity}).
In particular, constructing the \emph{acyclic signature} as $\sigma(D) \eqdef \Psi(D; -1)$ (and $\sigma(G) \eqdef |\sigma(D)|$ for any orientation $D$ of $G$), we can rephrase the parity problem as deciding whether $\sigma(D) = 0$ or not.
See \Cref{tab:SignatureEasyFacts} for an overview, and \Cref{ssec:AcyclicSignature} for various properties of the acyclic signature. 

Our acyclic polynomial should not be confused with the \emph{acyclic orientation polynomial} recently defined in \cite{HwangJungLeeOhYu2021-OtherAcylicPolynomial}: there, the authors assign a variable to each node of a graph, and associate each acyclic orientation to the product of the variables assigned to the sinks of the orientation.

\begin{table}
\centering
\renewcommand{\arraystretch}{1.3}
\begin{tabular}{lll}
\toprule
Graph $G$ & $\sigma(G)$ & Proof by \\ 
\midrule
Tree & $0$ & \Cref{exm:AcyclicPolynomialTrees} \\
Cycle $C_n$ & $2 \cdot (n+1 \mod 2) $ & \Cref{exm:AcyclicPolynomialCycles} \\
Complete bipartite graph $K_{2, 3}$ & $6$ & \Cref{exm:AcyclicPolynomialK23} \\
Complete bipartite graph $K_{m, n}$ & $0$ iff $m$ and $n$ odd & Appendix~\ref{app_bipartite} \\
Complete graph $K_n$ & $0$ & \Cref{exm:AcyclicPol:CompleteGraphs} \\ 

Graph $G$ with connected components $\V_1, \dots, \V_r$ & $\prod_k\sigma(G[\V_k])$ & \Cref{prop:AcyclicPol:Disconnected} \\
$1$-sum $G = H\oplus_1 H'$ & $\sigma(H) \cdot \sigma(H')$ & \Cref{prop:AcyclicPol:1sums} \\
$G$ with simplicial node $v$ with $b$ neighbors & $\sigma(G\ssm\{v\}) \cdot (b+1 \mod 2)$ & \Cref{lem:AcyclicPol:SimplicialNodeAdd} \\
Chordal graph & $0$ & \Cref{prop:AcyclicPol:ChordalGraph} \\
Odd number of edges & $0$ & \Cref{cor:OddEdgeNbImpliesSignature0} \\ 
\bottomrule
\end{tabular}
\renewcommand{\arraystretch}{1}
\caption{First examples and first properties of the acyclic signature.}
\label{tab:SignatureEasyFacts}
\end{table}

Stanley \cite{Stanley-acyclicOrientations} famously proved that the number of acyclic orientations of a graph $G$ can be recovered from its chromatic polynomial $\chi(G; x)$ by evaluating it at $-1$, \ie $\Psi(D; 1) = \chi(G; -1)$ for any orientation $D$ of $G$ (see \Cref{rmk:ChromaticPolynomial} for a discussion of this relation).
Since then, the study of the chromatic polynomial itself has grown into a branch of Combinatorics, and the focus on the evaluation of this polynomial at $-1$ together with the combinatorial reciprocity theorems it suggests \cite{GreeneZaslavsky1983-InterpretationOfWhitneyNumbers,GebhardSagan2000-SinksInAcyclicOrientations,Gessel2001-AcyclicOrientations,Lass2001-AcyclicOrientationsAndChromaticPolynomial,BernardiNadeau2020-CombinatorialReciprocity} is of particular relevance for our work.

As the chromatic polynomial can be computed exploiting a deletion-contraction phenomenon, it is natural to ask whether the acyclic polynomial also admits one.
The main objective of \Cref{ssec:AcyclicPolynomial} is to show that the acyclic polynomials partially admit a deletion-contraction phenomenon (\Cref{prop:Psi0Recursion,prop:Psi+Recursion}) which extends to the acyclic signatures (\Cref{thm:AcyclicSignatureContractionDeletionPhenomenon}), allowing for an algorithm that computes the acyclic signatures of some graphs (\Cref{rmk:ContractionDeletionAlgorithmForSignature}).
Furthermore, we prove in \Cref{rmk:NotATutteWhitneyInstance} that the acyclic polynomial is not an evaluation of the famous Tutte--Whitney polynomial (see  \cite{Whitney1935,Tutte1954,Crapo1969-TuttePolynomial} for the original articles, or \cite{Oxley2011-MatroidTheory} for a comprehensive presentation in the broader context of matroids).
Yet, it seems that there is more to this story, since recent results show that properties of acyclic orientations, besides counting, are encoded in the Tutte--Whitney polynomial, see \eg the work of Backman \cite{Backman2018-PartialGraphOrientationsAndTuttePolynomial}.

\medskip
The question of acyclic orientations of graphs also recently arose in the field of hyperplane arrangements \cite{BrennerCardinalMcConvilleMerinoMutze25-CombinatorialGenerationVII,KorberSchniedersStrickerWalizadeth25-HamiltonianCyclesSupersolvable}, where the acyclic orientations of a graph are in one-to-one correspondence with the chambers of the corresponding graphic arrangement, that is the arrangement of hyperplanes $\mathcal{A}_G := \{\b x\in \R^\V ~;~ x_i = x_j\}$ for $\{i, j\}$ an edge of $G$ with node set $\V$.

The graph of acyclic orientations $\AO$ is known in this context as the tope graph of the graphic arrangement, describing the neighboring structure of the chambers.
Equivalently $\AO$ is also the vertex-edge graph of the zonotope associated to the arrangement, known under the names \emph{acyclotopes} \cite{zaslavskygraphcloring} and \emph{graphical zonotopes} \cite{Postnikov2009,PadrolPilaudPoullot2025DeformedGraphicalZonotopes,PadrolPoullot2025IndecomposabilityAndBeyond}.
The interested reader may consult \cite[Section~2]{Stanley2007-IntroHyperplaneArrangements} or \cite[Section~1.1]{OrientedMatroids-RedBook} for a detailed presentation.
The interplay of properties of the tope graph, the arrangement and the zonotope is an active field of research \cite{koizumi2026magnitudehomologytopegraphs, bach2026acyclotopestocyclotopes}.  
(Note that our question need not be restricted to the graphic case in this setting.)

\begin{formulation}
Classify all (graphic) arrangements whose tope graph admits a Hamiltonian cycle.
Equivalently, classify all graphs such that the vertex-edge graph of the associated graphical zonotope (\aka acyclotope) is Hamiltonian.
\end{formulation}

For example, it is known that the tope graph of so-called reflection arrangements, of which $\mathcal{A}_{K_\ell}$ is an example, admits a Hamiltonian cycle \cite{conwaysloanewilks}. 
Recently, there has been another development for this more general setting as two groups of researchers independently found that the tope graph of so-called \emph{supersolvable arrangements} are Hamiltonian \cite{BrennerCardinalMcConvilleMerinoMutze25-CombinatorialGenerationVII, KorberSchniedersStrickerWalizadeth25-HamiltonianCyclesSupersolvable}. 
Applying this result to graphic arrangements recovers the result for chordal graphs from Savage, Squire and West \cite[Section~4.1]{SavageSquireWest93-GrayCodesAcyclicOrientations}. 
Furthermore, some of the observations we make in this work about the evaluations of the acyclic polynomial for graphs are already known in this more general hyperplane setting, employing different names and arguments (see \Cref{rem:hyperplane_setting}). 

In the second main part of our work, \Cref{sec:MultiPathGluing}, we prove that some operations can be performed on a graph whose graph of acyclic orientations is Hamiltonian, preserving this property.
% , such as adding a simplicial (\Cref{ssec:AddingSimpliceNode}) node which is more or less a direct implication of work by previous authors. 
We introduce two methods of \emph{pattern lacing} (\Cref{cor:SelfLacingIntoAnAOcycle,thm:ConcatenatingReplacingPatternsWithSharedEndPoints}) which can be seen as generalizing the zig-zag method described in \cite{BrennerCardinalMcConvilleMerinoMutze25-CombinatorialGenerationVII}.
This way, we prove that, under a certain parity condition, adding a path of even length (or two paths of odd lengths) between nodes $u$ and $v$ preserves having a Hamiltonian graph of acyclic orientations.
A graph consisting of several paths between two nodes $u$ and $v$ is called a \emph{multipath}.
Using these path gluings together with our results on the acyclic signature, we establish (``$G$ is {\good}'' means ``$\AO$ is Hamiltonian''):

\begin{thmUniv}[{\Cref{cor:MultiPathIsGood,exm:NewAOhamiltonianGraphs}}]\label{thmA}
We have the following positive results:
\begin{enumerate}[(i)]
\item A multipath $G$ is {\good} if and only if the number of edges of $G$ is odd.
\item Any $2$-sum of two cycle graphs $C_n$ and $C_m$ is {\good} if and only if $n + m$ is even.
\item For a tree $T = ([s], E)$ and numbers $n_1, \dots, n_s \geq 3$, a $(T,\, n_1, \dots, n_s)$-tree of cycles is any graph isomorphic to a $2$-sum $H \oplus_2 C_{n_s}$ where $H$ is a $(T\ssm s,\, n_1, \dots, n_{s-1})$-tree of cycles.
If $n_p$ is odd for all $p\in [s]$, then any $(T,\, n_1, \dots, n_s)$-tree of cycles is {\good}, for any tree $T$.
\item If $uv\in E$ is an edge of a {\good} graph $G = (\V, E)$, and $C_n$ is a cycle graph, then the $2$-sum $G\oplus_2 C_n$ is {\good} if and only if $n$ is odd.
\end{enumerate}
\end{thmUniv}

This theorem subsumes the one of Savage, Squire and West: ``For a cycle  $C_n$ on $n$ nodes, $\AO[C_n]$ is Hamiltonian if and only if $n$ is odd'' \cite[below Lemma~3.1 \& Theorem~4.2]{SavageSquireWest93-GrayCodesAcyclicOrientations}.
The above \Cref{thmA} is a specialization of the more general \Cref{cor:MultiPathGluingIsGood} made easier for the introduction.

The motivation behind considering this type of operation arises from the fact that if $\AO[H]$ is Hamiltonian for all the $2$-connected components of $G$, then $\AO$ is also Hamiltonian.
Moreover, it is well-known (see \cite[Section~3.1]{diestel} for details and figures) that $2$-connected graphs admit an (open) ear decomposition, \ie we can construct any $2$-connected graph by starting with a cycle and subsequently adding paths between pairs of nodes.
Therefore, exploring all the ways of adding paths to a graph that are compatible with the $\c{AO}$-Hamiltonicity of said graph means iteratively constructing a large example class with hopes of reaching a complete classification.
We show (recall that a node is \emph{simplicial} if all its neighbors pairwise share an edge, \ie form a clique):

\begin{thmUniv}[Simplified version \Cref{cor:AddingSimplicialNodesAndGluingMultiPath}]\label{thmB}
Let $G$ be a graph such that $\AO$ is Hamiltonian.
If $G'$ is obtained from $G$ by either
\begin{enumerate}[(i)]
\item the addition of a simplicial node $x_\circ$ on a clique of odd cardinality in $G$ and the (optional) gluing of a multipath of even cardinality at some nodes $u, v$ with $u\ne x_\circ$ and $v\ne x_\circ$; or 
\item the (optional) addition of a simplicial node and the gluing of a multipath of even cardinality at some nodes $u, v$ that share an edge in $G$;
\end{enumerate}

Then $\AO[G']$ is Hamiltonian.
\end{thmUniv}

% \begin{thmUniv}[Simplified version \Cref{cor:AddingSimplicialNodesAndGluingMultiPath}]\label{thmB}
% Let $G$ be a graph such that $\AO$ is Hamiltonian, and $G_0, G_1, \dots , G_k$ with $G_0 = G$ be a sequence of graphs.
% If $G_{i+1}$ is obtained from $G_{i}$ by either
% \begin{enumerate}[(i)]
% \item the addition of a simplicial node $x_\circ$ on a clique of odd cardinality and the gluing of a multipath of even cardinality at some nodes $u, v$ with $u\ne x_\circ$ and $v\ne x_\circ$; or 
% \item the addition of a simplicial node (or of none) and the gluing of a multipath of even cardinality at some nodes $u, v$ that share an edge.
% \end{enumerate}

% Then all $\AO[G_i]$ are Hamiltonian.
% \end{thmUniv}

A more flexible yet notation-heavy version of Theorem \ref{thmB} is proven in \Cref{cor:AddingSimplicialNodesAndGluingMultiPath}.
We also detail a third way of gluing several multipaths on different pairs of nodes in \Cref{ssec:StronglyInternallyEven}.

\medskip 
This article is organized as follows:
In \Cref{sec:Prelim}, we summarize all the graph-theoretic preliminaries that are needed for the present work, especially on orientations, acyclic orientations, but also the usual examples of the Gray code for binary sequences and the Steinhaus--Johnson--Trotter algorithm for permutations.
In \Cref{sec:AcyclicPolynomial} we introduce the acyclic polynomial and explain its fundamental properties.
Notably, we decompose it into partial acyclic polynomials (\Cref{prop:AcyclicPolyDecomposition}), and prove a deletion-contraction formula in certain cases (\Cref{prop:Psi0Recursion,prop:Psi+Recursion}).
This allows us to better understand the acyclic signature of graphs (\Cref{ssec:AcyclicSignature}) and to compute the acyclic signature of multipaths (\Cref{thm:AcyclicSignatureMultiPaths}), exemplifying how to certify that $\AO$ is not Hamiltonian for some graph $G$.
On the constructive side of Hamiltonicity considerations, in \Cref{sec:MultiPathGluing}, we give explicit constructions of Hamiltonian cycles for graphs obtained via graph operations such as path gluings or addition of a simplicial node.

Generalizations of strategies of pattern lacing can be found in Appendix~\ref{appendix:GeneralPatternLacing}, while Appendix~\ref{app_bipartite} is devoted to complete bipartite graphs, and Appendix~\ref{app:ParityArgumentMultiPathDirectCount} proposes an alternative proof that multipaths of even total length are not {\good} using a counting argument.
In Appendix~\ref{app:GraphExamples}, we give $22$ explicit examples of well-known graphs and the properties of their graph of acyclic orientations, for the reader to enjoy.

% \clearpage
\section{Preliminaries}\label{sec:Prelim}

We denote $\mathdefn{[n]} \eqdef \{1, \dots, n\}$, and $[i, j] = \{i, i+1, \dots, j\}$.
Moreover, let $\mathdefn{2^X} \eqdef \{Y ~;~ Y\subseteq X\}$ and $\mathdefn{\binom{X}{k}} \eqdef \{Y\in 2^X ~;~ |Y| = k\}$, where $|Y|$ is the size (\ie cardinality) of the set $Y$.
For two sets $X, Y$ we denote $\mathdefn{X\sqcup Y}$ for $X\cup Y$ when $X$ and $Y$ are disjoint sets.
The \defn{Cartesian product} of $X$ and $Y$ is $\mathdefn{X\times Y} \eqdef \{(x, y) ~;~ x\in X, y\in Y\}$, and the \defn{symmetric difference} of $X$ and $Y$ is $\mathdefn{X\symmdiff Y} \eqdef (X\cup Y) \ssm (X\cap Y) = (X\ssm Y)\sqcup(Y\ssm X)$.

An (undirected) \defn{graph} $\mathdefn{G} \eqdef (\V, E)$ is composed of a set of \defn{\nodes} $\V$, and a set of \defn{edges} $E \subseteq \binom{[n]}{2}$.
A \defn{digraph} (or simple directed graph) $\mathdefn{D} \eqdef (\V, \Arcs)$ is composed of a set of \nodes ~$\V$ and a set of \defn{\arcs} (or directed edges) $\Arcs = \bigl\{ (i, j) ~;~, i, j \in \V\bigr\}$, such that $\{i, j\} \ne \{k, \ell\}$ for $(i, j), (k, \ell) \in \Arcs$.
To shorten notation we write $ij$ for $\{i, j\}$ and $(i,j)$ when the context is clear. We denote by $N(G), E(G), \Arcs(G)$ the sets of nodes, edges and arcs of a given (di)graph $G$.
For a given node $i\in \V$, the nodes $j\in \V$ such that $ij\in E$ are called the \defn{neighbors} of $i$.
 
% For a \node ~$i\in \V$, its set of \defn{neighbors} (or adjacent \nodes) in a graph $G = (\V, E)$ is $\mathdefn{N(i)} \eqdef \{j\in \V ~;~ ij\in E\}$; in a digraph $D = (\V, \Arcs)$, its set of \defn{out-neighbors} is $\mathdefn{N_{\Out}(i)}\eqdef \{j\in \V ~;~ ij\in \Arcs\}$ while its set of \defn{in-neighbors} is $\mathdefn{N_{\In}(i)}\eqdef \{j\in \V ~;~ ji\in \Arcs\}$.
% \germain{Here commented: defn of (in/out-)neighbors}
For a (di)graph $G = (\V, E)$, and some $U\subseteq \V$, the associated \defn{induced sub(di)graph} is defined as $\mathdefn{G[U]} \eqdef \bigl(U, \{ij \in E ~;~ i, j\in U\}\bigr)$.

For a digraph $D = (\V, \Arcs)$, the associated \defn{underlying graph} is the graph $\mathdefn{\underline{D}} \eqdef (\V, E)$ where $E = \bigl\{\{i, j\} ~;~ (i, j) \in \Arcs\bigr\}$; conversely an \defn{orientation} of $G$ is a digraph $D$ such that $\underline{D} = G$.
We denote the set of orientations of $G$ by $\mathdefn{\c O(G)}$.
Similarly, for a digraph $D$, we call orientations of $D$ the orientations of $\underline{D}$, and denote $\c O(D)$ for $\c O(\underline{D})$.

A (di)graph $G = (\V, E)$ is \defn{bipartite} if there exists $\V_1, \V_2\subseteq \V$ such that $\V_1 \sqcup \V_2 = \V$ and there is no $ij\in E$ with $i, j\in \V_1$ nor with $i,j\in \V_2$.

\begin{example}
The following examples are well-studied in the literature.
\begin{enumerate}[(i)]
\item For $n\geq 1$, the \defn{$n^{\text{th}}$ cube-graph} $\mathdefn{\cube}$ is the graph on the set of \nodes~ $2^{[n]}$ where $X, Y \subseteq [n]$ are linked by an edge if and only if $|X\symmdiff Y| = 1$.
Moreover, the \defn{$n^{\text{th}}$~Boolean lattice} $\mathdefn{\Bool}$ is the orientation of $\cube$ with \arcs~$(X, Y)$ for $X \subseteq Y$ and  $|Y\ssm X| = 1$.
Cube-graphs and Boolean lattices are bipartite: $X$ and $Y$ belong to the same part if and only if $|X|$ and $|Y|$ have the same parity.
See \Cref{fig:AO_graphs_examples} (top left) for an illustration.

\item For $n\geq 1$, the \defn{$n^{\text{th}}$ complete graph} $\mathdefn{K_n}$ is the graph on the set of \nodes~$[n]$ where all $i$ and $j$ are linked by an edge for $i, j\in [n]$.
Complete graphs are not bipartite.
A subset of \nodes~$U\subseteq\V$ is a \defn{clique} in a graph $G = (\V, E)$ if its induced subgraph $G[U]$ is isomorphic to a complete graph; and a \node~ $i$ is \defn{simplicial} if its set of neighbors is a clique.

\item For $n\geq 1$, the \defn{$n^{\text{th}}$ permutahedral graph} $\mathdefn{\Pi_n}$ whose nodes are the permutations of $[n]$, where $\sigma$ and $\sigma'$ are linked by an edge if there exists $i\in [n-1]$ such that $\sigma' = \sigma\circ (i~\,~ i+1)$.
Moreover, the \defn{$n^{\text{th}}$ weak (Bruhat) order} $\mathdefn{\WeakOrder}$ is the orientation of $\Pi_n$ with arcs $(\sigma, \sigma')$ if $\sigma' = \sigma\circ (i~\,~ i+1)$ and $\sigma(i) < \sigma(i+1)$.
Permutahedral graphs and weak orders are bipartite.
\end{enumerate}
\end{example} 

\begin{figure}
    \centering
    \includegraphics[width=0.475\linewidth]{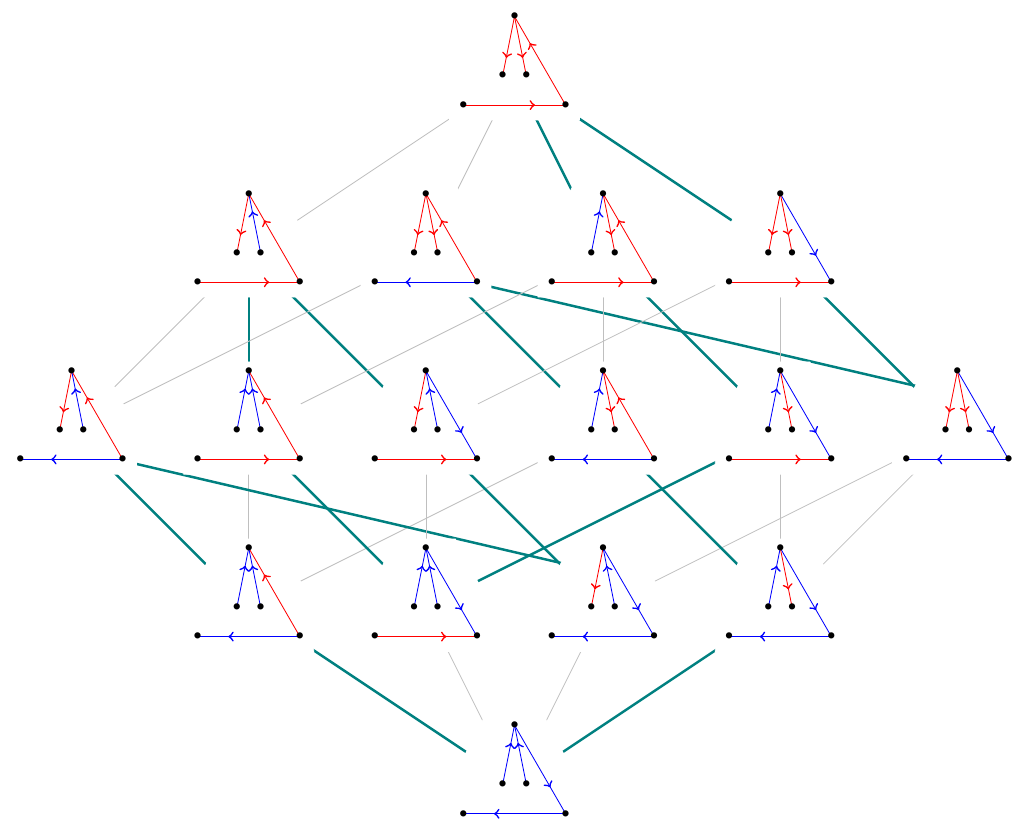}\hfill
    \includegraphics[width=0.475\linewidth]{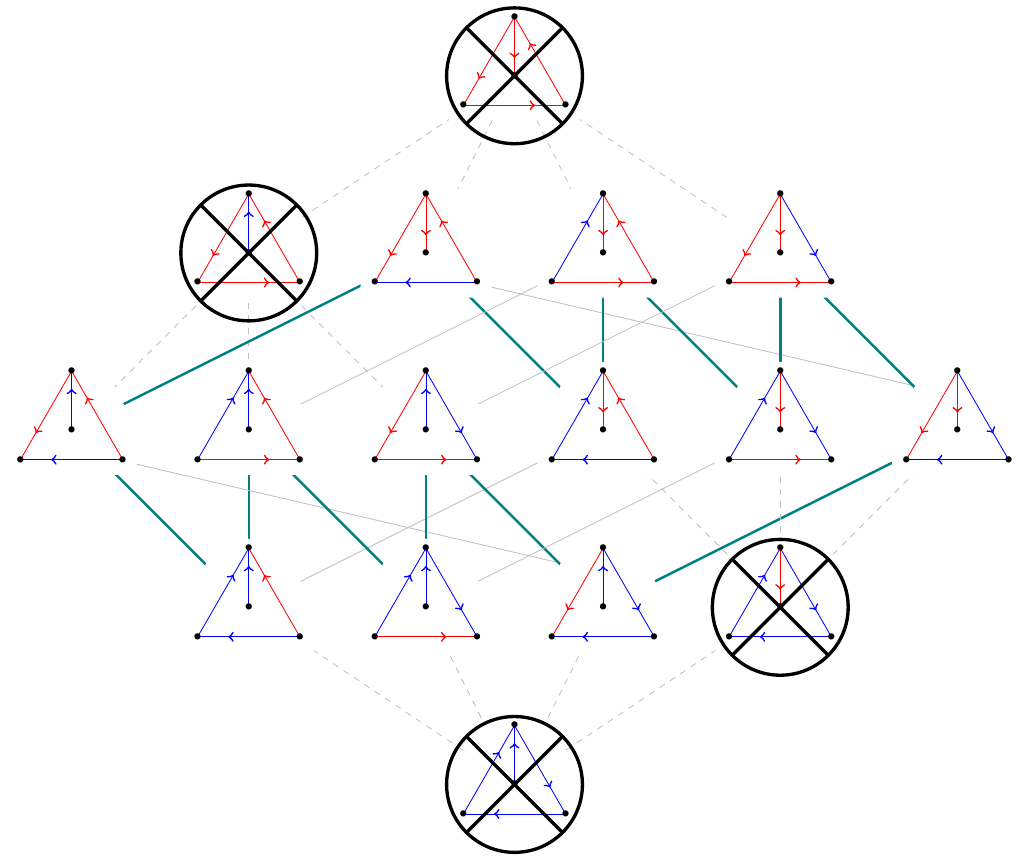}
    
    \includegraphics[width=0.475\linewidth]{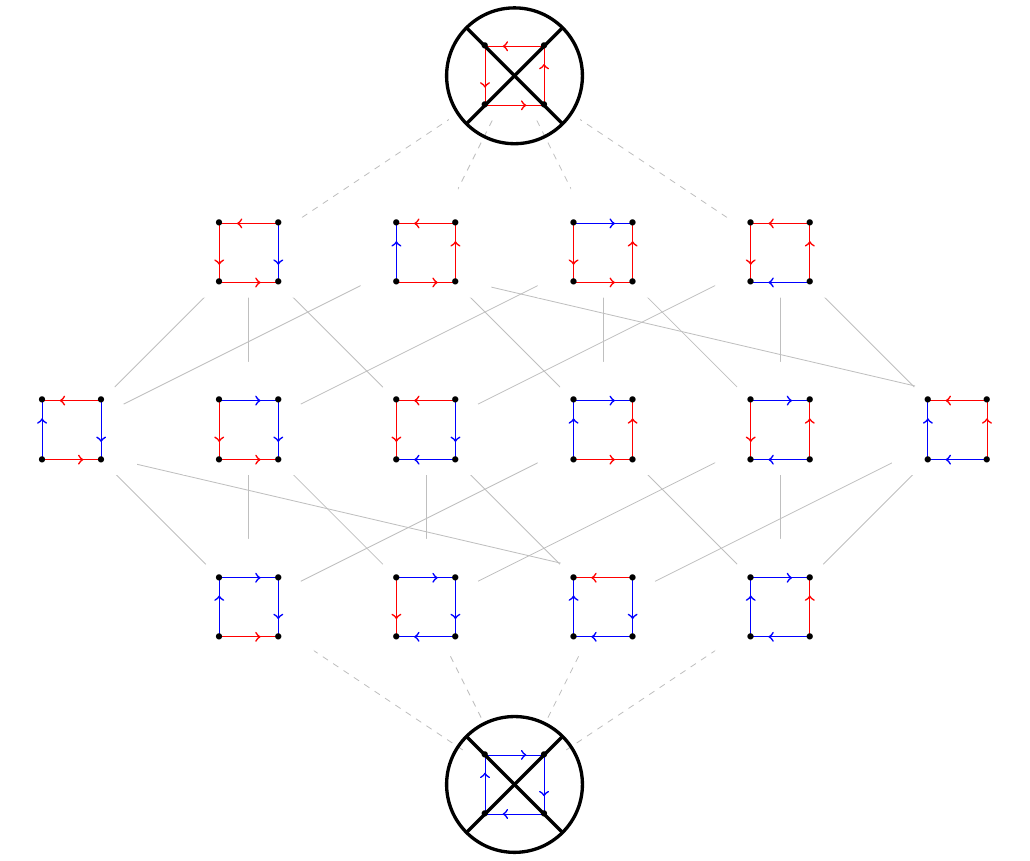}%
    \caption{
    The graphs of acyclic orientations of three graphs on $4$ arcs, drawn as a Boolean lattice $\c Bool_4$ with some elements removed.
    Each is a ranked poset where the chosen based orientation is in \textcolor{blue}{blue}.
    (Top left) For $G$ a tree, $\AO$ is a cube which is Hamiltonian as indicated by the \textcolor{teal}{teal cycle}.
    (Top right) For $G$ the $1$-sum of an arc $K_2$ and a triangle $C_3$, the graph $\AO$ is the Cartesian product $\AO[K_2]\times \AO[C_3] = K_2\times C_6$ which is Hamiltonian as indicated by the \textcolor{teal}{teal cycle}.
    (Bottom) For $G$ an even cycle, $\AO$ is not Hamiltonian because there are $8$ elements of odd ranks but $6$ elements of even ranks.
    }
    \label{fig:AO_graphs_examples}
\end{figure}

\subsection{On general orientations}\label{ssec:Prelim:GeneralOrientations}
% For a \arc $ij$, its \defn{re-orientation} is the \arc $ji$.
For any fixed orientation $D = (\V, \Arcs)$ of $G = (\V, E)$, the set of orientations of $G$ is in natural bijection with~$2^E$:
Namely, for $\omega \subseteq E$, let $\mathdefn{\Arcs^\omega} \eqdef \left\{\begin{array}{cl}
    ij & \text{if } ij\in \omega \\
    ji & \text{if } ij\notin \omega \\ 
\end{array} ~;~ ij\in \Arcs\right\}$. 

The \defn{re-orientation of $D$ with respect to $\omega$} is $\mathdefn{D^\omega} \eqdef (\V, \Arcs^\omega)$, and it is immediate that $\omega \mapsto D^\omega$ defines a bijection $2^{E} \to \c O(G)$. We define the \defn{(re-orientation) rank} $\mathdefn{\rank_D(D^\omega)} \eqdef |\omega|$, see \Cref{fig:AO_graphs_examples} for orientations drawn in ranked posets. 
With respect to $D$, for two orientations $D_1$ and $D_2$ of $G$, we can find $\omega_1,\omega_2 \subseteq E$ such that $D_1 = D^{\omega_1}$ and $D_2 = D^{\omega_2}$ and we write $\mathdefn{D_1 \preceq_D D_2}$ if $\omega_1 \subseteq \omega_2$.

When we choose a fixed orientation $D$ of $G$ in order to express all the other orientations of $G$ as a re-orientation of $D$, we call $D$ a \defn{base orientation} and denote the set of orientations of some rank $k\geq 0$ by
$\mathdefn{\c O_{k}(D)} \eqdef \{D'\in \c O(G) ~;~ \rank_D(D') = k\}$.
Using the bijection $\omega \mapsto D^\omega$, we get $|\c O_{k}(D)| = \left|\binom{E}{k}\right| = \binom{|E|}{k}$, for any base orientation $D$ and any~$k$.

\begin{remark}\label{rmk:SumEvenOddBinomNewton}
Using Newton's binomial formula, one has, for $m \geq 1$:
$$\left(\sum_{k \text{ even}} \binom{m}{k}\right) - \left(\sum_{k \text{ odd}} \binom{m}{k}\right) = \sum_{k=0}^m \binom{m}{k} (-1)^k = (1-1)^m = 0$$
Hence, $\sum_{k \text{ even}} \binom{m}{k} = \sum_{k \text{ odd}} \binom{m}{k} = 2^{m-1}$ for any $m\geq 1$.
We emphasize that this implies that
\begin{equation*} \label{equ:odd_even_orientations} 
\sum_{k \text{ even}} |\c O_{k}(D)| = \sum_{k \text{ odd}} |\c O_{k}(D)| = 2^{|E| - 1}, 
\end{equation*} 
for any base orientation $D$ of $G$.
This will prove useful in the following sections, \eg \Cref{ssec:AcyclicSignature}.
\end{remark}

Besides, note that for any $D_1 = D^{\omega_1}$ and $D_2 = D^{\omega_2}$, we can define the \defn{difference} between~$D_1$ and~$D_2$ as $\mathdefn{D_1 \symmdiff D_2} \eqdef \omega_1\symmdiff \omega_2$ which does not depend on the base orientation $D$:
these are the edges of $\underline{D}$ which are not oriented the same way in $D_1$ as in $D_2$.
Two orientations $D_1, D_2$ of $G$ are \defn{adjacent} if $|D_1\symmdiff D_2| = 1$.
The \defn{graph of orientations} (respectively \defn{digraph of orientations} with respect to some base orientation $D$) of $G$ is the graph whose \nodes~ are the orientations of $G$, and where two orientations are linked by an edge (resp. \arc) if they are adjacent (resp. adjacent with $D_1 \subseteq_D D_2$), see \Cref{fig:AO_graphs_examples} (top left).
We denote $\mathdefn{\c {O}(G)}$ (resp. $\mathdefn{\c O(D)}$) this (di)graph, where we slightly abuse notation and reuse the notation used for the set of orientations.
Fixing any base orientation $D\in \c O(G)$, the bijection $\omega \mapsto D^\omega$ defines an isomorphism of graphs from the graph of orientations $\c O(G)$ to the cube-graph $Q_{|E|}$, and from the digraph of orientations $\c O(D)$ to the Boolean lattice $\Bool[|E|]$.
Note that each  base orientation $D\in \c O(G)$ defines a different isomorphism $\c O(G) \to Q_{|E|}$, and a different orientation $\c O(D)$ of $\c O(G)$.

\subsection{Cycles, Hamiltonian graphs}\label{ssec:Prelim:CyclesAndHamiltonianGraphs}
A \defn{(directed) path} in a (di)graph $G$ is a sequence of different \nodes~ $i_1, \dots, i_r\in \V(G)$ such that $i_pi_{p+1}$ is a (directed) edge in $G$ for all $p\in [r-1]$.
A \defn{(directed) cycle} is a (directed) path $i_1, \dots, i_r\in \V$ such that $i_ri_1$ is a (directed) edge in $G$.
A graph is \defn{Hamiltonian} (resp. \defn{path-Hamiltonian}) if it admits a cycle (resp. a path) containing all its \nodes.
Clearly, a Hamiltonian graph is path-Hamiltonian.

We state here a well-known proposition concerning Hamiltonicity of bipartite graphs and give a short proof for the sake of self-containment.

\begin{proposition}\label{prop:Bipartite(Path)Hamilitonian}
Let $G = (\V, E)$ be a graph admitting a bipartition $\V = \V_1\sqcup \V_2$.
\begin{enumerate}[(i)]
\item If $G$ is Hamiltonian, then $|\V_1| = |\V_2|$; \label{item:BipartiteHamiltonian}
\item If $G$ is path-Hamiltonian with path $i_1, \dots, i_r$ where $i_1\in \V_1$, then either $|\V_1| = |\V_2|$ and~$i_r\in \V_2$, or $|\V_1| = |\V_2| + 1$ and $i_r\in \V_1$. \label{item:BipartitePathHamiltonian}
\end{enumerate}
\end{proposition}

\begin{proof}
\ref{item:BipartiteHamiltonian}
If $G$ is Hamiltonian, let $i_1, \dots, i_r$ be a cycle with $\V = \{i_1, \dots, i_r\}$.
As $i_pi_{p+1}\in E$ for all $p\in [r-1]$, we get that $i_p\in \V_1$ if and only if $i_{p+1}\in \V_2$ and conversely, by definition of the bipartition.
So $i_p \mapsto i_{p+1}$ for $p\in [r-1]$ and $i_r\mapsto i_1$ defines an involution of $\V$ mapping $\V_1$ onto~$\V_2$
Thus, $|\V_1| = |\V_2|$.

\ref{item:BipartitePathHamiltonian}
If $i_r\in \V_1$, then $i_p\mapsto i_{p+1}$ for $p\in [r-1]$ defines an involution of $\V\ssm\{i_r\}$ mapping $\V_1\ssm\{i_r\}$ onto $\V_2$, hence $|\V_1| = |\V_2| + 1$.
Else, $i_p\mapsto i_{p+1}$ for $p\in [r-1]$ and $i_r\mapsto i_1$ defines an involution of~$\V$ mapping $\V_1$ onto $\V_2$, hence $|\V_1| = |\V_2| + 1$. 
\end{proof}

\begin{example}\label{exm:CubeGraphIsGood}
The cube-graph is bipartite and Hamiltonian.
As it is a fundamental example supporting the constructions of \Cref{sec:MultiPathGluing}, we detail the properties of Hamiltonian cycles of the cube-graphs.

There are two ways to encode the nodes of $Q_n$, either as subsets $X\subseteq[n]$, or as strings of bits $\b\varepsilon = (\varepsilon_1, \dots, \varepsilon_n)\in \{0, 1\}^n$.
The bijection is given by $X(\b \varepsilon) = \{i\in [n] ~;~ \varepsilon_i = 1\}$, and reciprocally, $\varepsilon(X)_i = 1$ if and only if $i\in X$.

The bipartition of $Q_n$ is given by the parts $\{X\subseteq[n]~;~|X| \text{ is even}\} \simeq \{\b \varepsilon\in \{0, 1\}^n ~;~ \sum_i\varepsilon_i \text{ is even}\}$ and $\{X\subseteq[n]~;~|X| \text{ is odd}\} \simeq \{\b \varepsilon\in \{0, 1\}^n ~;~ \sum_i\varepsilon_i \text{ is odd}\}$.
In particular, the nodes $\emptyset \simeq (0, \dots, 0)$ and $[n] \simeq (1, \dots, 1)$ belong to the same part if and only if $n$ is even.

There are many ways to construct Hamiltonian cycles in $Q_n$, see \OEIS{A066037} for a count of them for small $n$.
First, note that $Q_2$ is a cycle on $4$ elements, so it trivially admits a Hamiltonian cycle.
Then, by induction, given two Hamiltonian cycles $\b C = (X_1, \dots, X_{2^n})$ and $\b C' = (Y_1, \dots, Y_{2^n})$ of $Q_n$ with a shared arc $(X_{2^n}, X_1) = (Y_{2^n}, Y_1)$ (\eg $\b C = \b C'$), one can construct a Hamiltonian cycle of $Q_{n+1}$ as:
$$\bigl(X_1,\, X_2,\, \dots,\, X_{2^n},\,\, Y_{2^n}\cup\{n+1\},\, Y_{2^n-1}\cup\{n+1\},\, \dots,\, Y_1\cup\{n+1\}\bigr) \,.$$

In term of sequences of bits, one constructs a highly symmetric Hamiltonian cycle $\bigl(\varepsilon^{(1)}, \dots, \varepsilon^{(2^n)}\bigr)$ as follows.
For each $i\in [n]$, the sequence $\bigl(\varepsilon_i^{(1)}, \varepsilon_i^{(2)}, \dots, \varepsilon_i^{(2^n)}\bigr)$ starts with a sequence of $2^{n-i}$ copies of $0$, then alternates between $2^{n-i+1}$ copies of $1$ and $2^{n-i+1}$ copies of $0$, until it ends with again $2^{n-i}$ copies of $0$.
Such a construction is illustrated by the columns of \Cref{tab:CubeGraph4}.

% The nodes of $Q_n$ are encoded by bitstrings of length $n$, thus there are $2^n$ strings to consider. Enumerate the bits from right to left with $0, \cdots, n-1$. For $i \in 0, \cdots, n-1$ consider the sequence $L_i = (0^{2^i}, 1^{2^i})$ and its reverse $L_i^{rev}$.\germain{Now, we have defined the cube-graph in order to not have to dive into such technical details. $L_i$ is just $[i]$ and $L_i^{rev}$ is $[n]\ssm[i+1]$ or $[i+1, n]$, for instance. Can I rewrite ?}\leonie{You can rewrite but keep the table please, I think it is quite instructive. You can also move it to a different place} For each $i$ alternate the bit sequence in the $i$-column between $L_i$ and $L_i^{rev}$. By construction, two neighboring bitsequences will only differ in one entry. See \Cref{tab:CubeGraph4} for an example for $n = 4$.

\begin{table}[th!] %h!
    \centering 
    \begin{tabular}{ccc|ccc}
        
        $k$ & String $\b \varepsilon^{(k)}$ & Subset $X_k$ & $k$ & String $\b \varepsilon^{(k)}$ & Subset $X_k$ \\ 
        \midrule 
        1 & 0000 & $\emptyset$ & 9 & 1100 & $\{1, 2\}$ \\ 
        2 & 0001 & $\{4\}$ & 10 & 1101 & $\{1, 2, 4\}$ \\ 
        3 & 0011 & $\{3, 4\}$ & 11 & 1111 & $\{1, 2, 3, 4\}$ \\ 
        4 & 0010 & $\{3\}$ & 12 & 1110  & $\{1, 2, 3\}$ \\
        
        5 & 0110& $\{2, 3\}$ & 13 & 1010 & $\{1, 3\}$ \\ 
        6 & 0111 & $\{2, 3, 4\}$ & 14 & 1011 & $\{1, 3, 4\}$ \\ 
        7 & 0101 & $\{2, 4\}$ & 15 & 1001 & $\{1, 4\}$ \\ 
        8 & 0100 & $\{2\}$ & 16 & 1000 & $\{1\}$ \\ 
        \bottomrule  
    \end{tabular} 
    \caption{
    A Gray code for $4$-bit binary strings, \aka a Hamiltonian cycle of $Q_4$.
    The $i^{\text{th}}$ column of the strings $\varepsilon^{(k)}$ alternates between $2^{5-i}$ copies of $1$ and $2^{5-i}$ copies of $0$ with a starting offset of $2^{4-i}$ copies of $0$.
    The subset of $[4]$ associated to each string of bits is indicated in the third column.
    }
    % \caption{Gray code for four-bit binary strings, encoding the Hamiltonian cycle on the cube-graph $Q_4$. The horizontal lines separate the sequences $L_i$ and $L_i^{rev}$.} 
    \label{tab:CubeGraph4} 
\end{table}

\end{example}

\begin{example}\label{exm:PermutahedralGraphIsGood}
The permutahedral graph $\Pi_n$ is bipartite and Hamiltonian.
On the one hand, since two adjacent permutations in $\Pi_n$ differ by a transposition, the usual signature of permutation provides a bipartition of $\Pi_n$:
no two even (resp. odd) permutations share an edge in $\Pi_n$.

We shortly present the famous Steinhaus--Johnson--Trotter algorithm \cite{Johnson63-GenerationOfPermutations} which yields a Hamiltonian cycle of the permutahedral graph.
See the survey of Savage \cite[Section 3]{Savege97-SurveyGrayCodes} for a more detailed presentation of the subject.

First, for a permutation $\tau$ of $n$ elements and $i\in [n+1]$, we define the permutation $\tau^{(i)}$ of $n+1$ elements by
$$\tau^{(i)}(j) = \left\{\begin{array}{ll}
\tau(j) & \text{if } j < i \\
n+1 & \text{if } j = i \\
\tau(j-1) & \text{if } j > i
\end{array}\right. \,.$$

The graph $\Pi_3$ trivially admits a Hamiltonian cycle since it is a cycle-graph on $6$ elements, \eg $123-132-312-321-231-213$.
By induction, let $(\sigma_1, \dots, \sigma_{n!})$ be a Hamiltonian cycle of $\Pi_n$,
then the following is a Hamiltonian cycle of $\Pi_{n+1}$, see \Cref{tab:PermutahedralGraph4}:
$$\bigl(\sigma_1^{(n+1)}, \dots, \sigma_1^{(1)},\, \sigma_2^{(1)}, \dots, \sigma_2^{(n+1)},\, \sigma_3^{(n+1)}, \dots, \sigma_3^{(1)},\, \dots,\, \sigma_{n!}^{(1)}, \dots, \sigma_{n!}^{(n+1)}\bigr) \,.$$

\begin{table}[th!] %h!
    \centering 
    \begin{tabular}{ccc|ccc|ccc|ccc}
        
        $i$ & $j$ & $\sigma_i^{(j)}$ & $i$ & $j$ & $\sigma_i^{(j)}$ & $i$ & $j$ & $\sigma_i^{(j)}$ & $i$ & $j$ & $\sigma_i^{(j)}$\\ 
        \midrule 
        1 & 4 & $123\textcolor{red}{\textbf{4}}$ & 2 & 3 & $13\textcolor{red}{\textbf{4}}2$ & 4 & 1 & $\textcolor{red}{\textbf{4}}321$ & 5 & 2 & $2\textcolor{red}{\textbf{4}}31$ \\ 
        1 & 3 & $12\textcolor{red}{\textbf{4}}3$ & 2 & 4 & $132\textcolor{red}{\textbf{4}}$ & 4 & 2 & $3\textcolor{red}{\textbf{4}}21$ & 5 & 1 & $\textcolor{red}{\textbf{4}}231$ \\ 
        1 & 2 & $1\textcolor{red}{\textbf{4}}23$ & 3 & 4 & $312\textcolor{red}{\textbf{4}}$ & 4 & 3 & $32\textcolor{red}{\textbf{4}}1$ & 6 & 1 & $\textcolor{red}{\textbf{4}}213$\\ 
        1 & 1 & $\textcolor{red}{\textbf{4}}123$ & 3 & 3 & $31\textcolor{red}{\textbf{4}}2$ & 4 & 4 & $321\textcolor{red}{\textbf{4}}$ & 6 & 2 & $2\textcolor{red}{\textbf{4}}13$\\
        
        2 & 1 & $\textcolor{red}{\textbf{4}}132$ & 3 & 2 & $3\textcolor{red}{\textbf{4}}12$ & 5 & 4 & $231\textcolor{red}{\textbf{4}}$ & 6 & 3 & $21\textcolor{red}{\textbf{4}}3$\\ 
        2 & 2 & $1\textcolor{red}{\textbf{4}}32$ & 3 & 1 & $\textcolor{red}{\textbf{4}}312$ & 5 & 3 & $23\textcolor{red}{\textbf{4}}1$ & 6 & 4 & $213\textcolor{red}{\textbf{4}}$\\ 
        \bottomrule  
    \end{tabular} 
    \caption{
    The Hamiltonian cycle of $\Pi_4$ obtained using the Steinhaus--Johnson--Trotter algorithm.
    Each permutation differs from the previous one by a transposition of two consecutive values.
    }
    \label{tab:PermutahedralGraph4} 
\end{table}

\end{example}

\subsection{Graph of acyclic orientations}\label{ssec:Prelim:AOgraph}

A directed graph is \defn{acyclic} if it contains no directed cycles, otherwise it is \defn{cyclic}.
For a (di)graph~$G$, its \defn{graph of acyclic orientations} $\mathdefn{\AO}$ is the subgraph of its graph of orientations $\c O(G)$ induced on the set of acyclic orientations of $G$, see \Cref{fig:AO_graphs_examples}.
Correspondingly, fixing a base orientation $D$ of $G$, the associated \defn{digraph of acyclic orientations} $\mathdefn{\AOD}$ is the subdigraph of its digraph of orientations~$\c O(D)$ induced on the set of acyclic orientations of $G$.
As sub(di)graphs of bipartite (di)graphs, $\AO$ and $\AOD$ are bipartite.
They are graded by the rank functions $\rank_D$.
We denote $\mathdefn{\AODk} \eqdef \{A \in \AO ~;~ \rank_D A = k\}$.

We denote $\mathdefn{\CO}$, $\mathdefn{\COD}$ and $\mathdefn{\CODk}$ the counterparts for cyclic orientations.

\begin{example}\label{exm:AOCompleteGraph}
For $n\geq 1$, the acyclic orientations of the complete graph $K_n$ are in bijection with the permutations of $[n]$.
Indeed, any orientation $D$ of $K_n$ yields a reflexive and transitive relation: $i \leq j$ if and only if there is a directed path from $i$ to $j$ in $D$.
If this orientation is acyclic, then this relation is anti-symmetric, \ie it is a permutation of $[n]$.
Two acyclic orientations of $K_n$ are adjacent if they differ on the orientation of a single arc, that is if the associated permutations differ on the order of two consecutive elements.
Hence the graph of acyclic orientations of the complete graph $\AO[K_n]$ is isomorphic to the permutahedral graph $\Pi_n$.
Furthermore, choosing any base acyclic orientation $D$ of $K_n$, the digraph of acyclic orientations $\AOD$ is isomorphic to the weak Bruhat order $\WeakOrder$.
\end{example}

Our main motivation in the following sections is analyzing Formulation \ref{for1}, thus we say that $G$ is \defn{\good} if $\AO$ is Hamiltonian.

%\begin{question} 
%For which\germain{Cite some papers.} graph $G$ is the graph of acyclic orientations $\AO$ Hamiltonian?
%\end{question}

\begin{example}\label{exm:KnownAOhamiltonicity}
The following are well-known:
\begin{enumerate}[(i)]
\item Trees are {\good}, see \Cref{fig:AO_graphs_examples} (top left):
Since cube-graphs are Hamiltonian, and no orientation of a tree $T$ can be cyclic, we have that $\AO[T] \simeq \cube$ is Hamiltonian.

\item Complete graphs are {\good}:
The permutahedral graph is Hamiltonian since the Steinhaus--Johnson--Trotter algorithm yields a Hamiltonian cycle on $\Pi_n$, see \Cref{exm:PermutahedralGraphIsGood} (or \cite{Johnson63-GenerationOfPermutations} or the survey \cite[Section~3]{Savege97-SurveyGrayCodes}); and the graph of acyclic orientations of the complete graph $K_n$ is isomorphic to $\Pi_n$, see \Cref{exm:AOCompleteGraph}.

\item The cycle graph $C_n$ (see \Cref{fig:AO_graphs_examples} (bottom), the wheel graph $W_n$ (a cycle graph together with a node connected to all others), and the ladder graph $L_n$ (two disjoint paths on $(a_1, \dots, a_n)$ and $(b_1, \dots, b_n)$ together with edges $a_ib_i$ for $i\in[n]$) are {\good} ~if and only if $n$ is odd, see \cite{SavageSquireWest93-GrayCodesAcyclicOrientations}.

\item Chordal graphs (see \Cref{fig:AO_graphs_examples} (top right)) are {\good}, see \cite[Theorem~4.1]{SavageSquireWest93-GrayCodesAcyclicOrientations} and \cite{BrennerCardinalMcConvilleMerinoMutze25-CombinatorialGenerationVII,KorberSchniedersStrickerWalizadeth25-HamiltonianCyclesSupersolvable}.

\item Complete bipartite graphs $K_{m, n}$ (disjoint sets of nodes $(a_1, \dots, a_m)$ and $(b_1, \dots, b_n)$ with edges $a_ib_j$ for all $i\in [m]$ and $j\in [n]$) are not {\good} if either $m$ or $n$ is even.
The authors of \cite{SavageSquireWest93-GrayCodesAcyclicOrientations} proved it using a parity argument, see \Cref{lem:ParityArgument}.
As we will discuss further in Appendix~\ref{app_bipartite}, the case where $n$ and $m$ are both odd is open.

\item We give $22$ examples of small graphs in Appendix~\ref{app:GraphExamples}, deciding $\c{AO}$-Hamiltonicity for some of them.
\end{enumerate}
\end{example}

We are especially interested in operations on graphs that preserve $\c{AO}$-Hamiltonicity.
In particular, we repeatedly use the following lemma, called the \emph{parity problem} in \cite{SavageSquireWest93-GrayCodesAcyclicOrientations}:

\begin{definition}\label{def:signature}
For a graph $G$, and $D$ an acyclic orientation of $G$, the associated \defn{acyclic signature} is $\mathdefn{\sigma(D)} \eqdef \sum_{k} (-1)^k \cdot |\AODk|$.
\end{definition}

\begin{lemma}\label{lem:ParityArgument}
Let $D$ be an orientation of a graph $G$.
If $\sigma(D) \ne 0$, then $G$ is not {\good}.
\end{lemma}

\begin{proof}
If two orientations are adjacent, then their ranks differ by $1$.
Hence a bipartition of $\AO$ is given by the parts $\bigcup_{k \text{ even}} \AODk$ and $\bigcup_{k\text{ odd}} \AODk$.
The contrapositive of \Cref{prop:Bipartite(Path)Hamilitonian}~\ref{item:BipartiteHamiltonian} and the definition of the acyclic signature yield the claim.
\end{proof}

\begin{remark}\label{rmk:PartiyArgumentForCyclicOrientations}
We can state a similar lemma for cyclic orientations:
If $\sum_{k} (-1)^k \cdot |\CODk| \ne 0$, then $G$ is not {\good}.

Indeed, for all $k$, we have $\c O_{k}(D) = \AODk\sqcup \CODk$, and by \Cref{rmk:SumEvenOddBinomNewton}, we know that $\sum_{k} (-1)^k \cdot  |\c O_{k}(D)| = 0$, so applying \Cref{lem:ParityArgument} gives the claim.
\end{remark}

\Cref{ssec:AcyclicSignature} is devoted to the properties and the computation of acyclic signatures.

\subsection{Operations on graphs}\label{ssec:Prelim:OperationsOnGraphs}

For the convenience of the reader, we recall here the principal operations that we will use on (di)graphs. For  (di)graphs $G = (\V, E), G' = (\V', E')$, we define:
\begin{itemize}
\item For some (directed) edge $e\in E$, the \defn{(edge) deletion} $\mathdefn{G\ssm e} \eqdef \bigl(\V, E\ssm\{e\}\bigr)$, and for some $\omega \subseteq E$, the deletion $\mathdefn{G\ssm \omega} \eqdef \bigl(\V, E\ssm \omega\bigr)$.

\item For a pair of \nodes~ (forming an edge or not) $i, j \in \V$, the \defn{contraction} $\mathdefn{G/_{ij}} = (\V', E')$ with $\V' = \bigl(\V \ssm \{i, j\}\bigr) \cup \{v_\circ\}$ for some arbitrary $v_\circ$, and $E'$ is the image (trimmed of its loops and multi-edges) of $E$ by the map $i\mapsto v_\circ$, $j\mapsto v_\circ$ and $k\mapsto k$ for $k\in \V\ssm\{i, j\}$.
For some $\{e_1, \dots, e_r\}\subseteq \binom{\V}{2}$, we define the contraction $\mathdefn{G/_Y} \eqdef (\dots (G/_{e_1})/_{e_2} \dots )/_{e_r}$.
The isomorphism class of $G/_Y$ does not depend on the order of the pairs of \nodes.
% \item For some $i\in \V$, the \defn{\node~ deletion} $\mathdefn{G\ssm i} \eqdef G\bigl[\V\ssm\{i\}\bigr]$,
% \item For some $i\in \V$, the \defn{\node~ contraction} $\mathdefn{G /_{i}} \eqdef G/_{\{e\in E ~;~ i\in e\}}$,
%\end{itemize}

%For two (di)graphs $G = (\V, E)$ and $G' = (\V', E')$, we define:
%\begin{itemize}
\item The \defn{union} $\mathdefn{G\cup G'} \eqdef (\V\sqcup \V', E\sqcup E')$.
\item $\mathdefn{G \times G'}$ the \defn{Cartesian product} on the node set $N \times N'$ and edge set $$\Bigl\{\bigl((x,x'), (x,y')\bigr) ~;~ x'y' \in E',\, x\in \V\Bigr\} \cup \Bigl\{\bigl((x, x'), (y, x')\bigr) ~;~ xy \in E,\, x'\in \V'\Bigr\} \,.$$
% \item For two cliques $K$ of $G$ and $K'$ of $G'$ with the same size $k$, and a bijection $\varphi : K \to K'$, the \defn{$k$-sum} of $G$ and $G'$ is the graph $\mathdefn{G \oplus_\varphi G'} \eqdef (G\cup G')/_{\left\{(i, \varphi(i)) ~;~ i\in K\right\}}$.
% We abbreviate to $\mathdefn{G \oplus_k G'}$ when $K$, $K'$ and $\varphi$ are clear from the context.
\item For disjoint cliques $K_1, \dots, K_r$ of $G$ and disjoint cliques $K_1', \dots, K_r'$ of $G'$ with $|K_p| = |K_p'|$ for all $p\in [r]$, and bijections $\varphi_p : K_p \to K'_p$, the \defn{gluing} of $G$ and $G'$ along $\varphi_1, \dots, \varphi_r$ is the graph $\mathdefn{G \oplus_{\varphi_1, \dots, \varphi_r} G'} \eqdef (G\cup G')/_{\{(i, \varphi_p(i) ~;~ i\in K_i,\, p\in [r]\}}$. We abbreviate $\mathdefn{G\oplus G'}$ when $K_1, \dots, K_r$, $K'_1, \dots, K_r'$, $\varphi_1, \dots, \varphi_p$ are clear from the context.
When $r = 1$, the gluing of $G$ and $G'$ on $K$ and $K'$ with respect to $\varphi$ is called a \defn{$k$-sum} where $k \eqdef |K| = |K'|$, denoted $\mathdefn{G\oplus_{\varphi}G'}$ or $\mathdefn{G\oplus_k G'}$ when $K, K', \varphi$ are clear from the context.
\end{itemize}

% \newpage
\section{Non-$\c{AO}$-Hamiltonicity criteria: modularity, and acyclic polynomials}\label{sec:NonHamiltonicity}\label{sec:AcyclicPolynomial}

\subsection{Prelude: Expanding on the modularity criterion of Savage, Squire and West}\label{ssec:Prelude}

The following \Cref{lem:NbAOgluePathOnEdge,thm:PrismOverTreeNotGood} are directly inspired from the method employed in the proof of \cite[Theorem~3.2]{SavageSquireWest93-GrayCodesAcyclicOrientations} where Savage, Squire and West showed that ladder graphs $L_n \eqdef P_n\times K_2$ are not {\good} if $n$ is even.
They prove the reciprocal in \cite[Section~4.2]{SavageSquireWest93-GrayCodesAcyclicOrientations}, but this is not the purpose of the present section, rather its aim is to develop a useful tool out of the ideas around the number of acyclic orientations modulo $4$.
In particular, we want to understand the $\c{AO}$-Hamiltonicity of the Cartesian product $G \times Q_n$ where $Q_n$ is the $n^{\text{th}}$ cube-graph.

The reader inclined to learn more about classes of graphs which are not {\good} is invited to jump directly to \Cref{cor:CubePrismsEtcAreNotAOH}.
We will need the following two lemmata which are already proven in \cite{SavageSquireWest93-GrayCodesAcyclicOrientations}.
We will discuss their proofs in \Cref{ssec:AcyclicPolynomial}.

\begin{lemma}\label{lem:ParityNbAO}
For any graph $G$, the number of acyclic orientations $\psi(G)$ is even.
\end{lemma}

\begin{lemma}[{\cite[Lemma~3.1]{SavageSquireWest93-GrayCodesAcyclicOrientations}}]\label{lem:DivisibilityBy4Criterion}
If $G$ has an even number of edges and $\psi(G)$ is not divisible by~$4$, then $G$ is not {\good}.
\end{lemma}

\begin{figure}
    \centering

    \begin{subfigure}{0.65\textwidth}
        \centering
        \includegraphics[width=\textwidth]{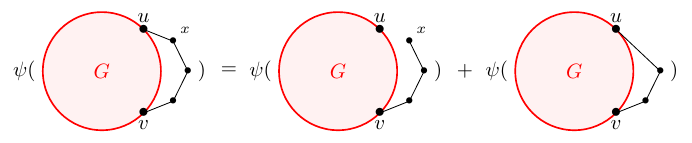}
        % \caption{}
        \label{fig:NbAOgluingPath}
    \end{subfigure}
    \hfill
    \begin{subfigure}{0.25\textwidth}
        \centering
        \includegraphics[width=\textwidth]{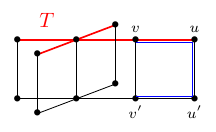}
        % \caption{}
        \label{fig:PrismOverTree}
    \end{subfigure}

    \caption{(Left) Decomposition of the number of acyclic orientations $\psi\bigl(\oplus^{\ell}_{u, v} G\bigr)$.
    (Right) Prism over a tree $T\times K_2$. In \textcolor{red}{red} the tree $T$, in \textcolor{blue}{blue} the path of length $3$ that is glued between the nodes $v$ and $v'$, through the leaf $u$ of $T$.}
    \label{fig:NbAOGluingPathAndPrismOverTree}
\end{figure}

\begin{lemma}\label{lem:NbAOgluePathOnEdge}
Let $G = (\V, E)$ be a graph with an edge $uv\in E$ and $\ell\geq 2$.
Then:
$$\psi\bigl(\oplus_{u, v}^{\ell} G\bigr) = \bigl(2^{\ell} - 1\bigr) \cdot\psi(G) \,.$$

In particular, $\psi(G)$ is divisible by $4$ if and only if $\psi\bigl(\oplus_{u, v}^{\ell} G\bigr)$ is divisible by $4$.
\end{lemma}

\begin{proof}
Let $x$ be the node neighboring $u$ on the induced path of length $\ell$ of $\oplus_{u, v}^{\ell} G$.
We perform a deletion-contraction on the edge $ux$.
By Stanley's formula, we get the following (where $P_{\ell-1}$ denotes a path of length $\ell -1$), see \Cref{fig:NbAOGluingPathAndPrismOverTree} (left):
$$\psi\bigl(\oplus_{u, v}^{\ell} G\bigr) = \psi(G \oplus_1 P_{\ell -1}) + \psi\bigl(\oplus_{u, v}^{\ell - 1} G)$$

Recall that $\psi(H\oplus_1 H') = \psi(H)\cdot \psi(H')$.
We get $\psi(G \oplus_1 P_{\ell -1}) = \psi(G) \cdot 2^{\ell - 1}$.
We proceed by induction, until $\ell = 2$, where the contraction $\oplus_{u, v}^{\ell - 1} G$ is just $G$ itself since $uv\in E$.
Hence, we get:
$$\psi\bigl(\oplus_{u, v}^{\ell} G\bigr) = \left(\sum_{k = 1}^{\ell-1} 2^k\right) \cdot \psi(G)  + \psi(G) = \bigl(2^{\ell} - 1\bigr)\cdot \psi(G)$$

Note that $\bigl(2^{\ell} - 1\bigr)$ is odd for $\ell\geq 2$, so it is co-prime with $4$.
Thus, by Euclid's lemma, $4$ divides $\psi\bigl(\oplus_{u, v}^{\ell} G\bigr)$ if and only if $4$ divides $\psi(G)$.
\end{proof}

% \begin{lemma}\label{lem:NbAOLeafSuppressionInTree}
% Let $T$ be a tree with a leaf $u$, and $K_2$ the graph with two nodes and one edge.
% Then
% $$\psi(T\times K_2) = 7 \cdot \psi\bigl((T\ssm u)\times K_2\bigr) \,.$$
% \end{lemma}

% \begin{proof}

% \end{proof}
Recall that $K_2$ is the graph with two nodes and one edge. 
\begin{theorem}\label{thm:PrismOverTreeNotGood}
Let $T$ be a tree. We have $\psi(T\times K_2) \equiv 2 \mod4$.
Moreover, if $T$ has an even number of nodes (\ie an odd number of edges), then $T\times K_2$ is not {\good}.
\end{theorem}

\begin{proof}
Let $n$ be the number of nodes of $T$ (so that $n-1$ is its number of edges).
Consider a leaf $u$ of $T$, and $v$ its neighboring node in $T$.
We denote $T'$ and $T''$ the two copies of $T$ in $T\times K_2$.
Hence, there is an edge $v'v''$ in $T\times K_2$, and the graph $T\times K_2$ is isomorphic to $\oplus_{v'v''}^{3} (T\ssm \{u\}) \times K_2$, see \Cref{fig:NbAOGluingPathAndPrismOverTree} (right).
By \Cref{lem:NbAOgluePathOnEdge}, we get that if $\psi\bigl((T\ssm\{u\})\times K_2\bigr)$ is not divisible by $4$, then neither is $\psi(T\times K_2)$.
By iteratively stripping $T$ of its leaves, we obtain that if $4$ does not divide $\psi(\{x\}\times K_2)$ where $x$ denotes any arbitrarily chosen node of $T$, then $4$ does not divide $\psi(T \times K_2)$.
Since $\psi(\{x\}\times K_2) = \psi(K_2) = 2$, and the number of acyclic orientations of a graph is always even, we get the first claim.

Moreover, the number of edges of $T\times K_2$ is $3n-2$ which is even if and only if $n$ is even.
We deduce by \Cref{lem:DivisibilityBy4Criterion} that if $n$ is even, then $T\times K_2$ is not {\good}.
\end{proof}

%We now go further in the study of Cartesian products of graphs.

\begin{proposition}\label{prop:NbAOmod4}
Let $e$ be an edge of a graph $G$, and $Q = K_{m_1} \times \dots \times K_{m_s}$ be a Cartesian product of cliques where $m_i = 2^{k_i}$ for some $k_i \geq 1$ for all $i\in [s]$.
Then, we have:
$$\psi(G \times Q) \equiv \psi\bigl((G/_e) \times Q\bigr) + \psi\bigl((G\ssm e) \times Q\bigr) \mod 4\,.$$
% Let $G$ be a graph and $K_{2^n}$ be a complete graph on $2^n$ nodes, then:\germain{Sadly, this does not work: the very end of the proof does not work (the base case of the induction) for $n \ne 1$.}\germainadd{Yet, one should state clearly how it works for $\dots\times K_2$, it is very interesting: it depends on the parity of the number of trees (not forests) that a deletion-contraction algorithm produces on the input $G$}
% $$\psi(G\times K_{2^n}) \equiv \psi(G) \quad \mod 4 \,.$$
\end{proposition}

\begin{proof}
Let $m \eqdef 2^k$ for some $k\geq 1$.
We start by proving the case $s = 1$, that is $\psi(G \times K_m) \equiv \psi\bigl((G/_e) \times K_m\bigr) + \psi\bigl((G\ssm e) \times K_m\bigr) \mod 4$.
For $u\in\V(K_m)$, we denote $G\times u$ the copy of $G$ associated to the node $u$.
Pick an edge $ij$ of $G$, and, for $u\in \V(K_m)$, let $ij\times u$ be the copy the edge $ij$ in $G\times u$.
We are going to apply Stanley's deletion-contraction formula to all the copies of $ij$.
Since $ij\times u$ and $ij\times v$ do not share any node as long as $u\ne v$, contracting or deleting $ij\times u$ leaves $ij\times v$ unaltered.
Consequently, we can delete or contract all the copies of $ij$ independently, leading to an expression of $\psi(G\times K_m)$ as a sum of $|2^{\V(K_m)}| = 2^{m}$ terms.
For $X \subseteq \V(K_m)$, we denote $H(X)$ the graph obtained from $G\times K_m$ by contracting all $ij\times u$ for $u\in X$ and deleting all $ij\times v$ for $v\notin X$.
Using Stanley's formula, we get:
\begin{equation}\label{eqn:IntermediateAOproduct}
\psi(G\times K_m) = \sum_{X\subseteq\V(K_m)} \psi(H(X))
\end{equation}

So far, we have not yet used the fact that $K_m$ is a complete graphs.
The group of automorphisms of the complete graph $K_m$ is the symmetric group on $\V(K_m) = [m]$.
In particular, for any $X, Y\subseteq\V(K_m)$ such that $|X| = |Y|$, there exists an automorphism $\sigma$ of $K_m$ that $\sigma(X) = Y$.
This extends to an automorphism on $G\times K_m$ by defining $\overline{\sigma}: i\times u \mapsto i\times \sigma(u)$.
Moreover, we have $\overline{\sigma}(H(X)) = H(Y)$, so the graphs $H(X)$ and $H(Y)$ are isomorphic.
In particular, $\psi(H(X)) = \psi(H(Y))$ as long as $|X| = |Y|$.
Consequently, we denote $H(k)$ the graph $H(X)$ for any $X\subseteq\V(K_m)$ with $|X| = k$.
Note that $H(0) = (G/_{ij})\times K_m$ since all the copies of $ij$ are contracted in $H(0)$, while $H(m) = (G\ssm ij)\times K_m$ since all the copies of $ij$ are deleted in $H(m)$.
Thus, grouping subsets by sizes, \Cref{eqn:IntermediateAOproduct} simplifies to:
$$\psi(G\times K_m) = \sum_{k = 0}^m \binom{m}{k}\cdot \psi(H(k))$$

It remains to use the fact that $m = 2^k$ and that $\psi(G)$ is an even number for any graph $G$, by \Cref{lem:ParityNbAO}.
By Lucas's theorem \cite{Lucas1878-TheorieDesFonctionsNumeriquementSimplementPeriodiques}, for any $k\notin\{0, 2^k\}$, the binomial coefficient $\binom{2^k}{k}$ is even.
So, for any $k\notin\{0, 2^k\}$, the term $\binom{m}{k}\cdot \psi(H(k))$ is divisible by $4$: all the terms of the above sum except for $k = 0$ and $k = 2^k$ are equal to $0$ modulo $4$.
We obtain:
$$\psi(G\times K_{2^k}) \equiv \psi\bigl((G\ssm ij)\times K_{2^k}\bigr) + \psi\bigl((G/_{ij})\times K_{2^k}\bigr) \mod 4$$

This proves that the function $f : G \mapsto \bigl(\psi(G\times K_{2^k})\mod4\bigr)$ defined from the set of graphs towards $\Z/4\Z$ satisfies the same deletion-contraction formula as the function $g : G \mapsto \bigl(\psi(G)\mod4\bigr)$, \ie both $f(G) = f(G/_e) + f(G\ssm e)$ and $g(G) = g(G/_e) + g(G\ssm e)$ for any edge $e$ of $G$.
Now fix $Q = K_{m_1} \times \dots \times K_{m_s}$ to be a Cartesian product of cliques where $m_i = 2^{k_i}$ for some $k_i \geq 1$ for all $i\in [s]$.
Since we have only used the deletion-contraction formula for $g$ (not the initialization for this recursive formula), one can apply the exact same reasoning to $f$ inductively in order to prove that $\psi(G \times Q) \equiv \psi\bigl((G/_{e}) \times Q\bigr) + \psi\bigl((G\ssm e) \times Q\bigr) \mod 4$.
\end{proof}

\begin{theorem}\label{thm:CubeGraphIsNotGood}
For any $n\geq 1$, the $n\textsuperscript{th}$ cube graph $Q_n \eqdef K_2\times \dots\times K_2$ satisfies $\psi(Q_n) \equiv 2\mod 4$.
Consequently, $Q_n$ is not {\good}.
\end{theorem}

\begin{proof}
The argument boils down to applying \Cref{lem:DivisibilityBy4Criterion} using the first part of the theorem, since $Q_n$ has an even number of edges.
It remains to prove the first claim.

Recall that $K_2$ is just an edge.
For $n = 1$, we have $\psi(Q_1) = \psi(K_2) = 2$.

For $n \geq 2$, note that $Q_n = K_2\times K_2\times \dots\times K_2 = K_2\times Q_{n-1}$.
We denote $e$ the edge of this distinguished ``first'' copy of $K_2$.
By \Cref{prop:NbAOmod4}, we have:
$$\psi(Q_n)
\equiv \psi\bigl((K_2/_e)\times Q_{n-1}\bigr) + \psi\bigl((K_2\ssm e)\times Q_{n-1}\bigr)
\equiv \psi(Q_{n-1}) + \psi(Q_{n-1}\sqcup Q_{n-1})
\mod4\,.$$
Last, remark that $\psi(H\sqcup H') = \psi(H)\cdot\psi(H')$ when $H$ and $H'$ are node-disjoint graphs.
As $\psi(H)$ is divisible by $2$ for all graph $H$, by \Cref{lem:ParityNbAO}, we get that $4$ divides $\psi(H\sqcup H')$ when $H$ and $H'$ are node-disjoint graphs.
Thus, we get:
$\psi(Q_n)
\equiv \psi(Q_{n-1}) \mod4$.
By induction, we get $\psi(Q_n) \equiv 2\mod 4$ for all $n\geq 1$.
\end{proof}

For the sake of completeness, we also prove the following:

\begin{proposition}\label{prop:ProductOfCliquesHave0Mod4NbAO}
Let $Q = K_p \times K_{m_1} \times \dots \times K_{m_s}$ be a Cartesian product of cliques where $m_i = 2^{k_i}$ for some $k_i \geq 1$ for all $i\in [s]$ with $s \geq 1$.
If $p \geq 3$, then $\psi(Q)\equiv 0\mod4$.
\end{proposition}

\begin{proof}
Let $u$ be a node of $K_p$, and let $e_1, \dots, e_{p-1}$ be the edges adjacent to $u$ in $K_p$.
Note that $K_p/_{e_i}$ is isomorphic to $K_{p-1}$, while $e_j$ is an edge in $K_p\ssm e_i$ as long as $i\ne j$.
Furthermore, $(K_p \ssm X)/_{e_i}$ is isomorphic to $K_{p-1}$ for any $X \subseteq \{e_1,\dots, e_{p-1}\}\ssm\{e_i\}$.
Hence, we can perform $p-1$ deletion-contractions on the edges $e_1, \dots, e_{p-1}$ using \Cref{prop:NbAOmod4}.
We get, modulo $4$ (where we denote $Q' \eqdef K_{m_1} \times \dots \times K_{m_s}$):
\begin{align*}
\psi(K_p \times Q')
&\equiv \psi\bigl((K_p/_{e_1} \times Q'\bigr) + \psi\bigl((K_p\ssm {e_1} \times Q'\bigr) \\
&\equiv \psi\bigl(K_{p-1} \times Q'\bigr) + \psi\bigl(((K_p\ssm {e_1})/_{e_2}) \times Q'\bigr) + \psi\bigl((K_p\ssm \{e_1, e_2\}) \times Q'\bigr) \\
&\equiv 2\cdot \psi\bigl(K_{p-1} \times Q'\bigr) + \psi\bigl(((K_p\ssm \{e_1, e_2\})/_{e_3}) \times Q'\bigr) + \psi\bigl((K_p\ssm \{e_1, e_2, e_3\}) \times Q'\bigr) \\
&\equiv \ldots \\
&\equiv (p-1) \cdot \psi\bigl(K_{p-1} \times Q'\bigr) + \psi((K_{p-1}\sqcup\{u\})\times Q')
\end{align*}

Recall that $\psi(H)$ is divisible by $2$ for any graph $H$ by \Cref{lem:ParityNbAO}.
Since $K_{p-1}\sqcup\{u\}$ is a disconnected graph, $\psi((K_{p-1}\sqcup\{u\})\times Q')$ is divisible by $4$ (because $\psi(H\sqcup H') = \psi(H)\cdot \psi(H')$ when $H$ and $H'$ are node-disjoint graphs).
Besides, if $p = 3$, then $4$ divides $(p-1) \cdot \psi\bigl(K_{p-1} \times Q'\bigr)$ as a product of even numbers; whereas if $p \geq 4$, then $4$ divides $(p-1) \cdot \psi\bigl(K_{p-1} \times Q'\bigr)$ by induction on $p$.
This gives the claim.
\end{proof}

\begin{corollary}\label{cor:NotGoodForProductWithK2ImpliesNotGoodForProductWithQn}
Let $G$ be a graph and $n\geq 1$, then
$\psi(G\times K_2)\equiv\psi(G\times Q_n) \mod 4$.

Moreover, if $\psi(G\times K_2)$ is not divisible by $4$ and if either $n\geq 2$ or $G$ has an even number of nodes, then $G\times Q_n$ is not {\good}.
\end{corollary}

\begin{proof}
Consider the two functions on graphs with values in $\Z/4\Z$ defined by $f : G \mapsto \bigl(\psi(G\times K_2)\mod 4\bigr)$ and $g : G\mapsto \bigl(\psi(G\times Q_n)\mod 4\bigr)$.
By \Cref{prop:NbAOmod4}, the functions $f$ and $g$ satisfy the same inductive formula:
$f(G) = f(G/_e) + f(G\ssm e)$ and $g(G) = g(G/_e) + g(G\ssm e)$ for any edge $e$ of $G$.
Hence, to prove that $f$ and $g$ are equal, it is enough to show that they are equal for the initialization graphs, \ie for any graph $G$ which is composed of isolated nodes with no edge.

Let $m$ be the number of nodes of such $G$.
If $m\geq 2$, then both $G\times K_2$ and $G\times Q_n$ have several connected components, so $\psi(G\times K_2)$ and $\psi(G\times Q_n)$ are divisible by $4$ because $\psi(H\sqcup H') = \psi(H)\cdot \psi(H')$ when $H$ and $H'$ are node-disjoint graphs, and $\psi(H)$ is divisible by $2$ for any graph $H$ by \Cref{lem:ParityNbAO}.
If $m = 1$, then $\psi(G\times K_2) = \psi(K_2) = 2$, and, by \Cref{thm:CubeGraphIsNotGood}, $\psi(G \times Q_n) = \psi(Q_n) \equiv 2\mod 4$.
Thus $f = g$ by induction, as claimed.

The ``moreover'' part boils down to using \Cref{lem:DivisibilityBy4Criterion} together with the first part of the corollary, where the condition that $n\geq 2$ or $G$ has an even number of nodes ensures that condition ensures that $G\times Q_n$ has an even number of edges.
\end{proof}

\begin{remark}
Let $Q = K_{m_1} \times \dots \times K_{m_s}$ with $m_i = 2^{k_i}$ for some $k_i \geq 1$ for all $i\in [s]$.
The proof of \Cref{cor:NotGoodForProductWithK2ImpliesNotGoodForProductWithQn} suggest that we can explicitly compute $\bigl(\psi(G\times Q)\mod 4\bigr)$ running the deletion-contraction algorithm on $G$ itself, instead of running it on $G\times Q$.
Indeed, the functions with values in $\Z/4\Z$ defined as $f : G \mapsto \bigl(\psi(G\times Q)\mod 4\bigr)$ and $g : G \mapsto \bigl(\psi(G)\mod 4\bigr)$ satisfy the same recursion, \ie $f(G) = f(G/_e) + f(G\ssm e)$ and $g(G) = g(G/_e) + g(G\ssm e)$ for any edge $e$ of $G$.
However, it is false to say that $\psi(G\times Q) \equiv \psi(G) \mod 4$ because $f$ and $g$ do not have the same value for initialization graphs:
If $G$ consists in $m$ isolated nodes and no edge, then $f(G) \equiv f(Q)$ if $m = 1$ and $f(G) = 0$ if $m\geq 2$, but $g(G) = 1$ for all $m\geq 1$.

\smallskip

For the case $Q = K_2$, and more generally for the cube-graph $Q = Q_n$, one can compute recursively $\psi(G\times K_2)\mod4$ as follows.
Fix a graph $G$.
\defn{Running} the deletion-contraction algorithm on $G$ \defn{until we get forests}, means if $G$ is a forest we do nothing, and otherwise we choose an edge $e$ of $G$, then run the deletion-contraction algorithm on $G/_{e}$ and on $G\ssm e$ until we get forests.
Let $t$ be the number of trees produced by running the deletion-contraction algorithm on $G$ until we obtain forests, then $\psi(G\times K_2) \equiv \psi(G\times Q_n) \equiv 2\cdot t\mod4$.
Hence, $\psi(G\times K_2)$ and $\psi(G\times Q_n)$ are divisible by $4$ if and only if $t$ is even.
To prove this claim, note that, by \Cref{prop:NbAOmod4,cor:NotGoodForProductWithK2ImpliesNotGoodForProductWithQn}, we have $\psi(G\times Q_n) \equiv \psi(G\times K_2) \equiv \psi\bigl((G/_e) \times K_2\bigr) + \psi\bigl((G\ssm e) \times K_2\bigr)$, and by \Cref{thm:PrismOverTreeNotGood}, we have $\psi(G\times K_2) \equiv 2\mod 4$ when $G$ is a tree, while $\psi(G\times K_2) \equiv 0\mod 4$ when $G$ is disconnected.

\smallskip

On the other hand, let $k_1 \geq 2$ and $Q = K_{m_1} \times \dots \times K_{m_s}$ with $m_i = 2^{k_i}$ for some $k_i \geq 1$ for all $i\geq 2$.
By the same reasoning as in the proof of \Cref{cor:NotGoodForProductWithK2ImpliesNotGoodForProductWithQn}, $\psi(G\times Q)$ is equivalent modulo $4$ to a certain multiple of $\psi(Q)$.
Since $\psi(Q) \equiv 0\mod 4$ by \Cref{prop:ProductOfCliquesHave0Mod4NbAO}, this implies $\psi(G\times Q) \equiv 0\mod 4$ for all $G$, and yields no conclusion on whether $G\times Q$ is {\good} or not.
\end{remark}

% \begin{proof}
% The number of edges of the Cartesian product $A \times B$ is $|\V(A)|\cdot|E(B)| + |E(A)|\cdot|\V(B)|$.
% Thus, if $n \geq 2$, then the number of edges of $G \times K_{2^n}$ is $2^n\cdot |E(G)| + \binom{2^n}{2}\cdot|\V(G)|$, where $\binom{2^n}{2}$ is even, according to Lucas' theorem (or according to Kummer's theorem, or by direct computation): in particular, this is even.
% On the other hand, the number of edges of $G\times K_2\times\dots \times K_2$ with $m$ copies of $K_2$ is $2^m\cdot |E(G)| + m\cdot2^{m-1}\cdot|\V(G)|$ which is even either if $m \geq 2$ or if $|\V(G)|$ is even.

% Fix $G$ and $n_1, \dots, n_r \geq 1$ with either $r\geq 2$ or $n_i \geq 2$ for some $i\in [r]$ or $|\V(G)|$ even.
% Then $H \eqdef G\times K_{2^{n_1}} \times K_{2^{n_2}} \times \dots \times K_{2^{n_r}}$ has an even number of edges by the previous argument.
% By \Cref{thm:ProductWithProductOfCompleteGraphs}, if $\psi(G)$ is not divisible by $4$, then $\psi(H)$ is also not divisible by $4$.
% Consequently, by \Cref{lem:DivisibilityBy4Criterion}, the graph $H$ is not {\good}.
% \end{proof}

We can now prove \Cref{thmZ} by listing classes of graphs which are not {\good}.
Beyond the usual path on $n$ nodes $P_n$, cycle on $n$ nodes $C_n$, star with $n$ outer nodes $K_{1, n}$, and complete graph on $n$ nodes $K_n$, we need to define the following two classes of graphs.

The \defn{wheel graph} $W_n$ is the cycle $C_n$ together with a \defn{central node} adjacent to all nodes of $C_n$.

For a tree $T = ([s], E)$ and numbers $n_1, \dots, n_s \geq 3$, a \defn{$(T,\, n_1, \dots, n_s)$-tree of cycles} is any graph isomorphic to a $2$-sum $H \oplus_2 C_{n_s}$ where $H$ is a $(T\ssm s,\, n_1, \dots, n_{s-1})$-tree of cycles (of course for $s = 1$, a $(T,\, n_1)$-tree of cycle is just a cycle graph $C_{n_1}$).

\begin{corollary}\label{cor:CubePrismsEtcAreNotAOH}
The following graphs are not {\good}:
\begin{compactenum}
\item Ladder graphs $L_n \eqdef P_n \times K_2$ for any $n$ even,\label{item:NotGood:Ladder}
\item Book graphs $B_n \eqdef K_{1, n} \times K_2$ for any $n$ odd,\label{item:NotGood:Book}
\item Prism graphs $\mathrm{Pr}_n \eqdef C_n\times K_2$ for any $n$ even,\label{item:NotGood:Prisms}
\item Toroidal grid graphs $T_{4, m} \eqdef C_4\times C_m$ for any $m$ even,\label{item:NotGood:ToroidalGrids}
\item Cube graphs $Q_n = K_2 \times \dots\times K_2$, with $n$ copies of $K_2$, for any $n\geq 2$,\label{item:NotGood:Cube}
\item The graph $T\times Q_n$ for any tree $T$ where either $n \geq 2$ or $T$ has an even number of nodes,\label{item:NotGood:TreeTimesCube}
\item Thick wheel graphs $\mathrm{TW}_n \eqdef W_n \times K_2$ for any $n$ odd,\label{item:NotGood:PrismOverWheel}
\item Any $(T,\, n_1, \dots, n_s)$-tree of cycles for any tree $T$ and any $n_1, \dots, n_s\geq 1$, if there is some $i\in [d]$ such that $n_i$ is even and if $\sum_{i=1}^s (n_i-1)$ is odd.\label{item:NotGood:TreeOfCycle}
\end{compactenum}
\end{corollary}

\begin{proof}
\eqref{item:NotGood:Ladder} and~\eqref{item:NotGood:Book}:
Ladder and book graphs are of the form $T\times K_2$ where $T$ is a tree (respectively a path and a star), so by \Cref{thm:PrismOverTreeNotGood}, they are not {\good} when they have an even number of nodes.
Ladder graphs already appear in \cite[Theorem~3.2]{SavageSquireWest93-GrayCodesAcyclicOrientations}.

\eqref{item:NotGood:Prisms}:
Let $\mathrm{Pr}_n = C_n \times K_2$.
We are going to prove by induction that $\psi(\mathrm{Pr}_n)$ is divisible by $4$ if $n$ is odd, and not divisible by $4$ if $n$ is even.
We denote $C_n'$ and $C_n''$ the two copies of $C_n$ in $\mathrm{Pr}_n$.
Fix any edge $e$ of $C_n$, and let $e'$ and $e''$ the copies of $e$ in $C_n'$ and $C_n''$.
We denote $H'$ (resp. $H''$) the graph obtained by contracting $e'$ and deleting $e''$ (resp. by deleting $e'$ and contracting $e''$) in $\mathrm{Pr}_n$.
Note that $H'$ and $H''$ are isomorphic.
By deletion-contraction on the edges $e'$ and $e''$, we have (using $C_n/_e \simeq C_{n-1}$):
\begin{align*}
\psi(\mathrm{Pr}_n)
&= \psi\bigl((C_n\ssm e) \times K_2\bigr) + \psi(H') + \psi(H'') + \psi\bigl((C_n/_e)\times K_2\bigr) \\
&= \psi\bigl((C_n\ssm e) \times K_2\bigr) + 2\cdot\psi(H') + \psi(\mathrm{Pr}_{n-1}) \\
&\equiv \psi\bigl((C_n\ssm e) \times K_2\bigr) + \psi(\mathrm{Pr}_{n-1}) \mod 4
\end{align*}

Since $(C_n\ssm e) \times K_2$ is a ladder graph on an even number of nodes, $4$ does not divide $\psi\bigl((C_n\ssm e) \times K_2\bigr)$, by the proof of~\eqref{item:NotGood:Ladder}, so $\psi\bigl((C_n\ssm e) \times K_2\bigr)\equiv 2\mod 4$.
Hence, by induction, $\psi(\mathrm{Pr}_n) \equiv 2\cdot (n-1) \mod 4$ because $\mathrm{Pr}_2 = C_4$ with $\psi(C_4) = 14 \equiv 2 \mod 4$.
Thus, if $n$ is even, $\psi(\mathrm{Pr}_n)$ is not divisible by $4$.
As the number of edges of $\mathrm{Pr}_n$ is $3n$, we get the claim.

\eqref{item:NotGood:ToroidalGrids}:
We just proved that $\psi(C_m\times K_2) \equiv 2\mod 4$.
Since $C_4 = K_2\times K_2 = Q_2$, by \Cref{cor:NotGoodForProductWithK2ImpliesNotGoodForProductWithQn}, we have $\psi(C_m\times C_4) \equiv \psi(C_m\times K_2) \not\equiv 0\mod 4$, so by \Cref{lem:DivisibilityBy4Criterion}, we get that $T_{4, m} = C_m\times C_4$ is not {\good} for all $m$ even because $T_{4, m}$ has an even number of edges.

\eqref{item:NotGood:Cube}:
This is \Cref{thm:CubeGraphIsNotGood}.

\eqref{item:NotGood:TreeTimesCube}:
Combine \Cref{thm:PrismOverTreeNotGood,cor:NotGoodForProductWithK2ImpliesNotGoodForProductWithQn} where $T\times Q_n$ has an even number of edges if $n\geq 2$ or if $T$ has an even number of nodes.

\eqref{item:NotGood:PrismOverWheel}:
We use the same decomposition that Savage, Squire and West applied to the wheel graph $W_n$ in the proof of \cite[Theorem~3.1]{SavageSquireWest93-GrayCodesAcyclicOrientations}.
Let $e$ be and edge of $W_n$ which does not contain the central node, and let $G_n = W_n \ssm e$ and $\mathrm{TG}_n = G_n \times K_2$.
Note that $W_n/_e = W_{n-1}$, hence by deletion-contraction and \Cref{prop:NbAOmod4}, we have:
$\psi(\mathrm{TW}_n) \equiv \psi(\mathrm{TW}_{n-1}) + \psi(\mathrm{TG}_n)\mod4$.
Moreover, let $f$ be and edge of $G_n$ not containing the central node $x$, then $G_n/_f = G_{n-1}$ and $G_n\ssm f$ is a $1$-sum of $G_{n-1}$ with a single edge.
In particular, by \Cref{lem:NbAOgluePathOnEdge}, $\psi\bigl((G_n \ssm f) \times K_2\bigr) \equiv \psi(\mathrm{TG}_{n-1})\mod 4$ since $(G_n \ssm f)\times K_2$ is obtained from $\mathrm{TG}_{n-1}$ by gluing a path of length $3$ on the edge $\{x\}\times K_2$.
Thus, by \Cref{lem:ParityNbAO,prop:NbAOmod4}, we get:
$\psi(\mathrm{TG}_{n}) \equiv 2\cdot \psi(\mathrm{TG}_{n-1}) \equiv 0\mod4$.
Consequently, $\psi(\mathrm{TW}_n) \equiv \psi(\mathrm{TW}_{n-1})\mod 4$, and by induction, since $\psi(\mathrm{TW}_1) \equiv \psi(C_4) \equiv 2\mod 4$, we obtain $\psi(\mathrm{TW}_n) \equiv 2\mod4$.
As the wheel graph $W_n$ has an even number of edges and $n+1$ nodes, the thick wheel graph has an even number of edges when $n$ is odd, so we conclude by \Cref{lem:DivisibilityBy4Criterion}.

\eqref{item:NotGood:TreeOfCycle}:
Fix a $(T, n_1, \dots, n_s)$-tree of cycle $G$ with $n_i$ even for some $i\in [s]$.
Note that the number of edges of $G$, namely $1 + \sum_{i=1}^s (n_i - 1)$, is even by hypothesis.
Let $i_1, \dots, i_s$ be an enumeration of $[s]$, with $n_{i_1}$ even, such that $T_j = T[\{i_1, \dots, i_j\}]$ is a tree.
In particular, $i_j$ is a leaf in $T_j$.
Let also $G_j$ be the $(T_j, n_{i_1}, \dots, n_{i_j})$-subtree of cycle of $G$.
As $G_j$ is obtained from $G_{j-1}$ by gluing a path of length $n_j - 1$ on an edge, $\psi(G_j)$ is divisible by $4$ if and only if $\psi(G_{j-1}) \mod 4$ is divisible by $4$, by \Cref{lem:NbAOgluePathOnEdge}.
Since $G_1$ is an even cycle with $n_{i_1}$ nodes, $\psi(G_1) \equiv 2\mod 4$ by \Cref{exm:AcyclicPolynomialCycles}, so $\psi(G)$ is not divisible by $4$.
By \Cref{lem:DivisibilityBy4Criterion}, $G$ is not {\good}.
\end{proof}

\begin{remark}
We have only presented a few possible applications of \Cref{lem:NbAOgluePathOnEdge,prop:NbAOmod4} together with the usual deletion-contraction formula.
The reader is invited to try to compute $\psi(G\times K_2)\mod 4$ for their favorite (family of) graphs $G$, and conclude on the $\c{AO}$-Hamiltonicity of $G\times K_2$ if the number of nodes of $G$ is odd, using \Cref{lem:DivisibilityBy4Criterion}.

Note that any result proven for $G\times K_2$ immediately extends to $G\times Q_n$ for $n\geq 1$ thanks to \Cref{cor:NotGoodForProductWithK2ImpliesNotGoodForProductWithQn}.
\end{remark}

We give one last structural result relying on the same ideas.
This can be used to start tackling non-$2$-connected graphs.

\begin{proposition}
If $G$ is disconnected, then $\psi(G\times Q_n) \equiv \psi(G) \equiv 0\mod4$ for all $n\geq 1$.

Moreover, for any graph $G$ and any tree $T$, we have for all $n\geq 1$:
$$\psi\bigl((G\oplus_1 T)\times Q_n\bigr) \equiv \psi\bigl((G\oplus_1 T)\times K_2\bigr) \equiv \psi(G\times K_2) \mod4 \,.$$
\end{proposition}

\begin{proof}
The first statement comes from $\psi(H\sqcup H') = \psi(H)\cdot \psi(H')$ for node-disjoint graphs $H$ and $H'$, and $2$ divides both $\psi(H)$ and $\psi(H')$ by \Cref{lem:ParityNbAO} (note that $(H\sqcup H')\times Q_n$ is also disconnected).

The first equivalence of the second statement is given by \Cref{cor:NotGoodForProductWithK2ImpliesNotGoodForProductWithQn}.
For the second one, let $u$ be a leaf of $T$ which is not the one shared by $G$ in $G\oplus_1 T$, let $v$ be the unique neighbor of $u$ in $T$, and let $T' = T\ssm u$.
Then $(G\oplus_1 T)\times K_2$ is obtained from $(G\oplus_1 T')\times K_2$ by gluing a path of length $3$ on the edge $\{v\}\times K_2$.
By \Cref{lem:NbAOgluePathOnEdge}, we have that $\psi\bigl((G\oplus_1 T)\times K_2\bigr) \equiv 7\cdot \psi\bigl((G\oplus_1 T')\times K_2\bigr) \equiv \psi\bigl((G\oplus_1 T')\times K_2\bigr) \mod4$.
Trimming $T$ of its leaves recursively, we get the claim.
\end{proof}

We conclude by discussing further possible research angles that use the same core idea of counting the number of acyclic orientations of graphs and checking its divisibility by $4$.

\begin{example}\label{exm:ClassesOfNonAOHgraphs}
Our computer experiments show that the following graphs $G$ satisfy $\psi(G) \equiv 2 \mod4$, and hence are not {\good} since they have an even number of edges:
\begin{compactenum}
\item The octahedral graph;
\item The Franklin graph;
\item The Herschel graph;
\item The M\"obius--Kantor graph;
\item The Desargues graph;
\item The triangular grid graphs $\mathrm{Tr}_{n, m}$ for $(n, m) \in \{(3, 3), (3, 5)\}$;
\item The graph of the antiprisms $A_n$ for $n \in \{3, 6\}$ but not for $n\in\{4, 5, 7, 8\}$;
\item The grid graphs $\mathrm{Gr}_{n, m} = P_n \times P_m$ for $(n, m) \in \{(3, 3), (3, 5), (3, 7)\}$ but not for $n = m = 4$;
\item The complete tripartite graphs $K_{m, n, p}$ for $n = m = 2$ and $2\leq p \leq 7$ and for $(n, m, p) = (2, 3, 4)$\item The generalized Petersen graphs $\mathrm{Pet}(n, k)$ for $(n, k) \in \{(6, 1), (8, 1), (8, 3), (10, 1), (10, 3)\}$ but not for $(n, k) \in \{(6, 2), (8, 2), (10, 2), (10, 4)\}$.
\end{compactenum}

The reader must be careful: each ``but not'' means that $\psi(G) \equiv 0\mod4$, it does not imply that $G$ is {\good}.
Note also that the number of acyclic orientations of Chv\'atal graphs is divisible by $4$, yet it is not {\good}, see \Cref{tab:SmallGraphsDetailed}: we will need the notion of acyclic signature studied in \Cref{ssec:AcyclicSignature} to understand this case.
\end{example}

In view of the previous example list, we conjecture the following.

\begin{conjecture}
For the following graphs $G$, we have $\psi(G) \equiv 2\mod4$, and consequently $G$ is not {\good} since $G$ has an even number of edges:
\begin{compactenum}
\item The grid graphs $\mathrm{Gr}_{n, m} = P_n \times P_m$ for $m, n\geq 3$ both odd;
\item The triangular grid graphs $\mathrm{Tr}_{n, m}$ for $m, n\geq 3$ both odd;
\item The generalized Petersen graphs $\mathrm{Pet}(n, k)$ for $n$ even and $k$ odd;
\item The complete tripartite graphs $K_{m, n, p}$ for $n$, $m$, $p$ not all odd.
\end{compactenum}
\end{conjecture}

We believe that methods similar to the one studied above for path gluing and Cartesian products may be developed in these cases.
Note that the ladder graphs (\ie grid graphs $\mathrm{Gr}_{2, m}$) and the complete bipartite graphs were already dealt with by Savage, Squire and West \cite{SavageSquireWest93-GrayCodesAcyclicOrientations}.

\subsection{The acyclic polynomial in general}\label{ssec:AcyclicPolynomial}

For graphs, using a deletion-contraction recursion proven by Stanley \cite{Stanley-acyclicOrientations}, one can compute the number~\mathdefn{$\psi(G)$} of acyclic orientations of the graph $G = (\V, E)$:
for all $uv\in E$, we have $\psi(G) = \psi(G/_{uv}) + \psi(G\ssm uv)$.
This yields a counting algorithm with time complexity $O\bigl(\varphi^{|\V| + |E|}\bigr)$ where $\varphi \eqdef \frac{1 + \sqrt{5}}{2}$, see \cite[end of Section~2.3]{Wilf-AlgorithmAndComplexityBook}.
Our ultimate goal is to present a similar, although incomplete, algorithm for computing the acyclic signature $\sigma(D)$ of an orientation of $G$.
This will be done in \Cref{ssec:AcyclicSignature}. It is natural to count acyclic orientations according to their rank (with respect to a fixed base orientation).

\subsubsection{Definition and first examples}

For a digraph $D$, the \defn{acyclic polynomial} is defined as
$$\mathdefn{\Psi(D; t)} \eqdef \sum_{A\in \AO[D]} t^{\rank_D A} \,.$$
By convention, $\Psi(D; t) = 1$ if $D$ has no edge.
Clearly, for any orientation $D$ of $G$, one has $\psi(G) = \Psi(D; 1)$, and $\sigma(D) = \Psi(D; -1)$.
 
We give $22$ examples of acyclic polynomials for small, well-known graphs in Appendix~\ref{app:GraphExamples}.

\begin{remark}
The coefficients of $\Psi(D; t)$ are the Whitney numbers of the second kind of the poset of regions of the graphic arrangement with base region given by $D$.
See \cite[Section~3.10]{stanleyenumerative}, for the general definition and background.
We will not make use of this perspective, since our goal is to present graph-theoretic proofs of every claim.
\end{remark}

%The acyclic polynomials of some standard graphs are easy to compute.

\begin{example}[Trees]\label{exm:AcyclicPolynomialTrees}
For a tree $T$ with $m$ edges, recall that all orientations are acyclic.
Hence, for any orientation $D$ of $T$, we get $|\AODk[D]| = |\c O_{k}(D)| = \binom{m}{k}$.
Thus we get
$$\Psi(D ; t) = \sum_{k} \binom{m}{k} t^k = (1+t)^m \,.$$
We recover that the number of acyclic orientations of $T$ is $\psi(T) = \Psi(D; 1) = 2^m$ while its signature is $\sigma(D) = \Psi(D; -1) = 0$.
\end{example}

\begin{example}[Cycles]\label{exm:AcyclicPolynomialCycles}
For a cycle $C_n = ([n], E)$ all orientations except 2 are acyclic.
Let $D$ be one of the cyclic orientations of $C_n$, hence $D^E$ is the other cyclic orientation of $C_n$.

Thus $\Psi(D; t) = (1+t)^n - 1 - t^n\,.$
We recover that the number of acyclic orientations of $C_n$ is $\psi(C_n) = \Psi(D; 1) = 2^n - 2$ while its signature is $\sigma(D) = \Psi(D; -1) = \left\{\begin{array}{ll}
0 & \text{ if } n \text{ is odd} \\
-2 & \text{ if } n \text{ is even}
\end{array}\right.$.

Changing the base orientation from a cyclic orientation $D$ to some acyclic orientation $A = D^X$ with $|X| = q$, we easily get
$\Psi(A; t) = (1+t)^n - t^q - t^{n-q}\,,$ yielding $\sigma(A) = \left\{\begin{array}{ll}
0 & \text{ if } n \text{ is odd} \\
(-1)^{q+1}\cdot 2 & \text{ if } n \text{ is even}
\end{array}\right.$.
\end{example}

\begin{example}\label{exm:AcyclicPolynomialK23}
As our running example, the complete bipartite graph $K_{2, 3}$ has $46$ acyclic orientations, and its acyclic signature is $\pm 6$ depending on the chosen orientation.
We depict all the possible acyclic polynomials  in \Cref{fig:AcyclicPolyK23}.
Out of the $2^6 = 64$ orientations of $K_{2, 3}$, we only get $8$ different polynomials:
as we will see in \Cref{cor:PsiOfReverseOrientation}, reversing the directions of all arcs yield the same acyclic polynomial;
and moreover, it is clear that if two base orientations $D$ and $D'$ differ by an automorphism of the underlying graph $G$, then their associated acyclic polynomials are identical.
\end{example}

% \germain{Think about other examples.}
% \begin{example}
% \germainadd{Vertebrate or skeletal graphs (\ie graphs for which $\AO$ is a lattice, \cite{Pilaud2024-AcyclicReorientations})???}
% \end{example}

\begin{figure}
    \centering
    \includegraphics[width=\linewidth]{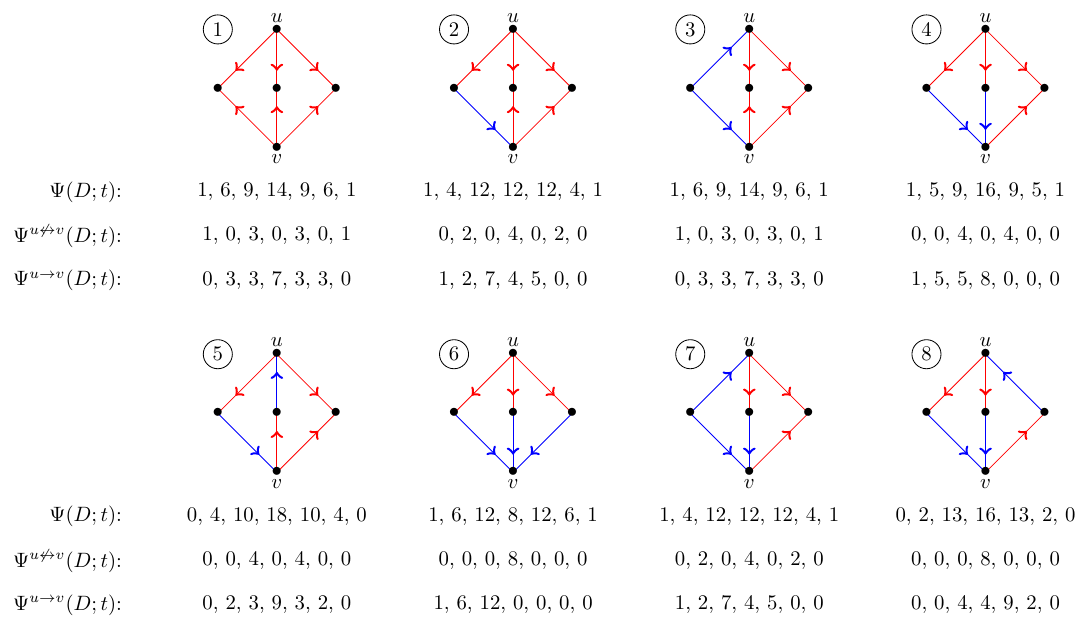}
    \caption[Acyclic polynomials of $K_{2, 3}$ for several base orientations]{The $8$ possible orientations of $K_{2, 3}$ up to symmetry and reversion of the directions of all arcs, with the corresponding acyclic polynomials and their decomposition.
    Each polynomial is given as a list of coefficients, \eg ``$1, 2, 7, 4, 5, 0, 0$'' means $1 + 2t + 7t^2 + 4t^3 + 5t^4 + 0t^5 + 0t^6$.
    Here, we write $K_{2, 3}$ as a bipartition between on the one side the graph with nodes $u, v$ and no edge, and on the other side the remaining edges (\ie $K_{2, 3}$ itself).
    Following the decomposition of \Cref{prop:AcyclicPolyDecomposition,cor:PsiIsPalindromic}, note that $\Psi(D; t)$ and $\Psi^{u\not\leftrightarrow v}(D; t)$ are palindromic, contrary to $\Psi^{u\to v}(D; t)$.
    Furthermore, $\Psi(D; t)$ is equal to the sum of $\Psi^{u\not\leftrightarrow v}(D; t)$, and $\Psi^{u\to v}(D; t)$, and the reverse of $\Psi^{u\to v}(D; t)$.
    Besides, the constant coefficient (and the coefficient on $t^6$) of $\Psi(D; t)$ is non-zero if and only if $D$ is acyclic.
    The number of acyclic orientations $\psi(K_{2, 3}) = 46$ is the sum of the coefficients of $\Psi(D; t)$, while the acyclic signature $\sigma(K_{2, 3}) = 6$ is their alternating sum (up to sign).}
    \label{fig:AcyclicPolyK23}
\end{figure}

\begin{remark}\label{rmk:NotATutteWhitneyInstance}
Since we will show that the acyclic polynomial admits a deletion-contraction phenomenon (see \Cref{ssec:DeletionContractionPhenomenon}), one might wonder if $\Psi(D; t)$ is a specialization of the famous Tutte--Whitney polynomial of $G$ for some orientation $D$ of the graph $G$.
We conjecture that this is not possible, and we now argue against it.

Let $G = K_{2, 3}$, and $D$ the orientation $\circled{1}$ of $G$ of \Cref{fig:AcyclicPolyK23}.
The Tutte--Whitney polynomial of $G$ is
$T_{K_{2, 3}}(x, y) = x^4 + 2x^3 + 3x^2 + 3xy + y^2 + x + y \,.$

On the other hand, $\Psi(D; t) = 1 + 6t + 9t^2 + 14t^3 + 9t^4 + 6t^5 + 1$ is of degree $6$ (since $D$ has been chosen acyclic), as indicated in \Cref{fig:AcyclicPolyK23}.
There are many ways to interpret ``being a specialization of $T_{K_{2, 3}}$''; we are going to prove that it is impossible that there exist polynomials $A(t), B(t)$ with real coefficients, such that $T_{K_{2, 3}}\bigl(A(t), B(t)\bigr) = \Psi(D; t)$.

Indeed, suppose it is the case, and let $a$ (resp. $b$) be the degree of $A(t)$ (resp. $B(t)$).
The degree of $T_{K_{2, 3}}\bigl(A(t), B(t)\bigr)$ depends on which of $4a$ and $2b$ is the greatest (since $a+b < \max(4a, 2b)$).
In particular, we need to have $\max(4a, 2b) \geq 6$.
\begin{compactitem}
\item If $4a > 2b$, then the degree of $T_{K_{2, 3}}\bigl(A(t), B(t)\bigr)$ is $4a$ which needs to be equal to $6$, but cannot.
\item If $4a < 2b$, then the degree of $T_{K_{2, 3}}\bigl(A(t), B(t)\bigr)$ is $2b$ which needs to be equal to $6$, so $b = 3$ and $a \leq 1$.
Hence, writing $A(t) = \alpha_0 + \alpha_1t$ and $B(t) = \beta_0 + \beta_1t + \beta_2t^2 + \beta_3t^3$, one can develop $T_{K_{2, 3}}\bigl(A(t), B(t)\bigr) = \Psi(D; t)$ to get $7$ equations involving the $6$ variables $\alpha_0, \alpha_1, \beta_0, \beta_1, \beta_2, \beta_3$.
A tedious computation (or a solver of polynomial systems) shows that this explicit system has no solutions.
\item If $4a > 6$ and $2b > 6$, then $4a = 2b$ and the leading coefficients of $A(t)^4$ and $B(t)^2$ compensate each other (otherwise the degree of $T_{K_{2, 3}}\bigl(A(t), B(t)\bigr)$ would not be~$6$).
Yet, since $4$ and $2$ are even powers, this is not possible.
\end{compactitem}

Consequently, there exists a graph $G$ (\eg $G = K_{2, 3}$) and an orientation $D$ of $G$ (actually several) such that for all polynomials $A(t), B(t)\in \R[t]$, we have $T_G(A(t), B(t)) \ne \Psi(D; t)$.
However, this does not rule out the possibility that $T_G(A(t), B(t)) = \Psi(D; t)$ for some orientation $D$ of some graph $G$ and for some polynomials $A(t), B(t)\in \C[t]$.
We still believe this cannot happen except in trivial cases, but cannot (un)prove it.
\end{remark}

\begin{remark}\label{rmk:ChromaticPolynomial}
It is well-known (\eg see \cite{Stanley-acyclicOrientations}) that $\psi(G) = \chi(G; -1)$ where $\chi(G; x)$ is the chromatic polynomial of $G$, \ie $\chi(G; c)$ is the number of ways one can color $G$ using $c$ colors, such that no edge has both endpoints of the same color.
The chromatic and acyclic polynomials shall not be confused: in general $\Psi(D; t) \ne \chi(\underline{D}; -t)$ for a digraph $D$.
For example, one can compare the results of \Cref{fig:AcyclicPolyK23} with: $\chi(K_{2, 3}; x) = 7x -17x^2 + 15x^3 - 6x^4 + x^5$.

We do not know if there exists a polynomial $\Phi(D; t, x)$, satisfying a deletion-contraction recursion, from which both $\Psi(D; t)$ and $\chi(\underline{D}; x)$ can be naturally recovered (\eg such that $\Phi(D; t, 1) = \Psi(D; t)$ and $\Phi(D; 1, x) = \chi(\underline{D}; x)$ for all $x, t$).
We saw in \Cref{rmk:NotATutteWhitneyInstance} that if such a polynomial were to exist, it could not be a specification of the Tutte--Whitney polynomial.
\end{remark}

%As a first property, we show a non-trivial relation between the polynomials $\bigl(\Psi(D; t) ~;~ D\in \c O(G)\bigr)$.

\begin{proposition}\label{prop:SumAcyclicPolynomials}
Let $G$ be a graph and $\c O(G)$ the set of all its orientations, then
$$\sum_{D\in \c O(G)} \Psi(D; t) = (1 + t)^{|E|}\cdot \psi(G).$$
\end{proposition}

\begin{proof}
Note that we have $\rank_D A = \rank_A D$, by definition (the number of arcs on which $A$ differs from $D$ is equal to the number of arcs on which $D$ differs from $A$).
We have:
$$\sum_{D\in \c O(G)} \Psi(D; t)
= \sum_{\substack{D\in \c O(G)\\A\in \AO}} t^{\rank_D A}
= \sum_{A\in \AO} \sum_{D\in \c O(G)} t^{\rank_A D}
= \sum_{A\in \AO} (1 + t)^{|E|}
= (1 + t)^{|E|}\cdot \psi(G)$$
where for any $A\in \AO$ we have: $\sum_{D\in \c O(G)} t^{\rank_A D} =\sum_{X\subseteq E} t^{\rank_A A^X} = \sum_{X\subseteq E} t^{|X|} = (1 + t)^{|E|}$.
\end{proof}

\begin{example}
For the complete bipartite graph $K_{2, 3}$, \Cref{prop:SumAcyclicPolynomials} gives $\sum_{D\in \c O(K_{2, 3})} \Psi(D; t) = 46\cdot (1 + t)^6$.
This can be checked by summing the polynomials of \Cref{fig:AcyclicPolyK23}, taking into account the reversion of all the arcs of each orientation and their multiplicities.
\end{example}

\subsubsection{Disconnected graphs and 1-sums of graphs}

As an easy property of acyclic polynomials, we show that the acyclic polynomial of a disconnected graph is the product of the acyclic polynomials of its connected components, while the acyclic polynomial of a $1$-sum of graphs is the product of the acyclic polynomials of the summands.

\begin{proposition}\label{prop:AcyclicPol:Disconnected}
Let $G = (\V, E)$ be a disconnected graph with connected components $\V_1, \dots, \V_r$, and $D$ an orientation of $G$, then:
$$\Psi(D; t) = \prod_{k=1}^r \Psi(D[\V_k]; t)\,.$$
\end{proposition}

\begin{proof}
It is trivial that the map $A \mapsto (A[\V_1], \dots, A[\V_r])$ is a bijection between $\AO$ and $\AO[{G[\V_1]}] \times \dots \times \AO[{G[\V_r]}]$.
Moreover, given a fixed acyclic orientation $D$ of $G$, we have $\rank_D A = \rank_{D[\V_1]} A[\V_1] + \dots + \rank_{D[\V_r]} A[\V_r]$.
Thus:
$$\Psi(D; t) = \sum_{A\in \AO} t^{\rank_D A} = \sum_{A_1, \dots, A_r\in \AO[{G[\V_1]}] \times \dots \times \AO[{G[\V_r]}]} \prod_{k=1}^r t^{\rank_{D[\V_k]} A[\V_k]} = \prod_{k=1}^r \Psi(D[\V_k]; t) \,,$$
where the last equality comes from exchanging the sum with the product and summing over each $\AO[{G[\V_k]}]$ independently.
\end{proof}

\begin{proposition}\label{prop:AcyclicPol:1sums}
Let $G = H\oplus_1 H'$ be a $1$-sum of graphs.
For any orientation $D$ of $G$, one naturally has $D = D[H]\oplus_1 D[H']$, and obtains
$\Psi(D; t) = \Psi(D[H]; t) \cdot \Psi(D[H']; t)\,.$
\end{proposition}

\begin{proof}
The map $A \mapsto (A[H], A[H'])$ is a bijection between $\AO$ and $\AO[H]\times \AO[H']$ since it is impossible to create a (directed) cycle in $G$ using arcs of both $H$ and $H'$.
Moreover, $\rank_D A = \rank_{D[H]} A[H] + \rank_{D[H']} A[H']$.
The claim follows.
\end{proof}

Consequently, it is enough to compute acyclic polynomials of $2$-connected graphs in order to recover the acyclic polynomials of all graphs.

\subsubsection{Simplicial nodes, chordal graphs, complete graphs}

A graph $G$ is \defn{chordal} if all its induced cycles have length $3$.
On the other hand, a \defn{perfect elimination ordering} of a graph $G = (\V, E)$ with $|\V| = n$ is a bijection $\theta : [n] \to \V$ such that for all $i\in [n]$ the neighbors of $\theta(i)$ in $\{\theta(1), \dots, \theta(i-1)\}$ form a clique in $G$ (\ie are pairwise connected by an edge in $G$). We say that $\theta(i)$ is \defn{simplicial} in $G[\{\theta(1), \dots, \theta(i-1)\}]$. 
It is well-known, see \cite{Rose1970-EliminationProcess} that $G$ is chordal if and only if $G$ admits a perfect elimination ordering.

\begin{remark}\label{rem:hyperplane_setting}
For a graph $G = (\V, E)$ with an orientation $D$, the coefficients of the acyclic polynomial $\Psi(D; t)$ are the Whitney numbers of the second kind of the graphic hyperplane arrangement $\c H_G \eqdef \bigl(\{\b x\in \R^{\V} ~;~ x_i = x_j\} ~;~ ij\in E\bigr)$ with respect to the base chamber corresponding to $D$, \ie $\pol[C]_D \eqdef \{x_i \leq x_j ~;~ ij\in D\}$.
In particular, the graphic hyperplane arrangement of chordal graphs are known to be supersolvable.
In this context, the formula of \Cref{prop:AcyclicPol:ChordalGraph} is already known, see \cite[Thm~4.4 \& Cor.~4.5]{bjorneredelmanzieglerlattices}.

Moreover, Edelman and Reiner \cite{edelman-reiner} proved that the sequence of exponents of the graphic arrangement associated to a chordal graph is precisely the elimination-degree sequence defined thereafter.
As the sequence of exponents is an invariant of the hyperplane arrangement, the independence from a specific base chamber (\ie from a specific perfect elimination ordering) follows automatically, see \cite[Lemma~3.4]{edelman-reiner}.

Accordingly, we do not pretend that \Cref{prop:AcyclicPol:ChordalGraph,prop:AllPOEhaveSameDegreeSeq,cor:AcyclicPol:AllPEOyieldSamePsi} are new, but we give graph-theoretical proofs below for the sake of self-containment, and to showcase how to compute acyclic polynomials.
Besides, \Cref{lem:AcyclicPol:SimplicialNodeAdd,thm:AcyclicPol:LogConcaveCoeff:AddedSimplicialNode,cor:AcyclicPol:LogConcaveCoeff:ChordalGraph} seem not to appear in the existing literature, although we believe that other authors may have already thought about it.
Finally, \Cref{thm:ChordalGraphsAreGammaAlternating} is new, as far as we are aware.
\end{remark}

In order to prove \Cref{prop:AcyclicPol:ChordalGraph}, we first state the following more general lemma.
A node is called \defn{simplicial} if its neighbors form a clique.

\begin{lemma}\label{lem:AcyclicPol:SimplicialNodeAdd}
Let $G = (\V, E)$ be a graph with a simplicial node $v\in \V$ with $b$ neighbors. 
Let $D$ be an orientation of $G$ where all edges $uv$ are directed from $u$ to $v$, and let $D' = D\ssm\{v\}$.
Then:
$$\Psi(D; t) = \Psi(D'; t) \cdot (1 + t + \dots + t^b) = \Psi(D'; t) \cdot \frac{1 - t^{b+1}}{1 - t}\,.$$

In particular $\psi(G) = \Psi(D; 1) = \psi(G\ssm\{v\}) \cdot (b+1)$, and if $b$ is odd then $\sigma(D) = \Psi(D; -1) = 0$, if $b$ is even, then $\sigma(D) = \sigma(D')$.
\end{lemma}

\begin{proof}
The rightmost equality is immediate.
The values of $\psi(G)$ and $\sigma(D)$ come from direct evaluations.

% The proof of the leftmost equality is by induction.
% If $G$ has $2$ nodes sharing an edge, then its acyclic polynomial is $1 + t$ for all its orientations, and we have $b_1 = 0$ and $b_2 = 1$, which gives the equality.

To ease notation, let $G' \eqdef G\ssm \{v\}$ and $U \subseteq \V$ be the set of neighbors of $v$ in $G$, and recall that $b = |U|$.
For an acyclic orientation $A$ of $G$, let $A' = A \ssm\{v\}$ and let $j$ be the number of neighbors $u$ of $v$ with $u\to v$ in $A$.
The acyclic orientation $A$ is fully determined by $(A', j)$ because since $A'$ defines a total order on $U$ (as $U$ induces a clique in $G$) the acyclicity of $A$ implies that if $u\to v \in A$ for some $u \in U$, then $u'\to v\in A$ for all $u'\in U$ with $u'\to u\in A'$.
Hence the sole number of arcs pointing towards $v$ fully determines these arcs.
Reciprocally, for all acyclic orientations $A'$ of $G'$ and all $j\in\{0, \dots, b\}$, one can construct an acyclic orientation $A$ of $G$ by $A\ssm\{v\} = A'$ and orienting $u\to v$ for $u\in U$ among the $j$ first nodes of $U$ according to the total order of $U$ imposed by $A'$, and $v \to u$ for all other $u\in U$.
Consequently, $A \mapsto (A', j)$ is a bijection between acyclic orientations of $G$ and pairs of acyclic orientations of $G'$ and $j\in\{0, \dots, b_n\}$.
Finally, we have $\rank_D A = \rank_{D'} A' + (b - j)$ since $(b - j)$ of the edges $uv$ for $u\in U$ are not oriented from $u$ to $v$ in $A$.
Thus:
$$\Psi(D; t) = \sum_{A \in \AO} t^{\rank_D A} = \sum_{j=0}^{b}\sum_{A'\in \AO[G']} t^{\rank_{D'} A'} \cdot t^{b - j} = \Psi(D'; t)\cdot (1 + t + \dots + t^b) \,. ~\qedhere$$
\end{proof}

For a chordal graph $G = (\V, E)$ and a perfect elimination ordering $\theta$, we denote $\mathdefn{D_\theta}$ the orientation of $G$ induced by $\theta$, that is the orientation with arcs from $\theta(i)$ to $\theta(j)$ for all $i < j$ such that $\theta(i)\theta(j)\in E$.
Besides, we define the \defn{$i^{\text{th}}$ elimination degree} as the size of the clique of previous neighbors of $\theta(i)$, that is $\mathdefn{b_i} \eqdef \bigl|\{j\in [i-1] ~;~ \theta(i)\theta(j)\in E\}\bigr|$, where by convention $b_1 = 0$.
We call the sequence $\bigr(b_1, \dots, b_n\bigr)$ the \defn{elimination-degree sequence}.
Note that if $G$ is connected, then $b_2 = 1$.

\begin{proposition}\label{prop:AcyclicPol:ChordalGraph}
Let $G$ be a connected chordal graph with a perfect elimination ordering $\theta$, and let $D_\theta$ be its induced orientation, and $\bigl(b_1, \dots, b_n\bigr)$ its elimination-degree sequence.
Then:
$$\Psi(D_\theta; t) 
= \prod_{i=1}^n \bigl(1 + t + \dots +t^{b_i}\bigr)
= \prod_{i=1}^n \frac{1 - t^{b_i + 1}}{1 - t} \,.$$

In particular, $\psi(G) = \Psi(D_\theta; 1) = \prod_{i=1}^n (b_i+1)$ and $\sigma(D_\theta) = \Psi(D_\theta; -1) = 0$.
\end{proposition}

\begin{proof}
Inductive application of \Cref{lem:AcyclicPol:SimplicialNodeAdd} on the sequence of induced subgraphs on $\{\theta(1), \dots, \theta(i)\}$ in which $\theta(i)$ is a simplicial node with $b_i$ neighbors.
For $\sigma(D_\theta)$, note that $b_2 = 1$ is odd.
\end{proof}

\begin{example}\label{exm:AcyclicPol:CompleteGraphs}
In particular, for the complete graph $K_n$, which is a chordal graph, the acyclic orientations are in bijection with the permutations of $[n]$.
Fixing $D$ to be any acyclic orientation of $K_n$, the number of acyclic orientations of $K_n$ of a given rank $k$ is the number of permutations on $n$ elements with $k$ inversions.
These are counted by the Mahonian numbers $T_{n, k}$ \OEIS{A008302}.
In particular, \Cref{prop:AcyclicPol:ChordalGraph} helps us to retrieve the famous formula:
$$\Psi(D; t) = \sum_{k=0}^{n(n-1)/2} T_{n, k}\,t^k = \prod_{j = 1}^n \frac{1-t^j}{1-t}\,.$$
We recover that the number of acyclic orientations of $K_n$ is $\psi(K_n) = \Psi(D; 1) = n!$, while its signature is $\sigma_D(K_n) = \Psi(D; -1) = 0$ for $n\geq 2$.
\end{example}

One may wonder what happens when we change the chosen perfect elimination ordering of a given graph $G$.
Luckily, the acyclic polynomial remains unchanged.

\begin{proposition}\label{prop:AllPOEhaveSameDegreeSeq}
Let $G$ be a connected chordal graph.
Up to a permutation, the elimination-degree sequence does not depend on the perfect elimination ordering.
\end{proposition}

\begin{proof}
For a given $G = (\V, E)$, consider the graph $\c P(G)$ whose nodes are the perfect elimination orderings of $G$ and where two perfect elimination orderings $\theta$ and $\theta'$ are linked by an edge if $\theta' = \theta \circ (j~\,~ j\!+\!1)$ where $(j~\,~ j\!+\!1)$ denotes a transposition of $[n]$ (\ie $\c P(G)$ is the subgraph of the permutahedral graph induced on the perfect elimination orderings of $G$).

According to \cite{ChandranIbarraRskeySawada03-PerfectEliminationOrderings}, $\c P(G)$ is connected:
the authors even prove that $\c P(G)$ is a Hamiltonian graph by expliciting a Gray code.
Hence, it is enough to prove that if $\theta$ and $\theta'$ are linked by an edge in $\c P(G)$, then their associated elimination degree sequences are the same up to permutation.
To this end, fix a perfect elimination ordering $\theta$ of $G$ and $j\in [n-1]$ such that $\theta' = \theta \circ (j~\,~ j\!+\!1)$ is also a perfect elimination ordering.
We denote $b_i$ (resp. $b'_i$) the $i^{\text{th}}$ elimination degree of $\theta$ (resp. of $\theta'$).
If $\theta(j)\theta(j+1)\notin E$, then $b_i = b'_i$ for all $i\notin\{j, j+1\}$, and $b'_j = b_{j+1}$ and $b'_{j+1} = b_j$ since swapping the order of $\theta(j)$ and $\theta(j+1)$ does not change the clique of previous nodes of any node.
If $\theta(j)\theta(j+1)\in E$, then $b_i = b'_i$ for all $i\in[n]$ since swapping the order of $\theta(j)$ and $\theta(j+1)$ only changes the cliques of previous nodes of $\theta(j)$ (whose size increases by $1$) and of $\theta(j+1)$ (whose size decreases by $1$) and their sizes have a difference of $1$.
In both cases, $b'$ is a permutation of $b$, which proves the claim.
\end{proof}

\begin{corollary}\label{cor:AcyclicPol:AllPEOyieldSamePsi}
Let $G$ be a connected chordal graph with two perfect elimination orderings $\theta$ and $\theta'$.
Then $\Psi(D_\theta; t) = \Psi(D_{\theta'}; t)$.
\end{corollary}

\begin{proof}
Direct application of \Cref{prop:AcyclicPol:ChordalGraph,prop:AllPOEhaveSameDegreeSeq}.
\end{proof}

\begin{remark}
For graphic hyperplane arrangements, it is known that the characteristic polynomial of the arrangement coincides with the chromatic polynomial of the graph, see \cite[Section~2.4]{OrlikTerao}.
Thus, we know that $\chi(G, x) = \prod_{i = 1}^n (x - b_i)$ according to the Factorization Theorem \cite[Theorem 4.6.21]{OrlikTerao}. In particular, this again shows that the chromatic polynomial and the acyclic polynomials (of perfect elimination orientations) are two different polynomials.
\end{remark}

We can further explore the distribution of the coefficients of the polynomial $\Psi(D_{\theta}; t)$ for $\theta$ a perfect elimination ordering of a chordal graph $G$.

\medskip

A sequence of positive integers $\b a = (a_1, \dots, a_s)$ is \defn{log-concave} if for all $i\in [s-1]$ we have $a_{i+1}^2 \geq a_i\cdot a_{i+2}$.
Such a sequence is \defn{unimodal} if there exists $m\in[s]$, called the \defn{mode} of $\b a$, such that $(a_1, \dots, a_m)$ is increasing, while $(a_m, \dots, a_s)$ is decreasing.
A notorious property of log-concave sequences is that they are unimodal.
Furthermore, if the polynomial $\sum_{i=1}^s a_i t^i$ is \defn{real-rooted} (\ie does not have complex non-real roots), then $\b a$ is log-concave, and hence unimodal.
We refer to Br\"and\'en's handbook \cite{Branden2015-UnimodalityLogConcavityRealRootednessGammaPositivity} for all the ``well-known'' facts on unimodality, log-concavity, real-rootedness, and, later, $\b \gamma$-vectors.

For chordal graphs endowed with the orientation induced by a perfect elimination ordering, the coefficients of the acyclic polynomial form a log-concave sequence.
%A more general statement is:

\begin{theorem}\label{thm:AcyclicPol:LogConcaveCoeff:AddedSimplicialNode}
Let $G = (\V, E)$ be a graph with a simplicial node $v\in \V$.
Let $D$ be an orientation of $G$ where all edges $uv$ are directed from $u$ to $v$, and let $D' = D\ssm\{v\}$.
If the coefficients of the polynomial $\Psi(D'; t)$ form a log-concave (hence unimodal) sequence, then so do the coefficients of the polynomial $\Psi(D; t)$.
\end{theorem}

\begin{proof}
Recall that if the coefficients of both polynomials $p(t)$ and $q(t)$ form log-concave sequences, then the coefficients of the polynomial $p(t)\cdot q(t)$ also form a log-concave sequence.
Note that, for all $b$, the coefficients of $1 + t + \dots + t^b$ are all $1$, hence form a log-concave sequence.
Consequently, by \Cref{lem:AcyclicPol:SimplicialNodeAdd}, if the coefficients of $\Psi(D'; t)$ form a log-concave sequence, then the coefficients of the product $\Psi(D'; t)\cdot (1 + t + \dots +t^b)$ form a log-concave sequence, whatever the value of $b$.
\end{proof}

\begin{corollary}\label{cor:AcyclicPol:LogConcaveCoeff:ChordalGraph}
Let $G$ be a connected chordal graph with a perfect elimination ordering $\theta$.
Then the coefficients of the polynomial $\Psi(D_\theta; t)$ form a log-concave (hence unimodal) sequence.
\end{corollary}

\begin{proof}
Inductive application of \Cref{thm:AcyclicPol:LogConcaveCoeff:AddedSimplicialNode} on the induced subgraphs on $\{\theta(1), \dots, \theta(i)\}$.
For $i = 2$, this induced subgraph is an edge whose acyclic polynomial $1 + t$ has a log-concave sequence of coefficients.
\end{proof}

\begin{remark}

The Mahonian numbers $T_{n, k}$ seen in \Cref{exm:AcyclicPol:CompleteGraphs} are log-concave (in $k$) according to \Cref{cor:AcyclicPol:LogConcaveCoeff:ChordalGraph} (see \cite{GhemitAhmia-MahonianNumbersLogConcave} for more details), even though the polynomial $\sum_{k=0}^{n(n-1)/2} T_{n, k}\,t^k = \prod_{j = 1}^n \frac{1-t^j}{1-t}$ is obviously not real-rooted.

However, note that the coefficients of $\Psi(D; t)$ given by $\circled{6}$ of \Cref{fig:AcyclicPolyK23}, where $G$ is the complete bipartite graph $K_{2, 3}$, do not form a log-concave (not even a unimodal) sequence.
This motivates the following question.
\end{remark}

\begin{question}
Classify the graphs $G$ and orientations $D$ such that $\Psi(D; t)$ is real-rooted, and the ones whose coefficients form a log-concave sequence, or a unimodal sequence.
\end{question}

\begin{definition}
If $\b a = (a_0, \dots, a_d)$ is a palindromic sequence of length $d$, then one can rewrite the associated polynomial in the so-called \defn{$\gamma$-basis}, that is:
$$\sum_{k = 0}^d a_k t^k = \sum_{j = 0}^{\left\lfloor \frac{d}{2} \right\rfloor} \gamma_j t^j (1+t)^{d-2j} \,,$$
where $\mathdefn{\b\gamma} \eqdef (\gamma_0, \dots, \gamma_{\left\lfloor \frac{d}{2} \right\rfloor})$ is called the \defn{$\gamma$-vector} associated to $\b a$.
\end{definition}

It is well-known that if $\b a$ is \defn{$\gamma$-positive}, that is if $\gamma_j \geq 0$ for all $j \in \{0, \dots, \left\lfloor \frac{d}{2} \right\rfloor\}$, then $\b a$ is unimodal.
Besides, if $\sum_{k = 0}^d a_k t^k$ is real-rooted, then $\b a$ is $\gamma$-positive.

The coefficients of $\Psi(D_{\theta}; t)$ are palindromic for a chordal graph $G$ with a perfect elimination ordering $\theta$, as can be seen from \Cref{prop:AcyclicPol:ChordalGraph}, and we will prove that this actually holds for any directed graph in \Cref{cor:PsiIsPalindromic}. 
Hence, one can wonder whether the coefficients of $\Psi(D; t)$ form a $\gamma$-positive sequence for some digraph $D$.
We give a negative but surprising answer in the case of chordal graphs:
We say that a sequence $\b a$ is \defn{$\gamma$-alternating} if the sign of $(-1)^j\cdot\gamma_j$ does not depend on $j\in \{0, \dots, \left\lfloor \frac{d}{2} \right\rfloor\}$.

\begin{theorem}\label{thm:ChordalGraphsAreGammaAlternating}
Let $G$ be a connected chordal graph with a perfect elimination ordering $\theta$.
Then the sequence of coefficients of the polynomial $\Psi(D_\theta; t)$ is $\gamma$-alternating.
\end{theorem}

\begin{proof}
We first prove that the $\gamma$-coefficients of the sequence $(1, 1, \dots, 1)$ of length $b$ is $\gamma_j = (-1)^j \cdot \binom{b-j}{j}$ and that products of $\gamma$-alternating sequences are $\gamma$-alternating. Then we conclude the proof using \Cref{prop:AcyclicPol:ChordalGraph}. 

Let $F_n(s, x)$ be the $n^{\text{th}}$ bivariate Fibonacci polynomial, that is $F(s, x) = s\cdot F_{n-1}(s, x) + x\cdot F_{n-2}(s, x)$, and $F_0(s, x) = 0$, $F_1(s, x) = 1$.
Using the Binet formula, we get $F_n(s, x) = \frac{\alpha(s, x)^n - \beta(s, x)^n}{\alpha(s, x) - \beta(s, x)}$ where $\alpha(s, x)$ and $\beta(s, x)$ are the roots of $Z^2 - sZ - x$.
The usual ``sum of diagonals'' argument on the Fibonacci numbers yields
$F_{b+1}(s, x) = \sum_{j=0}^{\lfloor b/2\rfloor}\binom{b-j}{j} s^{b-2j}x^j$.
In particular
$$\sum_{j=0}^{\lfloor b/2\rfloor} (-1)^j\binom{b-j}{j} t^j (1+t)^{b-2j}
= (1+t)^b \sum_{j=0}^{\lfloor b/2\rfloor} \binom{b-j}{j} 1^{b-2j} \left(\frac{-t}{(1+t)^2}\right)^j
= (1+t)^b F_{b+1}\left(1, \frac{-t}{(1+t)^2}\right).$$

Using $\alpha(1, -x) = \frac{1 + \sqrt{1-4x}}{2}$ and $\beta(1, -x) = \frac{1 - \sqrt{1-4x}}{2}$, and evaluating for $x = \frac{t}{(1+t)^2}$, we get (after several simplifications):
$(1+t)^b F_{b+1}\left(1, \frac{-t}{(1+t)^2}\right) = \frac{1 - t^{b+1}}{1 - t} = 1+t+t^2+\dots+t^b$.
Hence, $\gamma_j = (-1)^j \cdot \binom{b-j}{j}$ is indeed the $\gamma$-vector of the sequence $(1, 1, \dots, 1)$ of length $b$. This sequence is thus $\gamma$-alternating.

Besides, if $(-1)^j \gamma_j \geq 0$ and $(-1)^k \delta_k \geq 0$ for all $j, k\geq 0$, then
$$\left(\sum_j \gamma_j t^j (1+t)^{b-2j}\right)\cdot\left(\sum_k \gamma_k t^k (1+t)^{c-2k}\right) = \sum_{p} \left(\sum_{j = 0}^p\gamma_j\cdot\delta_{p-j}\right) t^p (1+t)^{b+c-2p} \,,$$
where for all $j\in\{0, \dots, p\}$, $(-1)^p \gamma_j\cdot\delta_{p-j} \geq 0$, so $(-1)^p\cdot \left(\sum_{j=0}^p \gamma_j\cdot\delta_{p-j}\right) \geq 0$.
Hence if two sequences are $\gamma$-alternating, then their product is also $\gamma$-alternating.

By \Cref{prop:AcyclicPol:ChordalGraph}, the sequence of coefficients of $\Psi(D_{\theta}; t)$ is a product of $\gamma$-alternating sequences:
it is $\gamma$-alternating.
\end{proof}

\begin{remark}
Note that we are not the first to describe a $\gamma$-alternating behavior:
in \cite[Corollary~7.10]{DAliVenturello-KozulGorensteinAlgebrasAndAlternatingGammaVectors}, the authors show that for a $(d-1)$-dimensional Cohen--Macaulay sphere $\Delta$ with positive $h$-vectors, the ring $R_\Delta$ they define is $\gamma$-alternating.
\end{remark}

% Since we will prove in \Cref{cor:PsiIsPalindromic} that $\Psi(D; t)$ is always palindromic of length $|E|$, we can ask the following question:

\begin{question}
Classify the graphs $G$ and orientations $D$ such that the coefficients of $\Psi(D; t)$ form a $\gamma$-positive or a $\gamma$-alternating sequence.
\end{question}

\subsubsection{Decomposition of the acyclic polynomial and reversion of the direction of all arcs}

In order to propose some deletion-contraction formulas, we will need to split the acyclic polynomial according to the following notion:

\begin{definition}\label{def:Direction}
For two distinct \nodes~ $u, v\in \V$ in a graph $G = (\V, E)$, and an acyclic orientation $A$ of $G$, we define the \defn{direction} between $u$ and $v$ in $A$ as:
$$\mathdefn{A(u, v)} \eqdef \left\{\begin{array}{cl}
    +1 & \text{ if there exists a directed path from } u \text{ to } v \text{ in } A, \\
    -1 & \text{ if there exists a directed path from } v \text{ to } u \text{ in } A, \\
    0 & \text{ otherwise }
\end{array}\right.$$

We allow such a directed path to be composed of a single edge.
Besides: $A(v, u) = -A(u, v)$.
For a sequence of acyclic orientations $\b A = (A_1, \dots, A_r)$, let $\mathdefn{\b A(u, v)} \eqdef \bigl(A_1(u, v), \dots, A_r(u, v)\bigr)$.
\end{definition}

Let $G = H \sqcup H'$ be a bipartition of the edges of the graph $G$.
Fix an orientation of $G$, say $D\sqcup D'$ where $D$ is an orientation of $H$, and $D'$ is an orientation of $H'$.
If $D\sqcup D'$ is acyclic, then both $D$ and $D'$ are acyclic.
Moreover, if $D \sqcup D'$ is acyclic, then for any nodes $u, v$ of $G$, we have $D(u, v)\cdot D'(u, v) \ne -1$.
The reciprocal is straightforward:

\begin{lemma}\label{lem:BipartitionEdgesYieldsCartesianProductAO}
Let $G = H\sqcup H'$ be a bipartition of the edges of the graph $G$.
The map $(A, A') \mapsto A\sqcup A'$ is a bijection from the pairs of acyclic orientations of $H$ and $H'$ such that $A(u, v)\cdot A'(u, v)\ne -1$ for all nodes $u, v$ of $G$, and the acyclic orientations of $G$.

Moreover, $\AO$ is the subgraph of the Cartesian product $\AO[H]\times \AO[H']$ induced on these pairs of acyclic orientations, that is $A\sqcup A'$ and $B\sqcup B'$ are adjacent in $\AO$ if and only if either $A = B$, and $A'$ and $B'$ are adjacent in $\AO[H']$, or $A' = B'$, and $A$ and $B$ are adjacent in $\AO[H]$.
\begin{comment}
\begin{compactitem} 
\item $A = B$, and $A'$ and $B'$ are adjacent in $\AO[H']$, or
\item $A' = B'$, and $A$ and $B$ are adjacent in $\AO[H]$.
\end{compactitem}
\end{comment}
\end{lemma}

\begin{proof}
Clearly, $(A, A')\mapsto A\sqcup A'$ is an injection, since this is just a standard operation between sets.
Moreover, $A\sqcup A'$ is acyclic if and only if for all $u, v\in \V$ there is no path from $u$ to $v$ and from $v$ to $u$ in $A\sqcup A'$, which is equivalent to: for all $u,v \in \V$ there is no path from $u$ to $v$ and from $v$ to $u$ both in $A$ (\ie $A$ is acyclic), and there is no path from $u$ to $v$ and from $v$ to $u$ both in $A'$ (\ie $A'$ is acyclic), and there is no path from $u$ to $v$ and from $v$ to $u$, one in $A$ and the other in $A'$ (\ie $A(u, v)\cdot A'(u, v) \ne -1$).
This proves the bijection.

Besides, $A\sqcup A'$ and $B\sqcup B'$ are adjacent if and only if they differ on exactly 1 edge: due to the bipartition of the edges of $G$, this edge is either in $H$ or in $H'$ not in both.
This proves the adjacency relation.
\end{proof}

% \begin{definition}
% For a graph $G$ and two nodes $u, v\in \V$, we denote by $\mathdefn{\AO^{+1}_{u, v}}$ (resp. $\mathdefn{\AO^{-1}_{u, v}}$, resp. $\mathdefn{\AO^0_{u, v}}$) the set of acyclic orientations of $G$ such that the direction between $u$ and $v$ is $+1$ (reps. is $-1$, resp. is $0$).
% \end{definition}

Some fundamental properties of the acyclic polynomial come from understanding what happens when we reverse the direction of all the arcs of an orientation at once.
We will often use the following straightforward lemma:

\begin{lemma}\label{lem:CompletementIsInvolution}
Fix a (di)graph $G = (\V, E)$, any $u, v\in \V$, and any orientation $D$ of $G$.
The map $X \mapsto E\ssm X$ satisfies $\rank_D D^{E\ssm X} = |E| - \rank_D D^X$, and it is:
\begin{compactenum}
\item an involution without fixed point on $\c O(G)$, and\label{item:CompletementIsInvolution:OnOrientations}
\item an involution without fixed point on $\AO$, and\label{item:CompletementIsInvolution:OnAcyclic}
\item an involution without fixed point on $\{A\in \AO ~;~ A(u, v) = 0\}$, and\label{item:CompletementIsInvolution:On0Acyclic}
\item a bijection between $\{A\in \AO ~;~ A(u, v) = +1\}$ and $\{A\in \AO ~;~ A(u, v) = -1\}$, and\label{item:CompletementIsInvolution:BetweenNon0Acyclic}
\item if $|E|$ is even, an involution without fixed point on $\{A \in \AO ~;~ \rank_D A \text{ is even}\}$, and an involution without fixed point on $\{A \in \AO ~;~ \rank_D A \text{ is odd}\}$, and\label{item:CompletementIsInvolution:InsideEvenRank}
\item if $|E|$ is odd, a bijection between the set $\{A \in \AO ~;~ \rank_D A \text{ is even}\}$ and the set\linebreak$\{A \in \AO ~;~ \rank_D A \text{ is odd}\}$.\label{item:CompletementIsInvolution:EvenOddRank}
\end{compactenum}
\end{lemma}

\begin{proof}
It is immediate that $X\mapsto E\ssm X$ is an involution of $2^E$.
For~\eqref{item:CompletementIsInvolution:OnOrientations}, note that $\c O(G)$ is isomorphic to $2^E$.
For~\eqref{item:CompletementIsInvolution:OnAcyclic}, note that reversing the directions of all the arcs at once does not change whether the orientation is acyclic or not.
For~\eqref{item:CompletementIsInvolution:On0Acyclic} and~\eqref{item:CompletementIsInvolution:BetweenNon0Acyclic}, note that reversing the directions of all the arcs at once changes a directed path from $u$ to $v$ into a directed path from $v$ to $u$, \ie $D^{E\ssm X}(u, v) = - D^X(u, v)$ for all orientations $D$ of $G$.

The claimed rank equality is by definition:
$\rank_D D^{E\ssm X} = |E\ssm X| = |E| - |X| = |E| - \rank_D D^X$.
This immediately implies~\eqref{item:CompletementIsInvolution:InsideEvenRank} and~\eqref{item:CompletementIsInvolution:EvenOddRank}.
\end{proof}

Using the fundamental involution, we can re-prove one of the powerful results of \cite{SavageSquireWest93-GrayCodesAcyclicOrientations} that they repeatedly use in order to show that some graphs are not {\good}.
Our proof is essentially the same as theirs, but written with our language.

\begin{theorem}[{\cite[Lemma~3.1]{SavageSquireWest93-GrayCodesAcyclicOrientations}}]\label{thm:DivisibilityBy4}
If a graph with an even number of edges is {\good}, then its number of acyclic orientations is divisible by $4$.
\end{theorem}

\begin{proof}
Let $G$ be such a graph, and $D$ some orientation of $G$. Since $|E|$ is even, \Cref{lem:CompletementIsInvolution}~\eqref{item:CompletementIsInvolution:InsideEvenRank} shows that $X\mapsto E\ssm X$ induces an involution without fixed point on $\{A \in \AO ~;~ \rank_D A \text{ is even}\}$.
This implies that the latter set is of even size.

By \Cref{lem:ParityArgument}, we have $\sigma(D) = \Psi(D; -1) = 0$.
So writing $\Psi(D; t) = \sum_{k = 0}^d p_k t^k$ we get that $\sum_{k \text{ even}} p_k = \sum_{k\text{ odd}} p_k$.
Since $\sum_{k\text{ even}} p_k = \bigl|\{A \in \AO ~;~ \rank_D A \text{ is even}\}\bigr|$ is even, we finally get that $\psi(G) = \Psi(D; 1) = \sum_k p_k$ is divisible by $4$.
\end{proof}

Thanks to the notion of direction and the fundamental involution, we can isolate two sub-polynomials hidden inside the acyclic polynomial.
The core idea is to split the edges of $G$ between two subgraphs $H$ and $H'$ which meet on exactly two nodes $u$ and $v$:
in an acyclic orientation of $G$, either there is no directed path between $u$ and $v$ inside the induced orientation of $H$, or there is a directed path from $u$ to $v$, or from $v$ to $u$.
The set of possible acyclic orientations of $H'$ depends only on which of these three cases we are in, not on the exact acyclic orientation of $H$.
In \Cref{fig:AcyclicPolyK23}, we illustrate the particular case where $H'$ is the subgraph of $K_{2, 3}$ with nodes $u, v$ and no edge.

\begin{proposition}\label{prop:AcyclicPolyDecomposition}
Let $G = H\sqcup H'$ be a bipartition of the edges of $G$ such that $\V(H) \cap \V(H') = \{u, v\}$ for two nodes $u, v$.
Then:
$\Psi(D\sqcup D'; t) = \Psi^{u\not\leftrightarrow v}(D\sqcup D'; t) + \Psi^{u\to v}(D\sqcup D'; t) + t^{|E|}\cdot \Psi^{u\to v}(D\sqcup D'; \tfrac{1}{t})$, where:
$$\mathdefn{\Psi^{u\not\leftrightarrow v}(D\sqcup D'; t)} \eqdef \sum_{\substack{A\sqcup A'\in \AO[G]\\A(u, v) = 0}} t^{\rank_{D\sqcup D'} A\sqcup A'} \hspace{0.5cm}\text{ and }\hspace{0.5cm}
\mathdefn{\Psi^{u\to v}(D\sqcup D'; t)} \eqdef \sum_{\substack{A\sqcup A'\in \AO[G]\\A(u, v) = 1}} t^{\rank_{D\sqcup D'} A\sqcup A'}$$
\end{proposition}

% \begin{proposition}\label{prop:AcyclicPolyDecomposition}
% For any $u, v\in \V$: $\Psi(D; t) = \Psi^{u\not\leftrightarrow v}(D; t) + \Psi^{u\to v}(D; t) + t^{|E|}\cdot \Psi^{u\to v}(D; \tfrac{1}{t})$, where:
% $$\mathdefn{\Psi^{u\not\leftrightarrow v}(D; t)} \eqdef \sum_{\substack{A\in \AO[D] \\ A(u, v) = 0}} t^{\rank_D A} \hspace{1cm}\text{ and }\hspace{1cm}
% \mathdefn{\Psi^{u\to v}(D; t)} \eqdef \sum_{\substack{A\in \AO[D] \\ A(u, v) = 1}} t^{\rank_D A}$$
% \end{proposition}

\begin{proof}
Clearly, for any $D\sqcup D'$, and $t$, we have:
$$\Psi(D\sqcup D'; t) = \sum_{\substack{A\sqcup A'\in \AO[G]\\ A(u, v) = 0}} t^{\rank_{D\sqcup D'} A\sqcup A'} ~+~\sum_{\substack{A\sqcup A'\in \AO[G]\\ A(u, v) = 1}} t^{\rank_{D\sqcup D'} A\sqcup A'} ~+~ \sum_{\substack{A\sqcup A'\in \AO[G]\\ A(u, v) = -1}} t^{\rank_{D\sqcup D'} A\sqcup A'}$$

Thus, we only need to prove that $\sum_{\substack{A\sqcup A'\in \AO[G]\\ A(u, v) = -1}} t^{\rank_{D\sqcup D'} A\sqcup A'} = t^{|E|}\cdot\Psi^{u\to v}(D\sqcup D'; \frac{1}{t})$.
Using the bijection $X \mapsto E\ssm X$ on the graph $G$ from \Cref{lem:CompletementIsInvolution}~\eqref{item:CompletementIsInvolution:BetweenNon0Acyclic}, we get:
$$\begin{array}{rcl}
\sum_{\substack{A\sqcup A'\in \AO[G]\\ A(u, v) = -1}} t^{\rank_{D\sqcup D'} A\sqcup A'}
&=& \sum_{\substack{A\sqcup A'\in \AO[G]\\ A(u, v) = +1}} t^{|E| - \rank_{D\sqcup D'} A\sqcup A'} \vspace{.25cm} \\
&=& t^{|E|}\cdot \sum_{\substack{A\sqcup A'\in \AO[G]\\ A(u, v) = +1}} \left(\frac{1}{t}\right)^{\rank_{D\sqcup D'} A\sqcup A'} \vspace{.25cm} \\
&=& t^{|E|}\cdot\Psi^{u\to v}(D\sqcup D'; \tfrac{1}{t}) ~ \qedhere
\end{array}$$
\end{proof}

% \begin{remark}
% By convention, $\Psi^{u\not\leftrightarrow v}(D; t) = 1$ and $\Psi^{u\to v}(D; t) = 0$ if $D$ has no edge but contains the nodes $u$ and $v$.
% \end{remark}

As a special case, we define the \defn{partial acyclic polynomials} $\mathdefn{\Psi^{u\not\leftrightarrow v}(D; t)}$ and $\mathdefn{\Psi^{u\to v}(D; t)}$ to the polynomials associated to the bipartition of edges $G = H\sqcup H'$ where $H'$ is the graph with nodes $u, v$ and no edge between them.
The polynomial $\Psi^{u\not\leftrightarrow v}(D; t)$ (resp. $\Psi^{u\to v}(D; t)$) counts the number of acyclic orientations of $G$ containing no directed path from $u$ to $v$ nor from $v$ to $u$ (resp. containing a directed path from $u$ to $v$), according to their rank with respect to the orientation $D$.

\smallskip

Recall that a polynomial $P(t) = \sum_{k = 0}^d p_k t^k$ is \defn{palindromic of length $d$} if its coefficients satisfy $p_{k} = p_{d-k}$ for all $k\in [d]$, or equivalently if $P(t) = t^{d}\cdot P(\tfrac1t)$.

\begin{corollary}\label{cor:PsiIsPalindromic}
The acyclic polynomial $\Psi(D; t)$ is palindromic of length $|E|$.

Furthermore, if $G = H\sqcup H'$ is a bipartition of the edges of $G$ with $\V(H)\cap\V(H') = \{u, v\}$, then the polynomial $\Psi^{u\not\leftrightarrow v}(D\sqcup D'; t)$ is palindromic of length $|E|$.
\end{corollary}

\begin{proof}
Using the involution $X \mapsto E\ssm X$ from \Cref{lem:CompletementIsInvolution}~\eqref{item:CompletementIsInvolution:On0Acyclic}, we get:
$$\Psi^{u\not\leftrightarrow v}(D\sqcup D'; t) = \sum_{\substack{A\sqcup A'\in \AO[G] \\ A(u, v) = 0}} t^{\rank_{D\sqcup D'} A\sqcup A'} = \sum_{\substack{A\sqcup A'\in \AO[G] \\ A(u, v) = 0}} t^{|E| - \rank_{D\sqcup D'} A\sqcup A'} = t^{|E|}\cdot \Psi^{u\not\leftrightarrow v}(D\sqcup D'; \tfrac{1}{t})$$

Hence, the polynomial $\Psi^{u\not\leftrightarrow v}(D; t)$ is palindromic.
Consequently, by \Cref{prop:AcyclicPolyDecomposition}, $\Psi(D; t)$ is also palindromic, as it is the sum of a palindromic polynomial $\Psi^{u\not\leftrightarrow v}(D; t)$ with another palindromic polynomial $\Psi^{u\to v}(D; t) + t^{|E|}\cdot \Psi^{u\to v}(D; \tfrac{1}{t})$.
\end{proof}

\begin{remark}
Note that $\Psi(D; t)$ is not necessarily of degree $|E|$ (nor does it need to have a non-zero constant term), especially when $D$ is not acyclic, see the example of cycles $C_n$ in \Cref{exm:AcyclicPolynomialCycles}, or of $K_{2, 3}$ in \Cref{exm:AcyclicPolynomialK23,fig:AcyclicPolyK23}.
\end{remark}

For two orientations $D$ and $D'$ of the same graph $G = (\V, E)$, it is not evident how to determine $\Psi(D'; t)$ given $\Psi(D; t)$.
However, for an orientation $D$, we can easily express the acyclic polynomial of the reverse orientation $D^E$:

\begin{corollary}\label{cor:PsiOfReverseOrientation}
For any digraph $D = (\V, E)$, we have:
$\Psi(D^E; t) = \Psi(D; t)$.

Furthermore, if $G = H\sqcup H'$ is a bipartition of the edges of $G$ with $\V(H)\cap\V(H') = \{u, v\}$, then we have $\Psi^{u\not\leftrightarrow v}((D\sqcup D')^E; t) = \Psi^{u\not\leftrightarrow v}(D\sqcup D'; t)$, and also $\Psi^{u\to v}((D\sqcup D')^E; t) = t^{|E|}\cdot \Psi^{u\to v}(D\sqcup D'; \tfrac{1}{t})$.
\end{corollary}

\begin{proof}
Note that $\rank_{D^E} D^X = |E\ssm X| = |E| - \rank_D D^X$.
Using the involution $X \mapsto E\ssm X$ from \Cref{lem:CompletementIsInvolution}~\eqref{item:CompletementIsInvolution:OnAcyclic}, we get
$$\Psi(D^E; t) = \sum_{A\in \AO[D^E]} t^{\rank_{D^E} A} = \sum_{A\in \AO[D]} t^{|E| -\rank_D A} = t^{|E|}\cdot \Psi(D; \tfrac{1}{t}).$$

Since $\Psi(D; t)$ is palindromic of length $|E|$ by \Cref{cor:PsiIsPalindromic}, the latter is equal to $\Psi(D; t)$.

The proof is the same for $\Psi^{u\not\leftrightarrow v}$.
From the equality on $\Psi$ and on $\Psi^{u\not\leftrightarrow v}$, we deduce the one of $\Psi^{u\to v}$ via the decomposition of \Cref{prop:AcyclicPolyDecomposition}.
\end{proof}

\begin{example}
By \Cref{cor:PsiOfReverseOrientation}, we can deduce from \Cref{fig:AcyclicPolyK23} the acyclic polynomials (and their decomposition) of the orientations obtained by reversing the direction of all the arcs in the drawn orientations.
\end{example}

\subsubsection{Automorphisms of (di)graphs and palindromicity of \texorpdfstring{$\Psi^{u\to v}(D\sqcup D'; t)$}{Psi(u to v)(D u D'; t)}}

It is straightforward to show that if two orientations $D$ and $D'$ of $G$ are isomorphic (as digraphs, \ie if there exists $\varphi:\V\to \V$ such that $xy\in D$ if and only if $\varphi(x)\varphi(y)\in D'$), then their acyclic polynomials are the same.
We can go one step further by considering automorphisms of digraphs.

\begin{definition}
For a (directed) graph $G = (\V, E)$, an \defn{automorphism} is a bijection $\varphi : \V \to \V$ such that $G$ has the (directed) edge $uv\in E$ if and only if $G$ has the (directed) edge $\varphi(u)\varphi(v)$.
\end{definition}
The set of automorphisms of a (directed) graph together with the composition of functions forms a group, which is denoted $\mathdefn{\Aut(G)}$.

\begin{theorem}
Let $G = H\sqcup H'$ be a graph with a bipartition of the edges of $G$ such that $\V(H)\cap\V(H') = \{u, v\}$, and $D\sqcup D'$ an orientation of $G$.
If the directed graph $D\sqcup D'$ admits an automorphism $\varphi$ which induces an automorphism of $D$ and such that $\varphi(u) = v$ and $\varphi(v) = u$, then $\Psi^{u\to v}(D\sqcup D'; t)$ is palindromic of length $|E|$.
\end{theorem}

\begin{proof}
Fix $G$, $D\sqcup D'$, $u, v\in \V$ and $\varphi$ as above.
For an acyclic orientation $A\sqcup A'$ of $G$, we denote $\varphi(A\sqcup A')$ the orientation whose arcs are from $\varphi(x)$ to $\varphi(y)$ for all arcs from $x$ to $y$ in $A\sqcup A'$.

First, if there is a directed path $x_0 \to x_1 \to \dots \to x_\ell$ in $A\sqcup A'$, then there is a directed path $\varphi(x_0) \to \varphi(x_1) \to \dots \to \varphi(x_\ell)$ in $\varphi(A\sqcup A')$.
Hence, since $\varphi$ is a bijection, $A\sqcup A'$ is acyclic if and only if $\varphi(A\sqcup A')$ is acyclic.
Moreover, since $\varphi(u) = v$ and $\varphi(v) = u$, there is a path from $u$ to $v$ in $A$ if and only if there is a path from $v$ to $u$ in $\varphi(A)$ because $\varphi$ is also an automorphism of $D$.
Second, let $X\subseteq E$ be such that $A\sqcup A' = (D\sqcup D')^X$, \ie $X$ is the set of edges $xy\in E$ such that we have an arc from $x$ to $y$ in $D\sqcup D'$ but an arc from $y$ to $x$ in $A\sqcup A'$.
For such $xy\in E$, we have $\varphi(x)\to \varphi(y)$ in $D\sqcup D'$ because $\varphi$ is an automorphism of $D\sqcup D'$, and $\varphi(x) \to \varphi(y)$ in $A\sqcup A'$ by definition.
Hence, if $A\sqcup A' = (D\sqcup  D')^X$, then $\varphi(A\sqcup A') = (D\sqcup D')^{\{\varphi(x)\varphi(y) ~;~ xy\in E\}}$.
In particular, we get: $\rank_{D\sqcup D'} \varphi(A\sqcup A') = \rank_{D\sqcup D'} A\sqcup A'$.

Now, we mix both $\varphi$ and the reversion of all arcs.
Note that if $A$ admits a directed path from $u$ to $v$, then $\varphi(A)$ admits a directed path from $v$ to $u$, so $\varphi(A)^E$ admits a directed path from $u$ to $v$ again.
Consequently, the map $A\sqcup A' \mapsto \varphi(A\sqcup A')^E$ defines a bijection (whose reciprocal is $A\sqcup A' \mapsto \varphi^{-1}((A\sqcup A')^E)$) of the set of acyclic orientations of $G$ with a directed path from $u$ to $v$ in $D$, such that $\rank_{D\sqcup D'} \varphi(A\sqcup A')^E = |E| - \rank_{D\sqcup D'} \varphi(A\sqcup A') = |E| - \rank_{D\sqcup D'} A\sqcup A'$.
This bijection yields the following transformation of the sum:
$$\begin{array}{rcl}
\Psi^{u\to v}(D\sqcup D'; t)
&=& \sum_{\substack{A\sqcup A'\in \AO[G] \\ A(u, v) = +1}} t^{\rank_{D\sqcup D'} A\sqcup A'} \vspace{0.25cm} 
= \sum_{\substack{B\sqcup B'\in \AO[G] \\ B(u, v) = +1}} t^{\rank_{D\sqcup D'} \varphi^{-1}(B\sqcup B')^E} \vspace{0.25cm} \\
&=& t^{|E|}\cdot \sum_{\substack{B\sqcup B'\in \AO[D] \\ B(u, v) = +1}} t^{-\rank_{D\sqcup D'} B\sqcup B'} \vspace{0.25cm} 
= t^{|E|}\cdot \Psi^{u\to v}(D\sqcup D'; \tfrac{1}{t})
\end{array}$$
This ensures that $\Psi^{u\to v}(D\sqcup D'; t)$ is palindromic of length $|E|$.
\end{proof}

\begin{example}
For the complete bipartite graph $K_{2, 3}$ depicted in \Cref{fig:AcyclicPolyK23}, the orientations $\circled{1}, \circled{3}$ and $\circled{5}$ admit the transposition between $u$ and $v$ as an automorphism of directed graph (remember that the chosen bipartition takes $H'$ to be the graph with nodes $u, v$ and no edge).
These are the only orientations of $K_{2, 3}$ to admit an automorphism that exchanges $u$ and $v$, and are indeed the only orientations for which $\Psi^{u\to v}(D; t)$ is palindromic of length $|E| = 6$.
\end{example}

\begin{remark}
One has to be careful not to confuse automorphisms of graphs and of digraphs.
For instance, let $G = (\V, E)$ be a path on $\ell$ nodes, with $\ell$ odd, denoted $x_1, \dots, x_\ell$ with $x_ix_{i+1}\in E$.
Let $u = x_1$ and $v = x_\ell$ be the endpoints of this path.
The graph $G$ admits an automorphism that exchanges $u$ and $v$, namely $\varphi : x_i \mapsto x_{\ell-i}$ for $i\in [\ell]$.
Moreover, this is the only automorphism of $G$.

Consequently, there is no orientation $D$ of $G$ which admits an automorphism that exchanges $u$ and $v$.
Indeed, such an automorphism of $D$ would induce an automorphism of $G$, so it would have to be the above $\varphi$, yet since $\ell$ is odd, there is an edge $x_{(\ell-1)/2}x_{(\ell+1)/2}$ in $G$ whose image by $\varphi$ is $x_{(\ell+1)/2}x_{(\ell-1)/2}$:
whatever how we orient this edge in $D$, its image by $\varphi$ would be the same edge but with the reverse orientation, so it would not be an arc of $D$.
\end{remark}

% \germainadd{Other stuffs to say? Surely...}

\subsubsection{Deletion-contraction phenomenon for the acyclic polynomial}\label{ssec:DeletionContractionPhenomenon}

By means of the decomposition of \Cref{prop:AcyclicPolyDecomposition}, we can present a deletion-contraction phenomenon.
It will be incomplete since it will fail to handle pairs of nodes which share a neighbor which is of degree 3 or more (see \Cref{rmk:ImpossibilityOfBetterPsi0Recursion}), but it will still enable us to compute the acyclic signature of new graphs, such as the multipaths (see \Cref{ssec:SignatureMultiPath}). We start with contraction. 

\begin{definition}\label{def:SuspendedNodes}
Let $G = H\sqcup H'$ be a graph with a bipartition of the edges of $G$ such that $\V(H)\cap\V(H') = \{u, v\}$.
We say that the nodes $u$ and $v$ are \defn{suspended over $H$} if all the common neighbors of $u$ and $v$ in $H$ are of degree $2$.
Moreover, for an orientation $D\sqcup D'$ of $G$, if $u$ and $v$ are suspended over $H$, then we denote by $x$ the number of directed paths of length $2$ from $u$ to $v$ or from $v$ to $u$ in $D$, and let $y$ be the number of paths of length $2$ of $H$ which are not directed paths in $D$, and we denote by $\Tilde{D}/_{uv}$ the contraction on $\{u, v\}$ of the orientation $D$ with the common neighbors of $u$ and $v$ deleted.
\end{definition}

\begin{example}
In \Cref{fig:AcyclicPolyK23}, the nodes $u$ and $v$ are suspended over $K_{2, 3}$, while any pair among the three other nodes is not suspended over $K_{2, 3}$.
For orientation $\circled{5}$, we have $x = 2$ and $y = 1$.
The orientation $\Tilde{D}/_{uv}$ is just a digraph with a single node, whatever the chosen orientation of $K_{2, 3}$.
\end{example}

\begin{proposition}\label{prop:Psi0Recursion}
Let $G = H\sqcup H'$ be a graph with a bipartition of the edges of $G$ such that $\V(H)\cap\V(H') = \{u, v\}$, and $D\sqcup D'$ an orientation of $G$.
Let also $\varepsilon = 1$ if $uv\in E(H)$ and $\varepsilon = 0$ if $uv\notin E(H)$.
If $u$ and $v$ do not have a common neighbor in $H$, then
$$\Psi^{u\not\leftrightarrow v}(D\sqcup D'; t)
= (1-\varepsilon) \cdot \Psi(D'; t) \cdot \Psi(D/_{uv}; t) \,.$$

More generally, if $u$ and $v$ are suspended over $H$, then with the notations of \Cref{def:SuspendedNodes}:
% there are paths $(u, w, v)$ of length $2$ between $u$ and $v$ in $H$, then let $x$ be the number of directed paths of length $2$ from $u$ to $v$ or from $v$ to $u$ in $D$, and let $y$ be the number of paths $(u, w, v)$ of $H$ which are not directed paths in $D$.
% If $u$ and $v$ do not have a common neighbor in $H$ except for nodes on a path of length $2$\leonie{weird formulation, not sure what you want, do better!}\germain{Give a name to this property like ``$u$ and $v$ are suspended over $H$'', and change accordingly (many times).}, then, denoting $\Tilde{D}$ the orientation $D$ with these common neighbors of $u$ and $v$ deleted, and , we get:
$$\Psi^{u\not\leftrightarrow v}(D\sqcup D'; t)
= (1 - \varepsilon) \cdot (2t)^x \cdot (1+t^2)^y \cdot
\Psi(D'; t) \cdot \Psi(\Tilde{D}/_{uv}; t) \,.$$
\end{proposition}

\begin{proof}
The second formula implies the first when $x = y = 0$, hence, we only prove the second.

If there is an edge between $u$ and $v$ in $H$, then all acyclic orientations $A\sqcup A'$ of $G = D\sqcup D'$ satisfy $A(u, v) = \pm 1$, so $\Psi^{u\not\leftrightarrow v}(D\sqcup D'; t) = 0$.
The right-hand side for the formula is also $0$ since $(1 - \varepsilon) = 0$ in this case.
Hence, we assume that there is no edge $uv\in E(H)$ in the following.

The set of acyclic orientations $A\sqcup A'$ of $G = D\sqcup D'$ such that $A(u, v) = 0$ is in bijection with the set of pairs $(A, A')\in \AO[H]\times\AO[H']$ such that $A(u, v) = 0$.
Consequently, since $t^{\rank_{D\sqcup D'} A\sqcup A'} = t^{\rank_{D} A}\cdot t^{\rank_{D'} A'}$:
$$\Psi^{u\not\leftrightarrow v}(D\sqcup D'; t)
= \sum_{\substack{A\sqcup A'\in \AO\\A(u, v) = 0}} t^{\rank_{D\sqcup D'} A\sqcup A'}
= \left(\sum_{A'\in \AO[D']} t^{\rank_{D'} A'}\right)\cdot\left(\sum_{\substack{A\in \AO[D]\\A(u, v) = 0}} t^{\rank_{D} A}\right)$$

The left parenthesis is $\Psi(D'; t)$.

For the right parenthesis, consider the map $A \mapsto \Tilde{A}/_{uv}$.
This maps an acyclic orientation of $D$ to an orientation of $\Tilde{D}/_{uv}$.
Moreover, if $A$ is acyclic and $A(u, v) = 0$, then $\Tilde{A}/_{uv}$ is acyclic too (because a directed path between $u$ and $v$ in $\Tilde{A}/_{uv}$ yields a directed path between $u$ and $v$ in $A$).
Besides, two orientations $A$ and $B$ of $D$ satisfy $\Tilde{A}/_{uv} = \Tilde{B}/_{uv}$ if and only if they only differ on the orientation of the paths of length $2$ between $u$ and $v$ (if there exist any).
Let $W\subseteq \V$ be the set of nodes $w$ such that there is a path $(u, w, v)$ in $H$, and let $W_x$ (resp. $W_y$) be the set of all nodes $w$ such that $(u, w, v)$ is a path in $H$ and is (resp. is not) a directed path in $D$.

Fix $w\in W$:
the orientation $A$ satisfies $A(u, v) = 0$ if and only if $w$ is either a source or a sink in $(u, w, v)$.
Hence for a given acyclic orientation $a$ of $\Tilde{H}/_{uv}$, the set of acyclic orientations $A$ of $H$ such that $\Tilde{A}/_{uv} = a$ is in bijection with partitions $I\sqcup J = W$, namely $(I, J)$ is associated to the orientation of $H$ given by the arcs of $a$ together with the orientation of $(u, w, v)$ where $w$ is a source if $w\in I$ (resp. a sink if $w\in J$).
Note that this set of orientations admits an involution $\eta$ without fixed point that exchanges the roles of $I$ and $J$.
We will sum over these possible partitions for a fixed $a\in \Tilde{D}/_{uv}$, it remains to compute $\rank_D A$ from $\rank_{\Tilde
{D}/_{uv}} a$ and $I, J$.

Let $A$ be an acyclic orientation of $H$ where $w\in W$ is either a source or a sink.

\begin{enumerate}
\item If $w\in W_x$, then $\rank_D A = \rank_D (A\ssm w) + 1$; 
\item If $w\in I$ (resp. $w\in J$) and $w$ is a sink in $D$, then $\rank_D A = \rank_D (A\ssm w) + 2$ (resp. $\rank_D A = \rank_D (A\ssm w)$); 
\item If $w\in I$ (resp. $w\in J$) and $w$ is a source in $D$, then $\rank_D A = \rank_D (A\ssm w)$ (resp. $\rank_D A = \rank_D (A\ssm w) + 2$).
\end{enumerate}

Consequently, up to using the involution $\eta$, we assume without loss of generality that all $w\in W_y$ are sinks in $D$.
Therefore, for a fixed $a \in \AO[\Tilde{D}/_{uv}]$, we get: $$\rank_D A = (\rank_{\Tilde{D}/_{uv}} a) + 1\cdot |W_x| + 0\cdot |I\cap W_y| + 2\cdot |J\cap W_y|
= (\rank_{\Tilde{D}/_{uv}} a) + x + 2\cdot|J\cap W_y|.$$
Finally, we obtain (remember that $|W_x| = x$, $|W_y| = y$ and $|W| = x + y$):
$$\begin{array}{rclr}
\sum_{\substack{A\in \AO[D]\\A(u, v) = 0}} t^{\rank_{D} A}
&=& \sum_{a\in \AO[\Tilde{D}/_{uv}]} t^{\rank_{\Tilde{D}/_{uv}} a} \cdot t^{x} \cdot \left(\sum_{J\subseteq W} t^{2|J\cap W_y|}\right) \vspace{0.2cm} & \\
&=& t^{x} \cdot \sum_{a\in \AO[\Tilde{D}/_{uv}]} t^{\rank_{\Tilde{D}/_{uv}} a} \cdot \left(\sum_{J\subseteq W_y} t^{2|J|} \cdot 2^{|W|-|W_y|}\right) \vspace{0.2cm} & \\
&=& t^{x} \cdot (1+t^2)^{|W_y|} \cdot 2^x \cdot \sum_{a\in \AO[\Tilde{D}/_{uv}]} t^{\rank_{\Tilde{D}/_{uv}} a} \vspace{0.2cm} & \\
&=& (2t)^{x} \cdot (1+t^2)^{y} \cdot \Psi(\Tilde{D}/_{uv}; t)  &\qedhere
\end{array}$$
\end{proof}

\begin{remark}
As typical use cases of \Cref{prop:Psi0Recursion}, note that if $H'= (\{u,v\}, \emptyset)$, then $\Psi(D'; t) = 1$; while if $H'$ is only the edge between $u$ and $v$, then $\Psi(D'; t) = 1 + t$.
\end{remark}

\begin{example}
In \Cref{fig:AcyclicPolyK23}, with the bipartition $K_{2, 3} = H \sqcup H'$ where $H'$ is the graph with nodes $u, v$ and no edge, the polynomial $\Psi^{u\not\leftrightarrow v}(D; t)$ can be deduced directly using the second formula of \Cref{prop:Psi0Recursion}:
for each orientation, denoting $x$ the number of directed paths between $u$ and $v$ in $D$, we get $\Psi^{u\not\leftrightarrow v}(D; t) = (2t)^x (1+t^2)^{3-x}$ whose value is either $8t^3$, or $4t^2 + 4t^4$, or $2t + 4t^3 + 2t^5$, or $1 + 3t^2 + 3t^4 + t^6$.
\end{example}

\begin{remark}\label{rmk:ImpossibilityOfBetterPsi0Recursion}
In \Cref{prop:Psi0Recursion}, the condition that $u$ and $v$ are suspended over $H$ is not artificial.
Consider the map $\delta : A \mapsto A/_{uv}$ from the acyclic orientations of $H$ to the acyclic orientations of $H/_{uv}$.
If all common neighbors of $u, v$ are of degree $2$, then the proof of \Cref{prop:Psi0Recursion} essentially relies on the fact that the fibers $\delta^{-1}(a)$ for $a\in \AO[H/_{uv}]$ are well-behaved:
they have the same size, and the rank of the pre-images can be controlled easily.
However, if there exists a common neighbor to $u, v$ which is connected to other nodes of $H$ (possibly to other common neighbors of $u$ and $v$), then nothing guarantees that the fibers of $\delta$ are well-behaved. 
%, and we do not know how to sum over them.

As illustrated in \Cref{fig:WeirdContraction}, is not even true in general that the polynomial $\Psi(\Tilde{D}/_{uv}; t)$ divides the polynomial $\Psi^{u\not\leftrightarrow v}(D\sqcup D'; t)$.
\end{remark}

\begin{figure}
    \centering
    \includegraphics[width=0.8\linewidth]{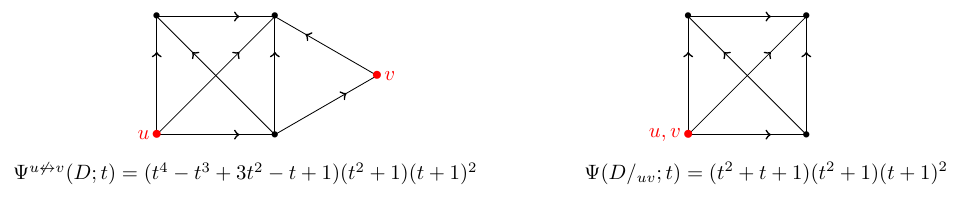}
    \caption{A digraph $D$ (left) and its contraction on $u, v$ (right), together with their respective (partial) acyclic polynomials.
    Both polynomials are factorized (over $\Z$), so the right one does not divide the left one.}
    \label{fig:WeirdContraction}
\end{figure}

We now turn our attention to deletion: in what follows, the digraphs $D$ and $D'$ should be thought as deletions from the digraph $D\sqcup D'$.
The result is easier to understand than the contraction of \Cref{prop:Psi0Recursion} but only allows us to express \emph{partial} acyclic polynomials $\Psi^{u\to v}$ using \emph{partial} acyclic polynomials of smaller digraphs.

\begin{proposition}\label{prop:Psi+Recursion}
Let $G = H\sqcup H'$ be a graph with a bipartition of the edges of $G$ such that $\V(H)\cap\V(H') = \{u, v\}$, and $D\sqcup D'$ an orientation of $G$.
Then
$$\begin{array}{rcl}
\Psi^{u\to v}(D\sqcup D'; t)
&=& \bigl(\Psi^{u\not\leftrightarrow v}(D'; t) + \Psi^{u\to v}(D'; t)\bigr)\cdot \Psi^{u\to v}(D; t) \vspace{0.2cm} \\
&=& \bigl(\Psi(D'; t) - t^{|E|}\cdot \Psi^{u\to v}(D'; \tfrac{1}{t})\bigr)\cdot \Psi^{u\to v}(D; t).
\end{array}$$
\end{proposition}

\begin{proof}
The set of acyclic orientations $A\sqcup A'$ of $G = D\sqcup D'$ such that $A(u, v) = 1$ is in bijection with the set of pairs $(A, A')\in \AO[H]\times\AO[H']$ such that $A(u, v) = 1$ and $A'(u, v)\ne -1$.
Consequently, since $t^{\rank_{D\sqcup D'} A\sqcup A'} = t^{\rank_{D} A}\cdot t^{\rank_{D'} A'}$:
$$\Psi^{u\to v}(D\sqcup D'; t)
= \sum_{\substack{A\sqcup A'\in \AO\\A(u, v) = 1}} t^{\rank_{D\sqcup D'} A\sqcup A'}
= \left(\sum_{\substack{A'\in \AO[D']\\A'(u, v) \in \{0, 1\}}} t^{\rank_{D'} A'}\right)\cdot\left(\sum_{\substack{A\in \AO[D]\\A(u, v) = 1}} t^{\rank_{D} A}\right)$$

The rightmost parenthesis is $\Psi^{u\to v}(D; t)$, while the left parenthesis is $$\sum_{\substack{A'\in \AO[D']\\A'(u, v) = 0}} t^{\rank_{D'} A'} + \sum_{\substack{A'\in \AO[D']\\A'(u, v) = 1}} t^{\rank_{D'} A'} = \Psi^{u\not\leftrightarrow v}(D'; t) + \Psi^{u\to v}(D'; t). $$ 
Using \Cref{prop:AcyclicPolyDecomposition}, we get the second formula.
\end{proof}

\begin{remark}\label{rmk:UncompleteDeletionContraction}
Sadly, \Cref{prop:Psi0Recursion,prop:Psi+Recursion} do not allow us to compute $\Psi(D; t)$ for any digraph $D$ via a deletion-contraction recursion.
Not only is the choice of the bipartition $G=H\sqcup H'$ slightly restrictive, but also, although the formula for $\Psi^{u\not\leftrightarrow v}(D\sqcup D; t)$ only involves $\Psi$ of smaller digraphs, the formula for $\Psi^{u\to v}(D\sqcup D'; t)$ involves some $\Psi^{u\to v}$ for smaller digraphs, and we provide no method (except exhaustive enumeration of acyclic orientations) to compute such polynomials.
Note that the polynomial $\Psi^{u\to v}(D; t)$ cannot be recovered knowing $\Psi(D; t)$ and $\Psi^{u\not\leftrightarrow v}(D; t)$ since it is not palindromic.
\end{remark}

\begin{remark}\label{rmk:StanleyDeletionContraction}
In \cite{Stanley-acyclicOrientations}, Stanley proved that $\psi(G) = \psi(G/_{uv}) + \psi(G\ssm uv)$ for all graphs $G$ and all edges $uv$ of $G$, where $\psi(G) = \Psi(D; 1)$ for any orientation $D$ of $G$.
Considering the bipartition of edges $G = H\sqcup H'$ where $H'$ is a single edge between $u$ and $v$, one can almost recover Stanley's deletion-contraction recursion.
Indeed, if $u,v$ have no common neighbors in $G$, then evaluating \Cref{prop:AcyclicPolyDecomposition,prop:Psi0Recursion,prop:Psi+Recursion} at $t = 1$ yields:
$$\begin{array}{rcl}
\Psi(D\sqcup D'; 1) &=& \Psi^{u\not\leftrightarrow v}(D\sqcup D'; 1) + 2\cdot\Psi^{u\to v}(D\sqcup D'; 1) \\
&=& \Psi(D'; 1)\cdot \Psi(D/_{uv}; 1) + 2\cdot\bigl(\Psi^{u\not\leftrightarrow v}(D'; 1) + \Psi^{u\to v}(D'; 1)\bigr)\cdot \Psi^{u\to v}(D; 1)
\end{array}$$
where $\Psi(D'; 1) = 2$, $\Psi^{u\not\leftrightarrow v}(D'; 1) = 0$, and $\Psi^{u\to v}(D'; 1) = 1$, since $D'$ is only an (oriented) edge.
Now, applying \Cref{prop:AcyclicPolyDecomposition,prop:Psi0Recursion} to $D$ itself gives:
$\Psi^{u\to v}(D; 1) = \frac{1}{2}\Psi(D; 1) - \frac{1}{2}\Psi^{u\not\leftrightarrow v}(D; 1) = \frac{1}{2}\Psi(D; 1) - \frac{1}{2}\Psi(D/_{uv}; 1)$.
Substituting in the above, we obtain:
$\Psi(D\sqcup D'; 1) = \Psi(D/_{uv}; 1) + \Psi(D; 1)$, which is precisely Stanley's recursion, as desired.

However, as our \Cref{prop:Psi0Recursion} does not apply if $u$ and $v$ are not suspended over $H$, one still needs to prove Stanley's recursion in these cases, and cannot recover it from the deletion-contraction phenomenon that we presented.
\end{remark}

\subsection{Acyclic signature}\label{ssec:AcyclicSignature}

We use the knowledge on the acyclic polynomial to compute the acyclic signature $\sigma(D) \eqdef \Psi(D; -1)$ of (di)graphs.
The actual value of the signature is not very important for the problem of Hamiltonicity of $\AO$:
we only care about whether $\sigma(D)$ is $0$ or not.
Indeed, by the parity argument (see \cite{SavageSquireWest93-GrayCodesAcyclicOrientations} or \Cref{lem:ParityArgument}), if $\sigma(D) \ne 0$ for some orientation $D$ or $G$, then $G$ is not {\good}.
The reciprocal is false, yet it seems non-trivial to produce an interesting $2$-connected graph with an even number of edges that would help to understand how far from characterizing $\c{AO}$-Hamiltonicity the acyclic signature is.

\subsubsection{Definition for (di)graphs, first examples and first properties}

%One may wonder whether the signature is a property of the graph $G$ itself or truly depends on the chosen orientation $D$.
We have seen in \Cref{fig:AcyclicPolyK23} that, up to a sign, the signature $\sigma(D)$ for $D$ an orientation of $K_{2, 3}$ does not depend on $D$.
This is actually a general fact:

\begin{proposition}\label{prop:SignatureIndependentFromOrientationUpToSign}
Let $D$ and $D'$ be two orientations of a graph $G$, then: $\sigma(D') = (-1)^{\rank_D D'}\cdot \sigma(D)$.
\end{proposition}

\begin{proof}
Suppose that $D' = D^Y$ for some $Y\subseteq 2^E$, and let $A = D^X$ be another orientation of $G$.
Thus, we have $A = (D')^{X\triangle Y}$ where $\triangle$ denotes the symmetric difference operator.
This implies that $\rank_{D'} A = |X\triangle Y| = |X| + |Y| - 2|X\cap Y| = \rank_D D' + \rank_D A - 2|X\cap Y|$ which has the same parity as $\rank_D D' + \rank_D A$, \ie $(-1)^{\rank_{D'} A} = (-1)^{\rank_D D'} \cdot (-1)^{\rank_D A}$.
Consequently:
$$\sigma(D') = \sum_{A\in \AO} (-1)^{\rank_{D'} A} = (-1)^{\rank_D D'} \cdot \sum_{A\in \AO} (-1)^{\rank_D A} = (-1)^{\rank_D D'} \cdot \sigma(D). ~ \qedhere$$
\end{proof}

In view of \Cref{prop:SignatureIndependentFromOrientationUpToSign}, we define $\mathdefn{\sigma(G)}$ to be the absolute value of $\sigma(D)$ for any orientation~$D$ of the graph $G$.
We can directly use the explicit formulas given in \Cref{ssec:AcyclicPolynomial} for $\Psi(D; t)$ to obtain formulas for $\sigma(G)$, see \Cref{tab:SignatureEasyFacts}.

% \begin{table}
% \centering
% \renewcommand{\arraystretch}{1.3}
% \begin{tabular}{c|c|c}
% Graph $G$ & $\sigma(G)$ & Proof by \\ \hline \hline
% Tree & $0$ & \Cref{exm:AcyclicPolynomialTrees} \\
% Cycle $C_n$ & $\left\{\begin{array}{ll} 0 & \text{if } n \text{ even} \\ 2 & \text{if } n \text{ odd}\end{array}\right.$ & \Cref{exm:AcyclicPolynomialCycles} \\
% Complete bipartite graph $K_{2, 3}$ & $6$ & \Cref{exm:AcyclicPolynomialK23} \\
% Complete graph $K_n$ & $0$ & \Cref{exm:AcyclicPol:CompleteGraphs} \\ \hline

% \makecell{Disconnected graph $G$\\with connected components $\V_1, \dots, \V_r$} & $\prod_k\sigma(G[\V_k])$ & \Cref{prop:AcyclicPol:Disconnected} \\
% $1$-sum $G = H\oplus_1 H'$ & $\sigma(H) \cdot \sigma(H')$ & \Cref{prop:AcyclicPol:1sums} \\
% \makecell{$G$ with simplicial node $v$\\with $b$ neighbors} & $\left\{\begin{array}{ll} \sigma(G\ssm\{v\}) & \text{if } b \text{ even} \\ 0 & \text{if } b \text{ odd}\end{array}\right.$ & \Cref{lem:AcyclicPol:SimplicialNodeAdd} \\
% Chordal graph & $0$ & \Cref{prop:AcyclicPol:ChordalGraph} \\
% Odd number of edges & $0$ & \Cref{cor:OddEdgeNbImpliesSignature0}
% \end{tabular}
% \renewcommand{\arraystretch}{1}
% \caption{First examples and first properties of the acyclic signature.}
% \label{tab:SignatureEasyFacts}
% \end{table}

\subsubsection{Towards a deletion-contraction algorithm}
Looking at the formulas we developed for $\Psi(D; t)$ in \Cref{ssec:AcyclicPolynomial}, one can see that some of them simplify drastically when evaluated at $t = 1$:
this allows for recovering many key properties of the number of acyclic orientations $\psi(G) = \Psi(D; 1)$, as we did in \Cref{rmk:StanleyDeletionContraction}.
Likewise, some of these formulas also simplify when evaluated at $t = -1$, yielding properties of the acyclic signature $\sigma(G) = \bigl|\Psi(D; -1)\bigr|$. % that we will now present.

%First, the parity of the number of edges of $G$ has a direct impact on its signature.

\begin{corollary}\label{cor:OddEdgeNbImpliesSignature0}
If $G$ has an odd number of edges, then $\sigma(G) = 0$.
\end{corollary}

\begin{proof}
By \Cref{cor:PsiIsPalindromic}, $\Psi(D; t)$ is palindromic of length $|E|$ for any orientation $D$ of $G = (\V, E)$, \ie $\Psi(D; t) = t^{|E|} \Psi(D; \tfrac{1}{t})$.
Evaluating at $t = -1$ for $|E|$ odd yields $\sigma(D) = - \sigma(D)$, so $\sigma(D) = 0$.
\end{proof}

For the case of graphs with an even number of edges, we can use the decomposition of \Cref{prop:AcyclicPolyDecomposition} and the deletion-contraction phenomenon described in \Cref{prop:Psi0Recursion,prop:Psi+Recursion}, but we need to work with $\sigma(D)$ instead of $\sigma(G)$ to keep track of signs.

\begin{definition}
Let $G = H\sqcup H'$ be a graph with a bipartition of the edges of $G$ such that $\V(H)\cap\V(H') = \{u, v\}$, and $D\sqcup D'$ an orientation of $G$.
We define the \defn{partial signatures} as $\mathdefn{\sigma^{u\not\leftrightarrow v}(D\sqcup D')} \eqdef \Psi^{u\not\leftrightarrow v}(D\sqcup D'; -1)$ and $\mathdefn{\sigma^{u\to v}(D\sqcup D')} \eqdef \Psi^{u\to v}(D\sqcup D'; -1)$.
\end{definition}

\begin{theorem}\label{thm:AcyclicSignatureContractionDeletionPhenomenon}
For $G$ a graph with an even number of edges,
let $G = H\sqcup H'$ be a graph with a bipartition of the edges of $G$ such that $\V(H)\cap\V(H') = \{u, v\}$, and $D\sqcup D'$ an orientation of $G$.
Then:
\begin{equation}\label{eqn:SigmaFromPartialSigmas}
\sigma(D\sqcup D') = \sigma^{u\not\leftrightarrow v}(D\sqcup D') + 2\cdot \sigma^{u\to v}(D\sqcup D')
\end{equation}
\begin{equation}\label{eqn:Sigma+}
\sigma^{u\to v}(D\sqcup D') = \bigl(\sigma(D') - \sigma^{u\to v}(D')\bigr)\cdot \sigma^{u\to v}(D) 
\end{equation}

% $$\begin{array}{rcl}
% \sigma(D\sqcup D') &=& \sigma^{u\not\leftrightarrow v}(D\sqcup D') + 2\cdot \sigma^{u\to v}(D\sqcup D') \vspace{0.2cm} \\
% \sigma^{u\to v}(D\sqcup D') &=& \bigl(\sigma(D') - \sigma^{u\to v}(D')\bigr)\cdot \sigma^{u\to v}(D) 
% \end{array}$$

Moreover, if $u$ and $v$ are suspended over $H$, then with the notations of \Cref{def:SuspendedNodes}, and $\varepsilon = 1$ if $uv\in E(H)$ and $\varepsilon = 0$ if $uv\notin E(H)$, we get:
\begin{equation}\label{eqn:Sigma0}
\sigma^{u\not\leftrightarrow v}(D\sqcup D') = (1 - \varepsilon) \cdot (-1)^x \cdot 2^{x+y} \cdot \sigma(D')\cdot \sigma(\Tilde{D}/_{uv})
\end{equation}
\end{theorem}

\begin{proof}
Direct applications of \Cref{prop:AcyclicPolyDecomposition,prop:Psi0Recursion,prop:Psi+Recursion} evaluated at $t = -1$ and $|E|$ even.
\end{proof}

%Before presenting an algorithm to compute certain acyclic signatures recursively, we need a last lemma.

\begin{lemma}\label{lem:AcyclicSignature:PartialToTrueSignature}
Let $G$ be a graph with nodes $u$ and $v$ which are suspended over $G$.
Then, we have
\begin{equation}\label{eqn:AcyclicSignature:PartialToTrueSignature}
\sigma^{u\to v}(D) = \frac{1}{2}\bigl(\sigma(D) - (-1)^x \cdot 2^{x+y} \cdot \sigma(\Tilde{D}/_{uv})\bigr) \,.
\end{equation}
\end{lemma}

\begin{proof}
Write the bipartition of edges $G = H\sqcup H'$ where $H'$ is the graph with nodes $u, v$ and no edge between them, by \Cref{thm:AcyclicSignatureContractionDeletionPhenomenon}~(\ref{eqn:SigmaFromPartialSigmas}), we get:
$\sigma(D) = \sigma^{u\not\leftrightarrow v}(D) + 2\cdot\sigma^{u\to v}(D)$, so we have $\sigma^{u\to v}(D) = \frac{1}{2}\bigl(\sigma(D) - \sigma^{u\not\leftrightarrow v}(D)\bigr)$.
Besides, by~(\ref{eqn:Sigma0})
$\sigma^{u\not\leftrightarrow v}(D) = (-1)^x \cdot 2^{x+y} \cdot 1 \cdot \sigma(\Tilde{D}/_{uv})$.
Thus,
$$\sigma^{u\to v}(D) = \frac{1}{2}\bigl(\sigma(D) - (-1)^x \cdot 2^{x+y} \cdot \sigma(\Tilde{D}/_{uv})\bigr) ~~\qedhere
$$
\end{proof}

\begin{remark}[Algorithm to compute the acyclic signature]\label{rmk:ContractionDeletionAlgorithmForSignature}
Constructing an algorithm to compute the acyclic signature of a graph $G$ is not always possible with our current formulas, but we describe how to do it for certain graphs.

First of all, if $G$ is not connected or not $2$-connected  \Cref{tab:SignatureEasyFacts} indicates that $\sigma(G)$ can be written as a product of the signatures of the $2$-connected components of $G$.
Secondly, if $G$ has an odd number of edges, then $\sigma(G) = 0$ by \Cref{cor:OddEdgeNbImpliesSignature0}.

Now, assume that $G$ is $2$-connected with an even number of edges.
Suppose there is a bipartition $G = H\sqcup H'$ of its edges such that $\V(H)\cap\V(H') = \{u, v\}$, $\vert E(H)\vert$ and $\vert E(H')\vert$ are both even, and $u$ and $v$ are suspended over $H$, then, combining \Cref{eqn:SigmaFromPartialSigmas,eqn:Sigma+,eqn:Sigma0,eqn:AcyclicSignature:PartialToTrueSignature} yields:
\begin{equation}\label{eqn:AcyclicSignatureContractionDeletionWhenAllGoesWell}
\sigma(D\sqcup D') = c\cdot b + \frac{1}{2}\cdot (c - d)\cdot (a + b)
\end{equation}
where, with the notations of \Cref{def:SuspendedNodes} for $x, y, \Tilde{D}$ and $x', y', \Tilde{D'}$:
$$\begin{array}{rclrclrclrcl}
a &=& \sigma(D), ~~  b &=& (-1)^x\cdot 2^{x+y}\cdot \sigma(\Tilde{D}/_{uv}) , ~~ c &=& \sigma(D') \text{, and } d &=& (-1)^{x'}\cdot 2^{x'+y'}\cdot \sigma(\Tilde{D'}/_{uv}) 
\end{array}\,.$$

Note that \Cref{eqn:AcyclicSignatureContractionDeletionWhenAllGoesWell} expresses the acyclic signature of $G$ using acyclic signatures of smaller (di)graphs, not using any partial signature.
Hence, this is a genuine induction step.
In particular, one can try dealing with the signatures $a$, $b$, $c$ and $d$ using the same method, or computing them directly via \Cref{tab:SignatureEasyFacts} or by exhaustive enumeration of acyclic orientations if the involved digraphs are tame enough.
Note that the common neighbors of $u$ and $v$ are taken care of in this induction step.

\end{remark}

\subsubsection{Acyclic signature of multipaths}\label{ssec:SignatureMultiPath}

\begin{definition}
A \defn{path} $\mathdefn{P_\ell}$ of \defn{length} $\ell\geq 1$ is a graph composed of $\ell+1$ \nodes~ $i_1\dots, i_{\ell+1}$, and edges $i_pi_{p+1}$ for $p\in [\ell]$.
A \defn{multipath} is a graph composed of several (internally) disjoint paths between the same endpoints $u$ and $v$, see \Cref{fig:MultiPath} for an illustration.
Precisely, for a sequence of integers $\b \ell = (\ell_1, \dots, \ell_s)$ with $\ell_1 \geq 1$ and $\ell_i \geq 2$ for all $i\geq 2$, the \defn{$\b \ell$-multipath} $\mathdefn{P_{\b \ell}}$ is the graph obtained as the contraction of paths $P_{\ell_1}, \dots, P_{\ell_s}$ at their endpoints, \ie the graph with nodes $u$, $v$ and $i_{1, 2}, \dots, i_{1, \ell_1}, \dots, i_{s, 2}, \dots, i_{s, \ell_s}$, and edges $ui_{p, 2}$, $i_{p, \ell_p}v$ for all $p\in [s]$ and $i_{p, m}i_{p, m+1}$ for all $p\in [s]$ and $m\in [2, \ell_p]$.

We denote $\mathdefn{D^{\b \ell}_{u\to v}}$ the acyclic orientation of $P_{\b \ell}$ ``from $u$ to $v$'', \ie with arcs $ui_{p, 2}$, $i_{p, \ell_p}v$ for all $p\in [s]$ and $i_{p, m}i_{p, m+1}$ for all $p\in [s]$ and $m\in [2, \ell_p]$.
\end{definition}

\begin{figure}[h]
    \centering
    \includegraphics[width=0.6\linewidth]{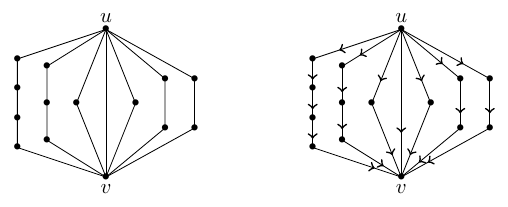}
    \caption[multipath]{
    For $\b\ell = (1, 2, 2, 3, 3, 4, 5)$, the multipath $P^{\b\ell}$ (left) and its orientation $D^{\b\ell}_{u\to v}$ (right).
    By \Cref{thm:AcyclicSignatureMultiPaths}, since $\sum_i \ell_i$ is even and some $\ell_i$ are odd, we have $\sigma(P^{\b\ell}) = -2 \ne 0$.
    }
    \label{fig:MultiPath}
\end{figure}

\begin{example}
The graph $K_{2, 3}$ is the multipath $P^{\b\ell}$ with $\b\ell = (2, 2, 2)$. Orientation $\circled{6}$ of \Cref{fig:AcyclicPolyK23} is the orientation $D^{\b\ell}_{u\to v}$, and we have $\sigma(D^{\b\ell}_{u\to v}) = +6 = (-1)^{3+1}\cdot(2^3 - 2)$.
\end{example}

The goal of this subsection is to compute the acyclic signature of multipaths.
The proof should be read as a showcase of the algorithm of \Cref{rmk:ContractionDeletionAlgorithmForSignature}.

\begin{theorem}\label{thm:AcyclicSignatureMultiPaths}
Fix any $\b\ell = (\ell_1, \dots, \ell_s)$ with $\ell_1 \geq 1$ and $\ell_i \geq 2$ for all $i\geq 2$.
If $\sum_i \ell_i$ is odd, then $\sigma(P^{\b\ell}) = 0$.
If $\sum_i \ell_i$ is even, then denoting $r$ the number of $i\in [s]$ such that $\ell_i$ is even, we have:
$$\sigma(D^{\b\ell}_{u\to v}) = \left\{\begin{array}{ll}
(-1)^{s+1}\cdot(2^s - 2) & \text{ if } \ell_i \text{ is even for all } i\in [s] \\
(-1)^r\cdot 2 & \text{ otherwise}
\end{array}\right. \,.$$
\end{theorem}

\begin{proof}
If $\sum_i \ell_i$ is odd, then $P^{\b\ell}$ has an odd number of edges, so by \Cref{cor:OddEdgeNbImpliesSignature0}, we get $\sigma(P^{\b\ell}) = 0$.

Suppose all $\ell_i$ are even for $i\in [s]$, let $u$ and $v$ be the two common endpoints of the multipath $P^{\b\ell}$, and let $\b\ell' = (\ell_1, \dots, \ell_s)$.
Assume by induction on $s$ that $\sigma(D^{\b\ell'}_{u\to v}) = (-1)^s\cdot(2^{s-1}-2)$, which holds true for $s = 1$ (where $P^{\b\ell}$ is a path, hence a tree) and $s = 2$ (where $P^{\b\ell}$ is an even cycle) by \Cref{tab:SignatureEasyFacts}.
Let $x$ be the number of $i\in [s-1]$ such that $\ell_i = 2$.
Fix the bipartition $P^{\b\ell} = P^{\b\ell'}\sqcup P^{\ell_s}$.
The deletion-contraction argument of \Cref{thm:AcyclicSignatureContractionDeletionPhenomenon} (see also \Cref{rmk:ContractionDeletionAlgorithmForSignature}) gives:
$$\begin{array}{rcl}
\sigma(D^{\b\ell}_{u\to v}) &=& \sigma^{u\not\leftrightarrow v}(D^{\b\ell}_{u\to v}) + 2 \cdot\sigma^{u\to v}(D^{\b\ell}_{u\to v}) \\
\sigma^{u\not\leftrightarrow v}(D^{\b\ell}_{u\to v}) &=& (-2)^x\cdot\sigma(D^{\ell_s}_{u\to v})\cdot\sigma(\Tilde{D}^{\b\ell}_{u\to v}/_{uv}) \\
\sigma^{u\to v}(D^{\b\ell}_{u\to v}) &=& \bigl(\sigma(D^{\ell_s}_{u\to v}) - \sigma^{u\to v}(D^{\ell_s}_{u\to v})\bigr) \cdot \sigma^{u\to v}(D^{\b\ell'}_{u\to v})
\end{array}$$

Since $P^{\ell_s}$ is a path, it is a tree, so $\sigma(D^{\ell_s}_{u\to v}) = 0$ by \Cref{tab:SignatureEasyFacts}.
Moreover, there is a unique acyclic orientation of $P^{\ell_s}$ which contains a directed path from $u$ to $v$, namely $D^{\ell_s}_{u\to v}$, so $\Psi^{u\to v}(D^{\ell_s}_{u\to v}; t) = 1$, and in particular $\sigma^{u\to v}(D^{\ell_s}_{u\to v}) = 1$.
It remains to compute $\sigma^{u\to v}(D^{\b\ell'}_{u\to v})$ which can be done using \Cref{lem:AcyclicSignature:PartialToTrueSignature} (where $y = 0$ by construction) because $u$ and $v$ have no common neighbors in $P^{\b\ell}$ except on paths of length $2$:
$\sigma^{u\to v}(D^{\b\ell'}_{u\to v}) = \frac{1}{2}\bigl(\sigma(D^{\b\ell'}_{u\to v}) - (-2)^{x} \cdot \sigma(\Tilde{D}^{\b\ell'}_{u\to v}/_{uv})\bigr)$. 

The graph $\Tilde{D}^{\b\ell}_{u\to v}/_{uv}$ is a $1$-sum of several cycles directed in a cyclic way of lengths $\ell_j$ for $j\in[s-1]$ such that $\ell_j\ne 2$:
these are $s-1-x$ even cycles.
Thus $\sigma^{u\to v}(D^{\b\ell'}_{u\to v}) = (-2)^{s-1-x}$ by \Cref{tab:SignatureEasyFacts}.
Putting everything back together, we get
$$\sigma(D^{\b\ell}_{u\to v}) = 0 + (0 - 1)\cdot \frac{1}{2}\bigl(\sigma(D^{\b\ell'}_{u\to v}) - (-2)^x\cdot (-2)^{s-1-x}\bigr).$$
As we have assumed by induction that $\sigma(D^{\b\ell'}_{u\to v}) = (-1)^s\cdot(2^{s-1} - 2)$, we indeed get $\sigma(D^{\b\ell}_{u\to v}) = (-1)^{s+1}\cdot(2^s - 2)$ if $\ell_i$ is even for all $i\in[s]$.

\smallskip

Last but not least, if there is some $j\in [s]$ such that $\ell_j$ is odd, then, since $\sum_i \ell_i$ is even, it implies that there exists another index $k\in[s]$, $k\ne j$ such that $\ell_k$ is odd.
Without loss of generality, we can assume that $\ell_{s-1}$ and $\ell_s$ are odd, and if there exists some $j\in [s]$ such that $\ell_j = 1$ then $\ell_s = 1$.
We denote $\b\ell' = (\ell_1, \dots, \ell_{s-2})$.
We consider the bipartition $P^{\b\ell} = P^{\b\ell'} \sqcup P^{(\ell_{s-1}, \ell_s)}$, and compute the corresponding (partial) signatures in a similar manner.
Note that $D^{(\ell_{s-1}, \ell_s)}_{u\to v}$ is a cycle of even length whose orientation differs on the direction of an odd number of arcs (namely $\ell_s$) from the cyclic orientation, so $\sigma(D^{(\ell_{s-1}, \ell_s)}_{u\to v}) = +2$.
Besides, $\Tilde{D}^{\b\ell'}_{u\to v}/_{uv}$ is a $1$-sum of cycles directed in a cyclic way:
if there is $j\in [s-2]$ such that $\ell_j$ is odd, then one of these cycles is of odd length, so $\sigma(\Tilde{D}^{\b\ell'}_{u\to v}/_{uv}) = 0$, otherwise, $\sigma(\Tilde{D}^{\b\ell'}_{u\to v}/_{uv}) = (-2)^{s-x-2}$.
We could compute $\sigma^{u\to v}(D^{(\ell_{s-1}, \ell_s)}_{u\to v})$ directly, but note that $u$ and $v$ have no common neighbors in $D^{(\ell_{s-1}, \ell_s)}_{u\to v}$, so we can instead apply \Cref{lem:AcyclicSignature:PartialToTrueSignature}, where $\Tilde{D}^{(\ell_{s-1}, \ell_s)}_{u\to v}/_{uv}$ is a $1$-sum of two odd cycles, so $\sigma(\Tilde{D}^{(\ell_{s-1}, \ell_s)}_{u\to v}/_{uv}) = 0$.
We now have all the pieces:
$$\begin{array}{rcl}
\sigma^{u\to v}(D^{(\ell_{s-1}, \ell_s)}_{u\to v}) &=& \frac{1}{2}\bigl(\sigma(D^{(\ell_{s-1}, \ell_s)}_{u\to v}) - \sigma(\Tilde{D}^{(\ell_{s-1}, \ell_s)}_{u\to v}/_{uv})\bigr) = \frac{1}{2}\cdot(2 - 0) = 1 \vspace{0.3cm} \\

\sigma(D^{\b\ell}_{u\to v}) &=& (-2)^x \cdot 2 \cdot \sigma(\Tilde{D}^{\b\ell'}_{u\to v}/_{uv}) + 2\cdot\bigl(2 - 1\bigr)\cdot\frac{1}{2}\bigl(\sigma(D^{\b\ell'}_{u\to v}) - (-2)^x\cdot \sigma(\Tilde{D}^{\b\ell'}_{u\to v}/_{uv})\bigr) \vspace{0.2cm} \\
&=& (-2)^x\cdot\sigma(\Tilde{D}^{\b\ell'}_{u\to v}/_{uv}) + \sigma(D^{\b\ell'}_{u\to v})
\end{array}
$$
To finish, if $\b\ell'$ contains only even elements, then $r = s - 2$, $\sigma(\Tilde{D}^{\b\ell'}_{u\to v}/_{uv}) = (-2)^{s-x-2} = (-2)^{r-x}$, and $\sigma(D^{\b\ell'}_{u\to v}) = (-1)^{r+1}(2^r - 2)$ by induction:
hence $\sigma(D^{\b\ell}_{u\to v}) = (-2)^r + (-1)^{r+1}(2^r - 2) = (-1)^r\cdot 2$.
If $\b\ell'$ contains some odd elements, then $\sigma(\Tilde{D}^{\b\ell'}_{u\to v}/_{uv}) = 0$, and $\sigma(D^{\b\ell'}_{u\to v}) = (-1)^r\cdot 2$ by induction:
hence $\sigma(D^{\b\ell}_{u\to v}) = (-1)^r\cdot 2$.
This concludes the proof.
\end{proof}

\begin{corollary}\label{cor:AOhamiltonianMultiPathImpliesSumLengthsEven}
For any $\b\ell = (\ell_1, \dots, \ell_s)$ with $\ell_1 \geq 1$ and $\ell_i \geq 2$ for all $i\geq 2$, if $P^{\b\ell}$ is {\good}, then $\sum_{i=1}^s \ell_i$ is odd.
\end{corollary}

\begin{proof}
Direct consequence of \Cref{thm:AcyclicSignatureMultiPaths,lem:ParityArgument}.
\end{proof}

\begin{remark}
One could decide whether the acyclic signature of multipaths is $0$ or not via a more direct count of the number of acyclic orientations of multipaths.
We prefer the proof of \Cref{thm:AcyclicSignatureMultiPaths} since it emphasizes the properties of the acyclic signature and hints at ways to compute it for larger families of graphs, but the interested reader may have a look at Appendix~\ref{app:ParityArgumentMultiPathDirectCount}.
\end{remark}

\begin{remark}
Let $H$ be a graph with two nodes $u, v$, and $H' = P^{\b\ell}$ a multipath with endpoints $u, v$ such that $\sum_i \ell_i$ is even.
One can deduce the acyclic signature of $H\sqcup H'$ from the acyclic signature of $H$ if $H$ has an even number of edges, and $u$ and $v$ do not have a common neighbor in $H$ except for nodes on induced paths of length $2$ between $u$ and $v$.
Indeed, in this context, we can use \Cref{eqn:AcyclicSignatureContractionDeletionWhenAllGoesWell}, where $c = \sigma(D^{\b\ell}_{u\to v})$ is known from \Cref{thm:AcyclicSignatureMultiPaths} and $d = (-2)^{x'}\cdot \sigma(\Tilde{D}^{\b\ell}_{u\to v}/_{uv})$ is equal to $0$ if there is some $j\in [s]$ such that $\ell_j$ is odd, and is equal to $(-2)^s$ otherwise (since $\Tilde{D}^{\b\ell}_{u\to v}/_{uv}$ is a $1$-sum of cycles directed in a cyclic way).
\end{remark}

% \clearpage
\section{\texorpdfstring{$\c{AO}$}{AO}-Hamiltonicity: (multi)path gluing, and simplicial nodes}\label{sec:MultiPathGluing}

Our goal in this section is to construct a Hamiltonian cycle in $\AO$ where $G$ is obtained from an {\good} graph $H$ by adding a multipath between two nodes of $H$.

We give an overview of the strategy for proving Hamiltonicity of graphs in a constructive way, in \Cref{ssec:LacingPatterns}.
We begin with a key element in the proof of Hamiltonicity of the graph of regions of supersolvable arrangements in \cite{BrennerCardinalMcConvilleMerinoMutze25-CombinatorialGenerationVII}.

\begin{definition}
Denoting $P_\ell$ the path on $\ell$ nodes with endpoints $u$ and $v$ for $\ell \geq 2$, a graph $G$ is \defn{$\ell$-suspended} over a graph $H = (N, E)$ if $G$ is a spanning graph of the Cartesian product $H \times P_\ell$, such that $G[\{x\} \times P_{\ell}] \simeq P_\ell$ for all $x \in \V$ and $G[N \times \{u\}] \simeq G[N \times \{v\}] \simeq H$.   
\end{definition}

This graph $G$ is well-behaved in terms of Hamiltonicity if the graph $H$ was Hamiltonian (see \cite[Lemma~15]{BrennerCardinalMcConvilleMerinoMutze25-CombinatorialGenerationVII}).
This is illustrated in \Cref{fig:GraphSuspension}.

\begin{figure}[h]
    \centering
    \includegraphics[width=0.7\linewidth]{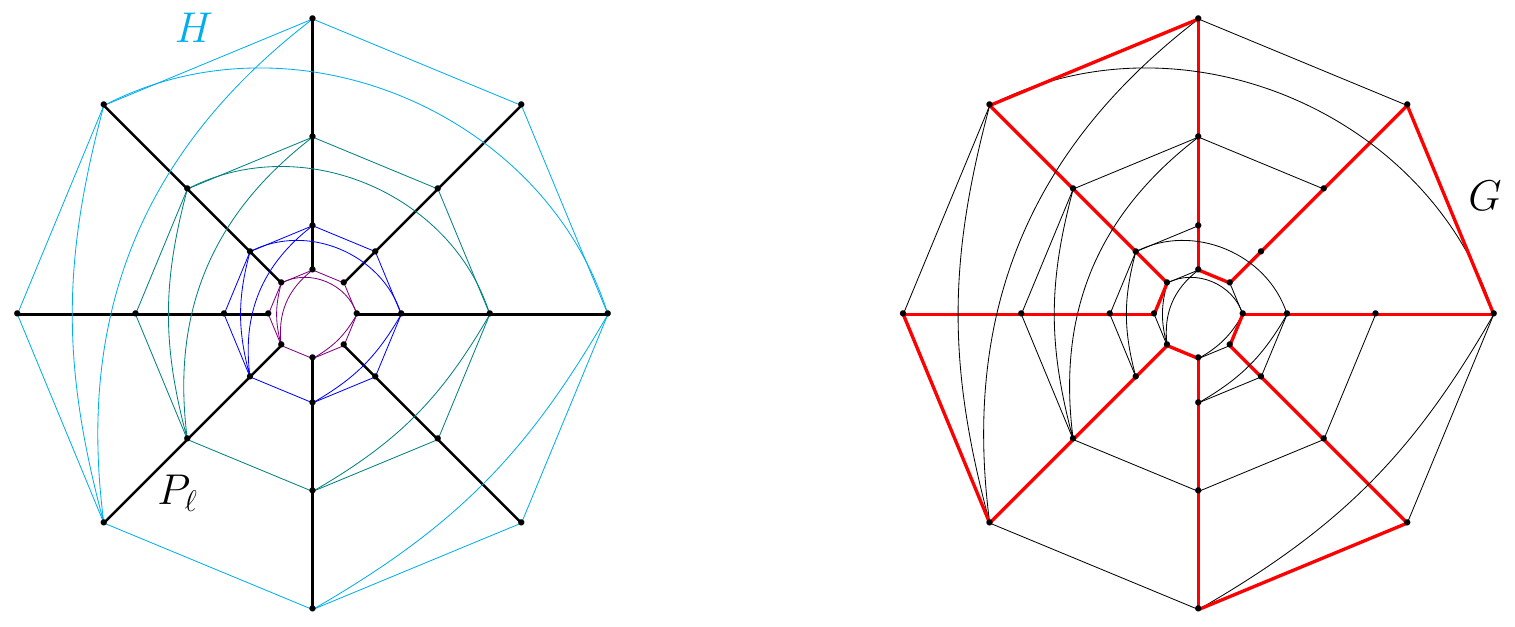}
    \caption{(Left) The Cartesian product $P_4\times H$ where the graph $H$ in \textcolor{cyan}{cyan}.
    (Right) An $4$-suspended graph $G$ over $H$.
    In \textcolor{red}{red}, a Hamiltonian cycle in $G$.
    In general, there might be some nodes missing in the copies of $H$ (see the nodes labeled as a red star $\textcolor{red}{\star}$ in a blue circle $\textcolor{blue}{\circ}$ in \Cref{fig:11ReplacingEven}) and thus, it is inevitable to use some of the black edges in the right picture to construct a Hamiltonian cycle.}
    \label{fig:GraphSuspension}
\end{figure}

The proof for supersolvable arrangements takes advantage of the fact that if we add new hyperplanes according to the modular chain witnessing supersovability, we subdivide each of the existing chambers in a very regular way, identical to all chambers.
Thus, we only need to specify how to walk the new chambers subdividing the existing ones, and then zigzag through the existing chambers through the graphic suspension. 

\medskip

We want to analyze which operations on graphs we can perform, such that if the base graph is {\good}, the resulting graph is also {\good}. This means that we generalize this approach of graphic suspension to a case where the two graphs $G[N \times \{0\}]$ and $G[N \times \{\ell\}]$ need not be copies of the same graph, but rather instances of graphs of the subdivisions of neighboring chambers of the original arrangement through the new hyperplanes. In this case, we find that it can be necessary to use the edges between two graphs $G[\{v\} \times i]$ and $G[\{w\} \times i]$, where $\{v,w\}$ is an edge in the original graph $H$. % This is what we want to call \defn{pattern lacing}.

For our problem, what the authors of \cite{BrennerCardinalMcConvilleMerinoMutze25-CombinatorialGenerationVII} proved implies that if $G$ is obtained from $H$ by adding a simplicial node on a clique of size $b$, then $\AO$ is $b$-suspended over $\AO[H]$, so the zigzag method guarantees that if $H$ is {\good}, then $G$ is also {\good}.
We will discuss this idea in \Cref{ssec:AddingSimpliceNode} since we need more conditions on the nodes and certain sequences associated to them. 

The question we are concerned with is what to do when $G$ is obtained from $H$ by an operation more complicated than adding a simplicial node, especially when $\AO$ is not suspended over $\AO[H]$. 
We propose two generalizations of the zigzag method, namely \Cref{cor:SelfLacingIntoAnAOcycle,thm:ConcatenatingReplacingPatternsWithSharedEndPoints}.
Both are based on the following core idea:
if $G = H\sqcup H'$ is a bipartition of the edges of $G$, where $H'$ shall be thought of as the edges we add to $H$, then $\AO$ is an induced subgraph of $\AO[H]\times\AO[H']$ (see \Cref{lem:BipartitionEdgesYieldsCartesianProductAO}).
For each acyclic orientation $A$ of $H$, we consider the graph of acyclic orientations of $G$ of the form $A\sqcup A'$ for all admissible $A'\in \AO[H']$:
if we find a Hamiltonian path $P_A$ through these acyclic orientations and if we manage to guarantee that we can ``string together'' the $(P_A ~;~ A\in \AO[H])$, then we would have a Hamiltonian path (or cycle) of $\AO$.
This is the idea of \Cref{thm:ConcatenatingReplacingPatternsWithSharedEndPoints}.

Yet, such a straightforward construction is usually not possible.
Therefore, instead of just one single orientation, we consider entire sequences $(A_1, \dots, A_s)$ of adjacent acyclic orientations of $H$ (\ie directed paths in $\AO[H]$), and find a Hamiltonian path in the subgraph of $\AO$ induced by all the acyclic orientations of the form $A_i\sqcup A'$ for $i\in [s]$ and $A'\in \AO[H']$.
We call such a Hamiltonian path a \emph{replacing pattern} for $(A_1, \dots, A_s)$.
To be able to ``string together'' these replacing patterns in order to construct a Hamiltonian path (or cycle) for $\AO$, we develop a theory of \emph{lacing} patterns, leading to \Cref{prop:LacingIsHamiltonianPath,cor:SelfLacingIntoAnAOcycle}.

We formalize these ideas, and apply them to gluing multipaths in \Cref{sssec:DoublePathGluingOddLength,sssec:PathGluingEvenLength,sssec:MultiPathGluingEvenTotalLength}.

\subsection{How to lace patterns}\label{ssec:LacingPatterns}

Recall from \Cref{def:Direction} that $A(u, v)$ is the \emph{direction} induced by the acyclic orientation $A$ on the nodes $u, v$ (\ie $\pm1$ if there is a directed path from $u$ to $v$ or from $v$ to $u$ in $A$, and $0$ otherwise), and that we denote $\b A(u, v) \eqdef \bigl(A_1(u, v), \dots, A_r(u, v)\bigr)$ for a sequence of acyclic orientations $\b A = (A_1, \dots, A_r)$.

Let $G = H\sqcup H'$ be a bipartition of the edges of the graph $G = (\V, E)$.
We have seen in \Cref{lem:BipartitionEdgesYieldsCartesianProductAO} that $\AO$ is the subgraph of the Cartesian product $\AO[H]\times \AO[H']$ induced on the pairs of orientations $(A, A')$ with $A\in \AO[H]$ and $A'\in \AO[H']$ such that $A(u, v) \cdot A'(u, v) \ne -1$ for all $u, v\in \V$.
Hence, for an acyclic orientation $A$ of $H$, we denote by $\mathdefn{\AO[H' | A]}$ the set of acyclic orientations $A'$ of $H'$ such that $A(u, v)\cdot A'(u, v) \ne -1$ for all $u, v\in \V$.

Let $\b A = (A_1, \dots, A_s)$ be a directed path in $\AO[H]$.
We denote by $\M$ the subgraph of $\AO$ induced on all the orientations $A_j \sqcup A'_k$ of $G$ for $A'_k \in \AO[H'|A_j]$ with $1 \leq j \leq s$.

\begin{definition}
A \defn{replacing pattern} $\b P$ for $\b A$ is a (directed) Hamiltonian path in $\M$ that starts at some $A_1 \sqcup A'$ for some $A'\in \AO[H' | A_1]$ and ends at some $A_s \sqcup B'$ for some $B'\in \AO[H' | A_s]$.
\end{definition}

We denote by $\mathdefn{\omega_i^{\b P}}$ the total order on $\AO[H'|A_i]$ (\ie the bijection from $\AO[H'|A_i]$ to $[m]$ with $m \eqdef |\AO[H'|A_i]|$) induced by the replacing pattern $\b P$, that is $\omega_i^{\b P}(A') \leq \omega_i^{\b P}(B')$ if $A_i \sqcup A'$ appears before $A_i \sqcup B'$ in the directed path $\b P$.

For a replacing pattern $\b P$ on $(A_1, \dots, A_s)$ and $i\in [s]$, we define the \defn{consecutive locus} $\mathdefn{\cons_i(\b P)}$ to be the set of indices $j$ such that for the orientation $A' \in \AO[H'|A_i]$ with $\omega_i^{\b P}(A') = j$, it holds that the node in $\b P$ after $A_i \sqcup A'$ is of the form $A_i \sqcup B'$ for some $B' \in \AO[H'|A_i]$.

%to be the set of indices $j$ such that there is an edge in $\b P$ from $(A_i, A')$ to $(A_i, B')$ with $\omega_i^{\b P}(A') = j$ and $\omega_i^{\b P}(B') = j+1$.
Denoting $m \eqdef |\AO[H'|A_s]|$, two replacing patterns $\b P$ on $(A_1, \dots, A_s)$ and $\b Q$ on $(A_s, \dots, A_t)$ \defn{can lace} if $\omega_s^{\b P} = \omega_s^{\b Q}$ and $\cons_s(\b P) \cup \cons_s(\b Q) = [m]$.
If they can lace,
we write $\b P$ as the concatenation of directed paths $\mathdefn{\b P_0}\cup \dots\cup \mathdefn{\b P_{m-1}}$ where each $\b P_j$ is the directed sub-path from the node $A_s \sqcup A'$ with $\omega_s^{\b P}(A') = j$ to the node $A_s \sqcup B'$ with $\omega_s^{\b P}(B') = j+1$ (by convention, $\b P_0$ starts at the starting node of $\b P$).
Similarly, we write $\b Q$ as the concatenation of directed paths $\mathdefn{\b Q_1} \cup \dots\cup \mathdefn{\b Q_m}$ where each $\b Q_j$ is the directed sub-path from the node $A_s \sqcup A'$ with $\omega_s^{\b Q}(A') = j$ to the node $A_s \sqcup B'$ with $\omega_s^{\b Q}(B') = j+1$ (by convention, $\b Q_m$ ends at the ending node of $\b Q$).
Then, the \defn{lacing $\b P\wedge \b Q$} is the directed path obtained as the concatenation of $\b P_0 \cup \b X_1 \cup \dots \cup \b X_{m-1}\cup \b Q_m$ where:
$$\b X_j = \left\{\begin{array}{ll}
\b P_j & \text{ if } j\in \cons_s(\b Q) \\
\b Q_j & \text{ otherwise}
\end{array}\right..$$
Note that the condition $\cons_s(\b P) \cup \cons_s(\b Q) = [m]$ ensures that the ``otherwise'' case implies $j\in \cons_s(\b P)$.

\begin{proposition}\label{prop:LacingIsHamiltonianPath}
Let $\b P$ (resp. $\b Q$) be a replacing pattern for the subsequence $(A_1, \dots, A_s)$ (resp. $(A_s, \dots, A_t)$).
If $\b P$ and $\b Q$ can lace, then their lacing $\b P\wedge \b Q$ is a replacing pattern for $(A_1, \dots, A_t)$.
\end{proposition}

\begin{proof}
Since $\b P$ and $\b Q$ can lace, we have $\omega_s^{\b P} = \omega_s^{\b Q}$, so $\b X_j$ and $\b X_{j+1}$ meet on the node $A_s\sqcup A'$ with $\omega_s{\b P}(A') = j+1 = \omega_s^{\b Q}(A')$.
Hence, by construction of $\b P\wedge \b Q$, it is a directed path inside $\M[A_1, \dots, A_s, \dots, A_t]$.
Moreover, $\b P\wedge\b Q$ starts at some $A_1\sqcup A'$ for some $A'\in \AO[H'| A_1]$ (because its starting node is the same as the one of $\b P$), and ends at some $A_t\sqcup B'$ for some $B'\in\AO[H'|A_t]$ (because its ending node is the same as the one of $\b Q$).
Hence, it is enough to prove that all nodes of $\M[A_1, \dots, A_s, \dots, A_t]$ appear in $\b P\wedge\b Q$.

We use the notations presented in the definition of the lacing.
Consider the restriction of $\b P\wedge \b Q$ to $\M[A_1, \dots, A_s]$: by construction, if $\b P_j\not\subset \b P\wedge\b Q$, then $j\in \cons_s(\b P)$, so $\b P_j$ is composed of a single edge from $A_s\sqcup A'$ to $A_s\sqcup B'$ for some $A', B'\in\AO[H'|A_s]$.
Note that these two nodes $A_s\sqcup A'$ to $A_s\sqcup B'$ are in the path $\b P\wedge \b Q$ since they belong to $\b Q_j$, only the edge is not there.
Hence, all the nodes appearing in $\b P$ (resp. $\b Q$) also appear in $\b P\wedge\b Q$.
As $\b P$ (resp. $\b Q$) is a Hamiltonian path in $\M[A_1, \dots, A_s]$ (resp. in $\M[A_s, \dots, A_t]$), we indeed get that $\b P\wedge \b Q$ is a Hamiltonian path in $\M[A_1, \dots, A_s, \dots, A_t]$.
\end{proof}

Finally, a replacing pattern $\b P$ for a directed cycle $(A_1, \dots, A_s)$ in $\AO[H]$ with $A_s = A_1$ \defn{can self-lace} if $\cons_s(\b P) = \{1, \dots, |\AO[H' | A_s]|\}$ and for the orientation $A'_\circ\in \AO[H'|A_s]$ such that $\omega_s^{\b P}(A'_\circ) = 1$ we also have $\omega_1^{\b P}(A'_\circ) = 1$.
The \defn{self-lacing} of $\b P$, denoted $\mathdefn{\wedge \b P}$, is defined as $\b P$ without its edges from  $A_s\sqcup A'$ to $A_s\sqcup B'$ for some $A', B'\in\AO[H'|A_s]$.

\begin{corollary}\label{cor:SelfLacingIntoAnAOcycle}
Let $G = H\sqcup H'$ be a bipartition of the edges of $G$.
If $\b A$ is a Hamiltonian cycle for $\AO[H]$ and $\b P$ a replacing pattern for $\b A$ that can self-lace, then $\wedge \b P$ is a Hamiltonian cycle for $\AO$.
\end{corollary}

\begin{proof}
Since $\omega_s^{\b P}$ and $\omega_1^{\b P}$ have the same first element $A'_\circ$, the first node of $\b P$ is equal to the last node of $\b P$ without its edges from $A_s\sqcup A'$ to $A_s\sqcup B'$ for some $A', B'\in\AO[H'|A_s]$.
Hence $\wedge \b P$ is a (directed) cycle.
Moreover, all the nodes of $\M[A_1, \dots A_{s-1}]$ appear in $\wedge \b P$, so because $A_1 = A_s$, the cycle $\wedge \b P$ is a Hamiltonian cycle for $\AO$.
\end{proof}

One can generalize \Cref{prop:LacingIsHamiltonianPath,cor:SelfLacingIntoAnAOcycle} in order to weaken their hypotheses.
As it makes the hypotheses less concise and handy, we keep both as stated above, for the more general statement see Appendix \ref{appendix:LooseLacing}.

%\medskip

%We also need the following theorem which is easier conceptually yet encumbering to state.

\begin{theorem}\label{thm:ConcatenatingReplacingPatternsWithSharedEndPoints}
Let $G = H\sqcup H'$ be a bipartition of the edges of $G$.
Suppose that there exists $X', Y'\in \AO[H']$ and a Hamiltonian path $\b A = (A_1, \dots, A_s)$ in $\AO[H]$ such that:
\begin{compactitem}
\item $X' \in \AO[H'|A_i]$ and $Y'\in \AO[H'|A_i]$ for all $i\in [s]$ (equivalently, $X'(u, v) = Y'(u, v) = 0$ for all nodes $u, v$ connected in $H$), and
% \item One can split the sequence $\b A$ into \defn{splitting subsequences} $\b A_j \eqdef (A_{i_j}, A_{i_j+1}, \dots, A_{i_{j+1}-1})$ for $j\in [r]$ with $i_1 = 1$ and $i_r - 1 = s$ such that for each $j\in [r]$, the sequence $\b A_j$ admits a replacing pattern which starts at $(A_{i_j}, A_j')$ and ends at $(A_{i_{j+1} -1}, B_j')$ either for $\{A_j', B_j'\} = \{X', Y'\}$, either both for $A_j' = B_j' = X'$ and for $A_j' = B_j' = Y'$,
\item One can split the sequence $\b A$ into \defn{splitting subsequences} $\b A_j \eqdef (A_{i_j}, A_{i_j+1}, \dots, A_{i_{j+1}-1})$ for $j\in [r]$ with $i_1 = 1$ and $i_r - 1 = s$ such that for each $j\in [r]$, the sequence $\b A_j$ admits either:
\begin{compactitem}
\item a replacing pattern which starts at $A_{i_j} \sqcup A_j'$ and ends at $A_{i_{j+1} -1} \sqcup B_j'$ either for $\{A_j', B_j'\} = \{X', Y'\}$, or
\item two replacing patterns which start at $A_{i_j} \sqcup A_j'$ and end at $A_{i_{j+1} -1} \sqcup B_j'$, one for $A_j' = B_j' = X'$ and one for $A_j' = B_j' = Y'$,
\end{compactitem}
\end{compactitem}
then there exists a Hamiltonian path in $\AO$ obtained as a concatenation of some of these replacing patterns.

Moreover, if $\b A$ is a Hamiltonian cycle of $\AO[H]$, and $\#\{j\in [r]~;~ A'_j\ne B'_j\}$ is even, % if the number of $j\in [r]$ such that $A_j'\ne B_j'$ is even
then $G$ is {\good}.
\end{theorem}

\begin{proof}
We will concatenate replacing patterns.
Note that one can reverse the direction of every arc in a replacing pattern to exchange its starting and ending nodes.

Let $\varsigma_j = \#\{k\in[j-1] ~;~ A'_k \ne B'_k\}$.
We choose the replacing pattern $\b P_j$ as follows:
\begin{compactitem}
\item if $A'_j \ne B'_j$ and $\varsigma_j$ is even, then $\b P_j$ starts at $A_{i_j}\sqcup X'$ and ends at $A_{i_{j+1}-1}\sqcup Y'$;
\item if $A'_j \ne B'_j$ and $\varsigma_j$ is odd, then $\b P_j$ starts at $A_{i_j}\sqcup Y'$ and ends at $A_{i_{j+1}-1}\sqcup X'$;
\item if $A'_j = B'_j$ and $\varsigma_j$ is odd, then $\b P_j$ starts at $A_{i_j}\sqcup Y'$ and ends at $A_{i_{j+1}-1}\sqcup Y'$;
\item if $A'_j = B'_j$ and $\varsigma_j$ is even, then $\b P_j$ starts at $A_{i_j}\sqcup X'$ and ends at $A_{i_{j+1}-1}\sqcup X'$.
\end{compactitem}
For all $j\in [r-1]$, by construction of $\varsigma_j$, the orientation of $H$ of the last node of $\b P_j$ is $A_{i_{j+1}-1}$ and the orientation of $H$ of the first node of $\b P_{j+1}$ is $A_{i_{j+1}}$ and the orientation of $H'$ is the same (either $X'$ or $Y'$), so these two nodes are adjacent in $\AO$.
In particular, the concatenation $(\b P_1, \b P_2, \dots, \b P_r)$ is a Hamiltonian path in $\AO$.

Moreover, if $\#\{j\in [r] ~;~ A'_j \ne B'_j\}$ is even, then either $\varsigma_r$ is odd and $A'_j\ne B'_j$, so its last node is $A_s\sqcup X'$, or $\varsigma_r$ is odd and $A'_j = B'_j$, so its last node is $A_s\sqcup X'$.
If, in addition, $A_s$ and $A_1$ are adjacent in $\AO[H]$ (\ie if $\b A$ is a Hamiltonian cycle of $\AO[H]$), then the last node of $\b P_r$ and the first node of $\b P_1$ are adjacent in $\AO$ so $(\b P_1, \b P_2, \dots, \b P_r)$ is a Hamiltonian cycle in $\AO$.
\end{proof}

As before, one can generalize \Cref{thm:ConcatenatingReplacingPatternsWithSharedEndPoints} in order to weaken its hypotheses, for the more general statement, see Appendix \ref{appendix:SplittingSubsequences}.

\subsection{Multipath gluing and 0-even sequences}\label{ssec:MultiPathGluingAnd0Even}

For a graph $G = (\V, E)$ and two \nodes~ $u, v\in \V$, the \defn{path gluing} of length $\ell$ at $u, v$, denoted $\mathdefn{\oplus_{u, v}^\ell G}$, is the graph obtained by adding $\ell-1$ \nodes~ to $G$ and constructing a path of length $\ell$ between $u$ and $v$, supported on these new \nodes. 
Formally, using the notation we have introduced for gluings in \Cref{ssec:Prelim:OperationsOnGraphs}, let $P_\ell$ be a path of length $\ell$ with end nodes $x$ and $y$, and let $\varphi_1 : \{u\} \to \{x\}$ and $\varphi_2 : \{v\}\to\{y\}$, then $\oplus_{u, v}^\ell G \eqdef G\oplus_{\varphi_1, \varphi_2} P_\ell$.

The \defn{multipath gluing} of length $\b \ell = (\ell_1, \dots, \ell_m)$ at $u, v$ is then defined recursively for $m\geq 2$ as the graph $\mathdefn{\oplus_{u, v}^{\b \ell} G}  \eqdef \oplus_{u, v}^{\ell_m}\bigl(\oplus_{u, v}^{(\ell_1, \dots, \ell_{m-1})} G\bigr)$.

\begin{definition}\label{def:zero_evenness}
A sequence of integers $\b a = (a_1, \dots, a_r)$ is \defn{$0$-even} if, between any two non-zero entries of $\b a$, there is an even number of zeros in $\b a$.
\end{definition}

\subsubsection{Gluing a path of even length}\label{sssec:PathGluingEvenLength}

Our main objective in this section is to prove \Cref{thm:PathGluingEvenIsGood}.
This motivates us to define special subsequences as well as a proposition on regularity of replacing patterns (see Proposition \ref{prop:ReplacingPatternsViaDrawings}).
The proof is delayed to page~\pageref{proof:MainThm}.

\begin{theorem}\label{thm:PathGluingEvenIsGood}
Let $G = (\V, E)$ be a {\good} graph with $\b A = (A_1, \dots, A_s)$ a Hamiltonian cycle in $\AO$.
For $u, v\in \V$ in the same connected component of $G$ and an even number $\ell\geq 2$, if $\b A(u, v)$ is $0$-even, then $\oplus_{u, v}^{\ell} G$ is {\good}.
Moreover, under these assumptions, there exists a Hamiltonian cycle $\b B = (B_1, \dots, B_t)$ in $\AO[\oplus_{u, v}^\ell G]$ such that $\b B(u, v)$ is $0$-even.
\end{theorem}

Our proof will explicitly construct a new Hamiltonian cycle $\b B$ for $\AO[\oplus_{u, v}^\ell G]$ with $\ell$ even, from any Hamiltonian cycle $\b A$ of $\AO$ (such that $\b A(u, v)$ is $0$-even) and any Hamiltonian cycle of the cube-graph $Q_{\ell}$.

For the cube-graph $Q_{\ell}$, fix a Hamiltonian cycle $\mathdefn{\b S} \eqdef (S_1, \dots, S_{2^\ell})$ of $Q_\ell \simeq \AO[P_\ell]$ (see \Cref{exm:CubeGraphIsGood} for its existence), where we interpret each $S_p$ as a subset of $E(P_\ell)$ (instead of a subset of $[\ell]$).
Without loss of generality, we can suppose that $\mathdefn{S_{2^{\ell}} = [\ell] \simeq E(P_\ell)}$, \ie $S_{2^{\ell}}(u, v) = -1$; and we denote by $\mathdefn{\mu}$ the index in  $[2^\ell]$ such that $\mathdefn{S_\mu = \emptyset}$, \ie $S_\mu(u, v) = 1$.
Note that $Q_\ell$ is bipartite:
$\emptyset$ and $E(P_\ell)$ belong to the same part if and only if $\ell$ is even.
Consequently, as we have supposed $\ell$ even, this implies that $\mu$ is even and not $0$.

Let $\b A = (A_1, \dots, A_s)$ be a Hamiltonian cycle in $\AO[G]$ such that $\b A(u, v)$ is $0$-even.
Without loss of generality, we can suppose that $A_1(u, v) = 1$.
For $a, b\in \{-1, +1\}$, a \defn{$(a\to b)$-subsequence} of $\b A$ is a sequence $(A_p, A_{p+1}, \dots, A_q)$ for some $p \leq q$ considered modulo $r$ such that $A_p(u, v) = a$ and $A_q(u, v) = b$ and $A_j(u, v) = 0$ for all $p < j < q$.
If $p = q-1$, then we call it a \defn{degenerate} $(a\to b)$-subsequence of $\b A$. This is the case if and only if the edge between $u$ and $v$ exists in $G$. 

\begin{proposition}\label{prop:ReplacingPatternsViaDrawings}
Let $G = (\V, E)$ be a {\good} graph with $\b A = (A_1, \dots, A_s)$ a Hamiltonian cycle in $\AO$.
For $u, v\in \V$ in the same connected component of $G$ and an even number $\ell\geq 2$, if $\b A(u, v)$ is $0$-even, then for any $(a\to b)$-subsequence $(A_p, \dots, A_q)$ of $\b A$ with $a, b\in \{-1, +1\}$, there exists a replacing pattern $\b P$ such that:
\begin{compactenum}[(a)]
\item $\b P(u, v)$ is $0$-even, and\label{item:0EvenReplacingEvenPath}
\item if $b = 1$ (resp. $b = -1$), then $\omega_q^{\b P} = 1\, 2\, \dots 2^{\ell}-1$ (resp. $\omega_q^{\b P} = \mu+1\, \mu+2\, \dots \mu-1$), and\label{item:omegaReplacingEvenPath}
\item $\cons_p(\b P) \supseteq \{S_{2k} ~;~ k \in [2^{\ell}]\}$ and $\cons_q(\b P) \supseteq \{S_{2k+1} ~;~ k \in [2^{\ell}]\}$, and\label{item:consReplacingEvenPath}
\item if $\b A$ is not degenerate and $A_q(u, v) = 1$ (resp. $A_q(u, v) = -1$), then $\cons_q(\b P) = [2^\ell]\ssm\{2^\ell\}$ (resp. $\cons_q(\b P) = [2^\ell]\ssm\{\mu\}$).\label{item:CyclicConditionReplacingEven}
\end{compactenum}
\end{proposition}

\begin{proof}
The proof is by sketching the relevant replacing patterns with the claimed properties.

For a non-degenerate $(1\to1)$-subsequence, \Cref{fig:11ReplacingEven} illustrates a replacing pattern with the claimed properties.
The node in the $i^{\text{th}}$ row and the $j^{\text{th}}$ column represents the orientation $A_i\sqcup S_j$ for $i\in[p, q]$ and $j\in [2^{\ell}]$.
A node is drawn with a red star $\textcolor{red}{\star}$ if $A_i(u, v) = 1$ or $S_j(u, v) = 1$;
it is drawn with a blue circle $\textcolor{blue}{\circ}$ if $A_i(u, v) = -1$ or $S_j(u, v) = -1$;
otherwise, it is drawn with a black bullet $\bullet$.
Hence, $A_i\sqcup S_j$ is an acyclic orientation of $\oplus_{u, v}^\ell G$ if and only if its associated node is not both a red star $\textcolor{red}{\star}$ and a blue circle $\textcolor{blue}{\circ}$.
As $\b A = (A_p, \dots, A_q)$ is a path in $\AO$, and $(S_1, \dots, S_{2^{\ell}})$ is a cycle in $\AO[P_\ell]$, the graph $\M$ contains (at least) edges between vertically adjacent nodes, and edges between horizontally adjacent nodes, and edges between the rightmost and the leftmost nodes of each row (which we represent by a dotted line).

We have chosen and oriented some of these edges in order to construct a replacing pattern $\b P$.
A vertical arrow coming from above the first row indicates the starting node of $\b P$, while a vertical arrow going below the last row indicates its ending node.
To prove that this pattern is $0$-even, it is enough to show that each (maximal) path of consecutive black bullets $\bullet$ contains an even number of such black bullets $\bullet$.
This is ensured by the fact that $\mu$ is even.
To determine $\omega_p^{\b P}$ (resp. $\omega_q^{\b P}$), one follows along the directed path $\b P$ and each time one sees a new node of the first (resp. last) row, one records the number of its column.
To determine $\cons_p(\b P)$ (resp. $\cons_q(\b P)$), one follows along the nodes of the first (resp. last) row in the order given by $\omega_p^{\b P}$ (resp. $\omega_q^{\b P}$), \ie left-to-right starting from the first column, and one records which node has an edge going horizontally to its right.

\begin{figure}[h!]
    \centering
    \includegraphics[width=0.8\linewidth]{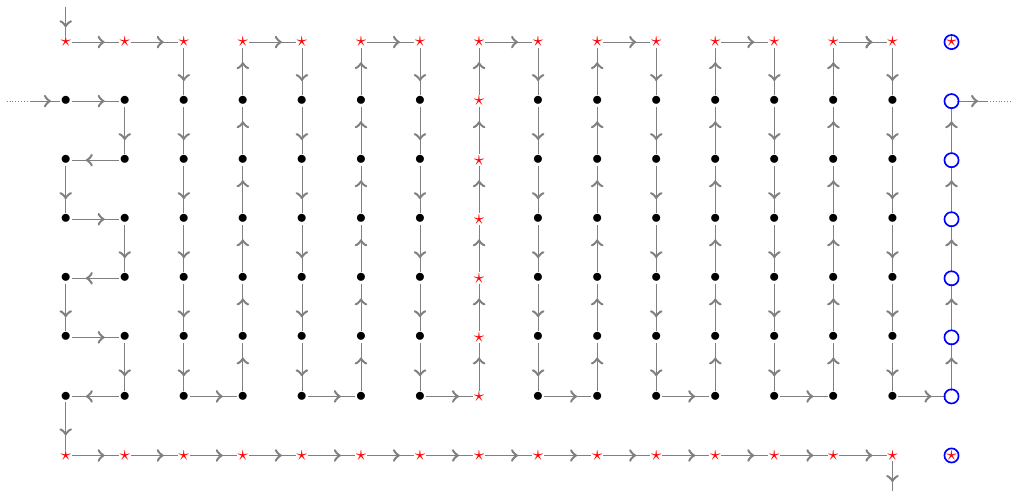}
    \caption{Replacing pattern for a $(1\to1)$-subsequence.
    }
    \label{fig:11ReplacingEven}
\end{figure}

For a $(1\to-1)$-subsequence, the drawing is given in \Cref{fig:1-1ReplacingEven}, with the same conventions.
Note that, here, $\cons_q(\b P)$ is determined following along the nodes of the last row from left to right starting from the $(\mu+1)^{\text{th}}$ column, and recording which node has an edge going horizontally to its right.

\begin{figure}[h!]
    \centering
    \includegraphics[width=0.8\linewidth]{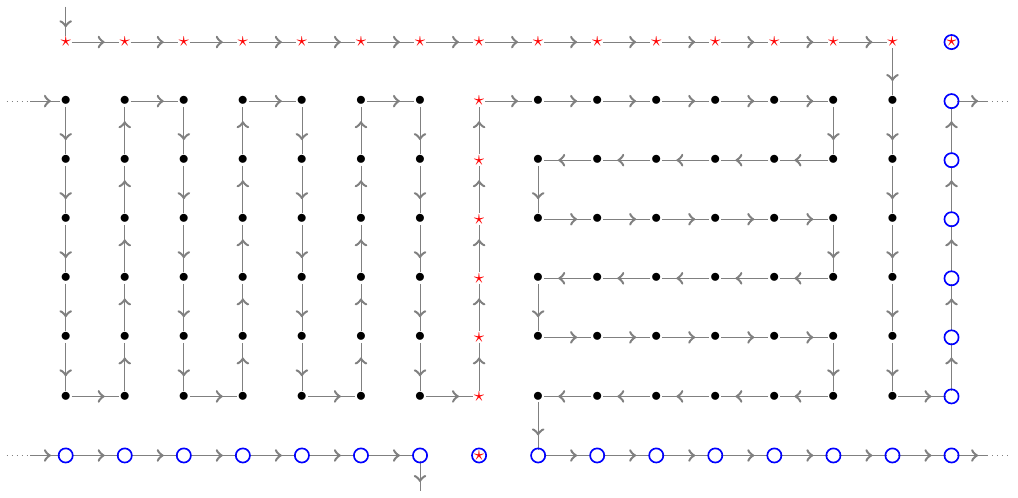}
    \caption{Replacing pattern for a $(1\to-1)$-subsequence.}
    \label{fig:1-1ReplacingEven}
\end{figure}

For degenerate $(1\to1)$-subsequences and degenerate $(1\to-1)$-subsequences, the drawings are given in \Cref{fig:Deg11ReplacingEven,fig:Deg1-1ReplacingEven}, with the same conventions.
\begin{figure}[h!]
    \centering
    \includegraphics[width=0.8\linewidth]{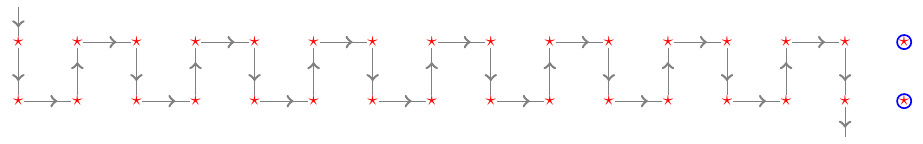}
    \caption{Replacing pattern for a degenerate $(1\to1)$-subsequence.}
    \label{fig:Deg11ReplacingEven}
\end{figure}

\begin{figure}[h!]
    \centering
    \includegraphics[width=0.8\linewidth]{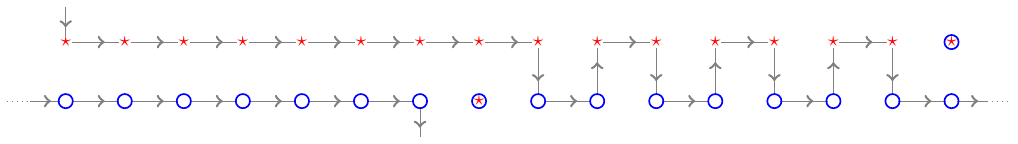}
    \caption{Replacing pattern for a degenerate $(1\to-1)$-subsequence.}
    \label{fig:Deg1-1ReplacingEven}
\end{figure}

The drawings for the (degenerate) $(-1 \to -1)$-subsequences and for the (degenerate) $(-1\to 1)$-subsequences can be obtained from the ones of \Cref{fig:11ReplacingEven,fig:1-1ReplacingEven,fig:Deg11ReplacingEven,fig:Deg1-1ReplacingEven} by swapping red stars $\textcolor{red}{\star}$ and blue circles $\textcolor{blue}{\circ}$.
\end{proof}

\begin{proof}[Proof of \Cref{thm:PathGluingEvenIsGood}]\label{proof:MainThm}
Fix a Hamiltonian cycle $\b A = (A_1, \dots, A_s)$ of $\AO$ with $A_1 = A_s$ such that $\b A(u, v)$ is $0$-even.

On the one hand, we suppose that there exists $j\in [s]$ such that $A_j(u, v) = 0$.
Without loss of generality, we can consider that $A_1(u, v) = 1$ and $A_{s-1}(u, v) = 0$.
Let $p_1, \dots, p_r$ be such that the $(a\to b)$-subsequences of $\b A$ for all possible $a, b\in\{-1, +1\}$ are $(A_{p_i}, A_{p_i+1}, \dots, A_{p_{i+1}})$ for $i\in [r]$ (in particular $p_1 = 1 = p_{r+1}$).

Fix $i\in [r]$.
Let $\b P_i$ be a replacing pattern for $(A_{p_i}, A_{p_i+1}, \dots, A_{p_{i+1}})$ given by \Cref{prop:ReplacingPatternsViaDrawings}.
If $A_{p_{i+1}}(u, v) = 1$ (resp. $A_{p_{i+1}}(u, v) = -1$), then
by~\eqref{item:omegaReplacingEvenPath} we get $\omega_{p_{i+1}}^{\b P_i} = \omega_{p_{i+1}}^{\b P_{i+1}}$, and
by~\eqref{item:consReplacingEvenPath} we get $\cons_{p_{i+1}}(\b P_i) \cup \cons_{p_{i+1}}(\b P_{i+1}) = [2^\ell]\ssm\{2^\ell\}$ (resp. $\cons_{p_{i+1}}(\b P_i) \cup \cons_{p_{i+1}}(\b P_{i+1}) = [2^\ell]\ssm\{\mu\}$).

Thus by \Cref{prop:LacingIsHamiltonianPath}, $\b P_i$ and $\b P_{i+1}$ can lace for all $i\in [r]$.
Thus $\b B \eqdef \b P_1 \wedge \dots \wedge \b P_r$ is a replacing pattern for $\b A$.
It remains to prove that $\b B$ can self-lace.
As we have assumed that $A_1(u, v) = 1$ and $A_{s-1}(u, v) = 0$, the last $(a\to1)$-subsequence $(A_{p_{r-1}}, \dots, A_{p_r})$ is not degenerate.
Thus by~\eqref{item:CyclicConditionReplacingEven}, we have $\cons_s(\b B) = [2^{\ell}]\ssm\{2^{\ell}\}$.
Combining this with $\omega_s^{\b B} = 1\, 2\, \dots \, 2^{\ell}-1$ proves that $\b B$ can self-lace.
By \Cref{cor:SelfLacingIntoAnAOcycle}, we get that $\b B$ is a Hamiltonian cycle for $\AO[\oplus_{u, v}^\ell G]$.

\medskip

On the other hand, assume that $A_j(u, v) \ne 0$ for all $j\in [s]$ (equivalently, that $G$ contains the edge between $u$ and $v$).
Suppose that there are two consecutive $A_i(u, v)$ and $A_{i+1}(u, v)$ of the same sign, so we can assume that $A_1(u, v) = 1$ and $A_{s-2}(u, v) = A_{s-1}(u, v) = -1$.

The reasoning is very similar, as consecutive (degenerate) $(a\to b)$-subsequences can lace according to \Cref{prop:ReplacingPatternsViaDrawings}.
However, we need to modify the last step of the chosen replacing patterns so that the resulting $\b B$ can self-lace and yield a Hamiltonian cycle for $\AO[\oplus_{u, v}^\ell G]$.
We swap the last replacing pattern for $(A_{s-2}, A_{s-1})$ by the one drawn in \Cref{fig:Alt-1-1DegReplacingEvenPath} where the last vertical edge connects $A_{s-1}\sqcup S_1$ to $A_1\sqcup S_1$, closing the Hamiltonian cycle of $\AO[\oplus_{u, v}^\ell G]$.
\begin{figure}[h]
    \centering
    \includegraphics[width=0.8\linewidth]{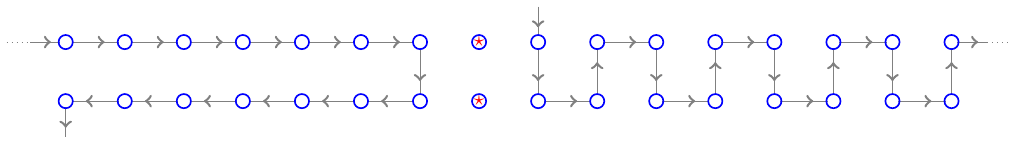}
    \caption{Alternative replacing pattern for a degenerate $(-1\!\to\!-1)$-subsequence.}
    \label{fig:Alt-1-1DegReplacingEvenPath}
\end{figure}

Otherwise, $\b A(u, v)$ alternates between $+1$ and $-1$ say $A_{2k}(u, v) = 1$ and $A_{2k+1}(u, v) = -1$, then we can replace each $(A_{2k}, A_{2k+1})$ by the replacing pattern of \Cref{fig:Alt1-1DegReplacingEvenPath} to directly get a cycle in $\AO[\oplus_{u, v}^\ell G]$ (no lacing is needed).
\begin{figure}[h]
    \centering
    \includegraphics[width=0.8\linewidth]{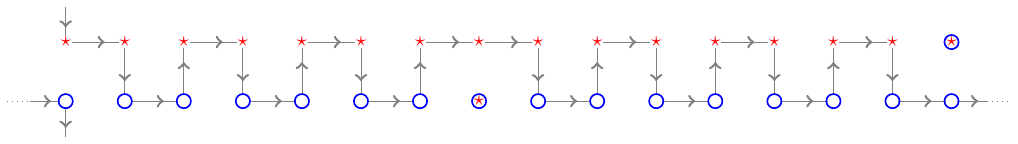}
    \caption{Alternative replacing pattern for a degenerate $(1\to-1)$-subsequence.}
    \label{fig:Alt1-1DegReplacingEvenPath}
\end{figure}
\end{proof}

\subsubsection{Gluing two paths of odd lengths}\label{sssec:DoublePathGluingOddLength}

The main objective of this section is to prove \Cref{thm:PathGluing2OddIsGood}.
As before we state it here, and delay the proof to page~\pageref{proof:MainThm2}.
The proof is very similar to \Cref{thm:PathGluingEvenIsGood}, except that it is conceptually easier since no lacing is needed; but the sketches are more intricate since the replacing patterns are larger.

\begin{theorem}\label{thm:PathGluing2OddIsGood}
Let $G = (\V, E)$ be a {\good} graph with $\b A = (A_1, \dots, A_s)$ a Hamiltonian cycle in $\AO$.
For $u, v\in \V$ in the same connected component of $G$ and two odd numbers $\ell_1, \ell_2\geq 2$, if $\b A(u, v)$ is $0$-even, then $\oplus_{u, v}^{\ell_1, \ell_2} G$ is {\good}.
Moreover, under these assumptions, there exists a Hamiltonian cycle $\b B = (B_1, \dots, B_t)$ in $\AO[\oplus_{u, v}^{\ell_1, \ell_2} G]$ such that $\b B(u, v)$ is $0$-even.
\end{theorem}

Our proof will explicitly construct a new Hamiltonian cycle $\b B$ for $\AO[\oplus_{u, v}^{\ell_1, \ell_2} G]$ with $\ell_1, \ell_2$ odd, from any Hamiltonian cycle $\b A$ of $\AO$ (such that $\b A(u, v)$ is $0$-even) and almost any Hamiltonian cycle of the cube-graphs $Q_{\ell_1}$ and $Q_{\ell_2}$.

For the cube-graph $Q_{\ell}$ for $\ell$ odd, fix a Hamiltonian cycle $\mathdefn{\b S} \eqdef (S_1, \dots, S_{2^\ell})$ of $Q_\ell \simeq \AO[P_\ell]$ (see \Cref{exm:CubeGraphIsGood} for its existence), where we interpret each $S_p$ as a subset of $E(P_\ell)$ (instead of a subset of $[\ell]$).
Without loss of generality, we can suppose that $\mathdefn{S_{2^{\ell}} = [\ell] \simeq E(P_\ell)}$, \ie $S_{2^{\ell}}(u, v) = -1$; and we denote $\mathdefn{\mu}\in [2^\ell]$ such that $\mathdefn{S_\mu = \emptyset}$, \ie $S_\mu(u, v) = 1$.
Note that $Q_\ell$ is bipartite:
$\emptyset$ and $E(P_\ell)$ do not belong to the same part if and only if $\ell$ is odd.
Consequently, as we have supposed $\ell$ odd, this implies that $\mu$ is odd.
Note also that $\mu\notin\{1, 2^{\ell}-1\}$ since $\emptyset$ and $[2^\ell]$ are not adjacent in the cube-graph $Q_{\ell}$.

For given odd $\ell_1, \ell_2$, we denote $\b S' = (S'_1, \dots, S'_{\mu_1}, \dots, S'_{2^{\ell_1}})$ the chosen Hamiltonian cycle of $Q_{\ell_1}$, and denote $\b S'' = (S''_2, \dots, S''_{\mu_2}, \dots, S''_{2^{\ell_2}})$ the chosen Hamiltonian cycle of $Q_{\ell_2}$, where $\mu_1$ and $\mu_2$ are odd with $\mu_1, \mu_2\notin\{1, 2^{\ell}-1\}$.

\begin{proposition}\label{prop:GluingTwoOddPathsReplacingPatterns}
Let $G = (\V, E)$ be a {\good} graph with two nodes $u, v$ in the same connected component of $G$, and let $\ell_1, \ell_2\geq 3$ be odd numbers.
We denote $A \sqcup A' \sqcup A''$ an orientation of $\oplus_{u, v}^{\ell_1, \ell_2} G$ where $A\in \AO$, $A'\in \AO[P_{\ell_1}]$ and $A''\in \AO[P_{\ell_2}]$.
\begin{compactenum}
\item If $A(u, v) = +1$ (resp. $A(u, v) = -1$), there exists a $0$-even replacing pattern\footnote{Technically, we should write $\M[(A)]$, but we abuse notation for readability and prefer $\M[A]$ here.} of $\M[A]$ that starts at $A \sqcup S'_{\mu_1+1} \sqcup S''_{\mu_2+1}$ and ends at $A \sqcup S'_{2} \sqcup S''_{2^{\ell_2}-1}$, and\label{item:2OddReplacing1-1}
\item If $A_1(u, v) = A_2(u, v) = 0$ and $A_1$ and $A_2$ are adjacent in $\AO$, then there exists a $0$-even replacing pattern of $\M[A_1, A_2]$ that starts at $A_1 \sqcup S'_{\mu_1+1} \sqcup S''_{\mu_2-1}$ and ends at $A_2 \sqcup S'_{\mu_1+1} \sqcup S''_{\mu_2-1}$, and there exists a $0$-even replacing pattern of $\M[A_1, A_2]$ that starts at $A_1 \sqcup S'_{2} \sqcup S''_{2^{\ell_2}-1}$ and ends at $A_2 \sqcup S'_{2} \sqcup S''_{2^{\ell_2}-1}$.\label{item:2OddReplacing00}
\end{compactenum}
\end{proposition}

\begin{proof}
The proof is by sketching the relevant replacing patterns. 

\eqref{item:2OddReplacing1-1}
For $A(u, v) = +1$, \Cref{fig:1And-1DoubleODDPathGluing} (left) illustrates a replacing pattern of $\M[A]$ with the claimed properties.
The node in the $i^{\text{th}}$ row and the $j^{\text{th}}$ column represents the orientation $A \sqcup S'_i \sqcup S''_j$ for $i \in [2^{\ell_1}]$ and $j\in [2^{\ell_2}]$.
The drawing conventions (red star $\textcolor{red}{\star}$, blue circle $\textcolor{blue}{\circ}$, black dot $\bullet$) are the same as in \Cref{fig:11ReplacingEven,fig:1-1ReplacingEven,fig:Deg11ReplacingEven,fig:Deg1-1ReplacingEven,fig:Alt-1-1DegReplacingEvenPath,fig:Alt1-1DegReplacingEvenPath}:
As $A(u, v) = +1$ here, all nodes are red stars $\textcolor{red}{\star}$.
As $\ell_1, \ell_2 \geq 3$, none of the drawn ``even'' number is $0$.
In particular, $\M[A]$ contains (at least) edges between vertically (resp. horizontally) adjacent nodes, and edges between rightmost (resp. topmost) and leftmost (resp. bottom-most) node of each row (resp. column).

We have chosen and oriented some of these edges in order to construct a replacing pattern $\b P$ of $(A)$.
Note that this pattern is $0$-even since each maximal sub-path of black dots $\bullet$ is of even length (since there are no black dots $\bullet$).
Moreover, $\b P$ starts at $A \sqcup S'_{\mu_1+1} \sqcup S''_{\mu_2-1}$ and ends at $A\sqcup S'_{2} \sqcup S''_{2^{\ell_2}-1}$ as desired.

For $A(u, v) = -1$, the drawing is given in \Cref{fig:1And-1DoubleODDPathGluing} (right):
As $A(u, v) = -1$ here, all nodes are blue circles $\textcolor{blue}{\circ}$.

\begin{figure}[h]
    \centering
    \includegraphics[width=\linewidth]{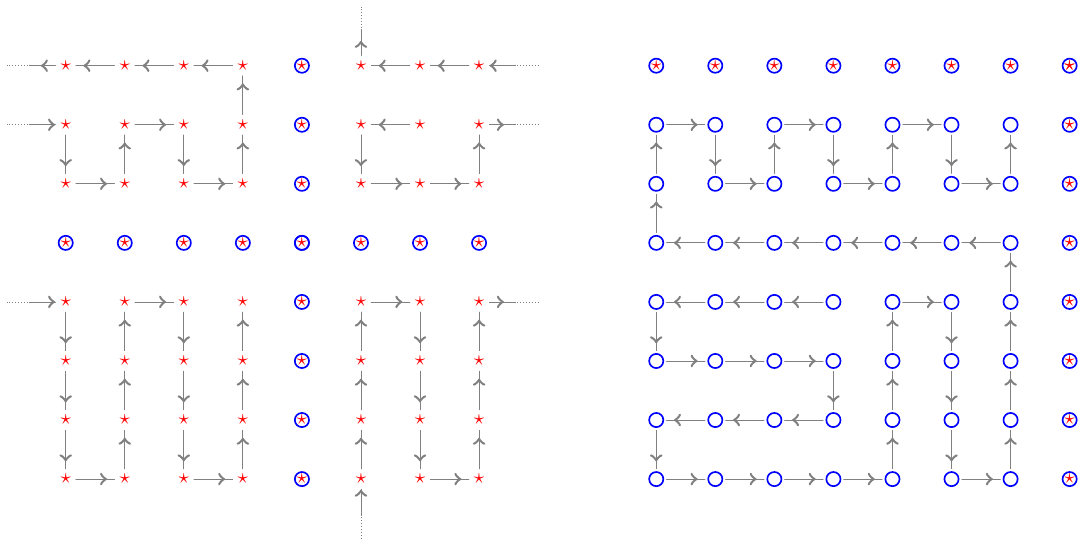}
    \caption{(Left) Hamiltonian path from $A \sqcup S'_{\mu_1+1} \sqcup S''_{\mu_2-1}$ to $A \sqcup S_{2} \sqcup S_{2^{\ell_2}-1}$ in the case $A(u, v) = +1$.
    (Right) Same in the case $A(u, v) = -1$.}
    \label{fig:1And-1DoubleODDPathGluing}
\end{figure}

\medskip

\eqref{item:2OddReplacing00}
For $A_1(u, v) = A_2(u, v) = 0$, \Cref{fig:00DoubleOddPathGluing} (left) illustrates a replacing pattern for $(A_1, A_2)$.
The drawing should be interpreted as showing $8$ rows and $8$ columns where row $i$ column $j$ contains two nodes $v_{i, j, 1}$ and $v_{i, j, 2}$ (the first being on the near lower-right of the second):
$v_{i, j, 1}$ represents $A_1 \sqcup S'_i\sqcup S''_j$ while $v_{i, j, 2}$ represents $A_2\sqcup S'_i \sqcup S''_j$.
Hence, in $\M[A_1, A_2]$, there are edges between vertically and horizontally adjacent nodes (and looping right-to-left or top-to-bottom), but also between $v_{i, j, 1}$ and $v_{i, j, 2}$ for all $i\in[2^{\ell_1}]$, $j\in[2^{\ell_2}]$.

We have chosen and oriented some of these edges in order to construct a replacing pattern $\b P$ of $\M[A_1, A_2]$.
As, in \Cref{fig:00DoubleOddPathGluing} (left), the part of the path on the top-right may not be easily generalized, we detail it in \Cref{fig:00DoubleOddPathGluing} (right):
that is we draw a $0$-even replacing pattern for the subgraph induced on $A_k \sqcup S'_i \sqcup S''_j$ with $k\in \{1, 2\}$, $i\in [1, \mu_1 - 1]$ and $j\in[\mu_2+1, 2^{\ell_2}]$ which starts at $A_1\sqcup S'_2 \sqcup S''_{\mu_2+1}$ and ends at $A_2 \sqcup S'_2\sqcup S''_{\mu_2+1}$.
If the reader wonders how to generalize the part $A_k\sqcup S'_i\sqcup S''_j$ for $k\in \{1, 2\}$, $i\in [\mu_1+1, 2^{\ell_1}]$ and $j\in \{1\} \cup [\mu_2, 2^{\ell_2}]$, especially if $\mu_1$ is congruent to $3$ modulo $4$, then he or she just need to subdivide the vertical arcs between the two bottom rows.
Note that the replacing pattern $\b P$ is $0$-even, and starts at $A_1 \sqcup S'_{\mu_1+1} \sqcup S''_{\mu_2-1}$ and ends at $A_2 \sqcup S'_{\mu_1+1} \sqcup S''_{\mu_2-1}$.
A symmetric drawing yields a $0$-even replacing pattern that starts at $A_1 \sqcup S'_{2} \sqcup S''_{2^{\ell_2}-1}$ and ends at $A_2 \sqcup S'_{2} \sqcup S''_{2^{\ell_2}-1}$.
\begin{figure}[h]
    \centering
    \includegraphics[width=\linewidth]{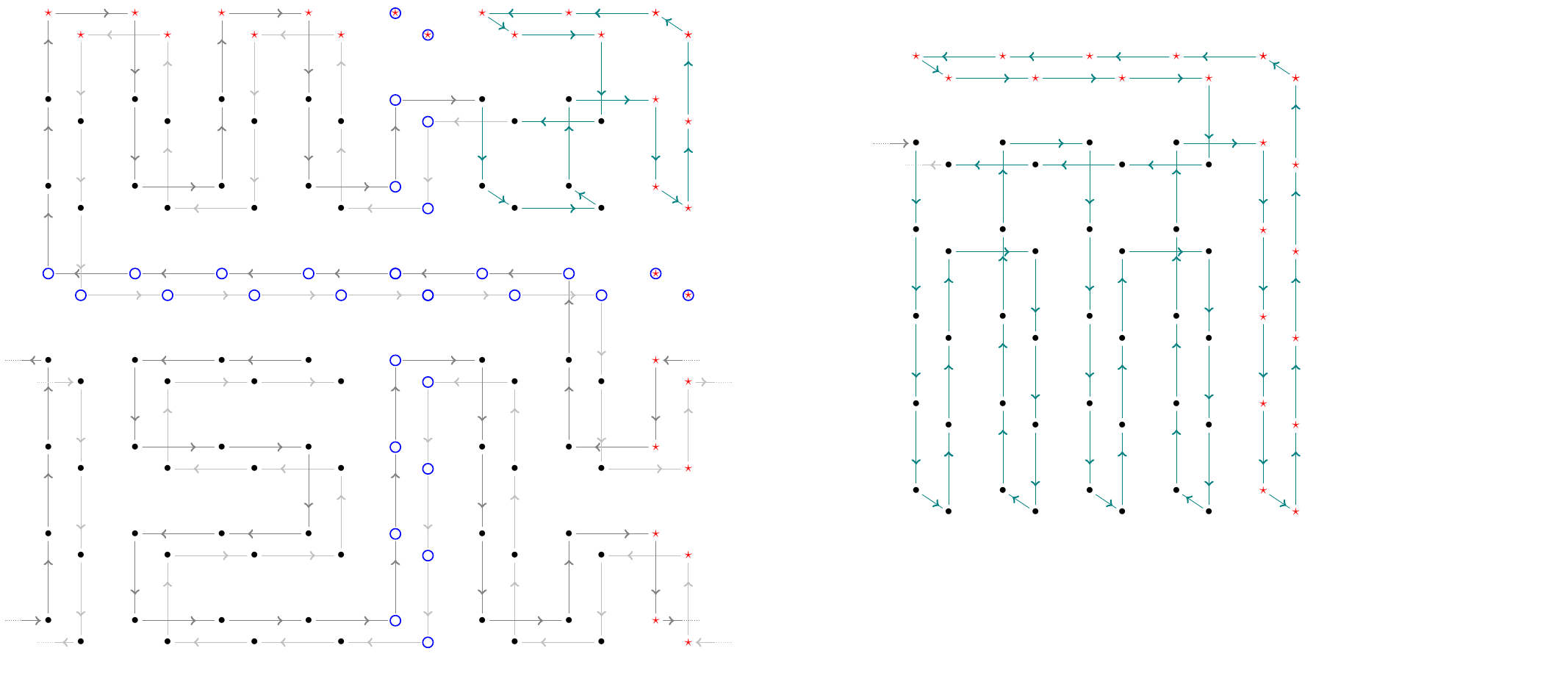}
    \caption{(Left) Hamiltonian path from $A \sqcup S'_{\mu_1+1} \sqcup S''_{\mu_2-1}$ to $A' \sqcup S'_{\mu_1+1} \sqcup S''_{\mu_2+1}$ in the case $A(u, v) = A'(u, v) = 0$.
    (Right) Details of the switching pattern in the general case.}
    \label{fig:00DoubleOddPathGluing}
\end{figure}
\end{proof}

\begin{proof}[Proof of \Cref{thm:PathGluing2OddIsGood}]\label{proof:MainThm2}
For odd $\ell_1, \ell_2 \geq 3$, this is a direct consequence of \Cref{thm:ConcatenatingReplacingPatternsWithSharedEndPoints,prop:GluingTwoOddPathsReplacingPatterns} using the splitting $i_{j+1} = i_j$ if $A_{i_j}(u, v) \ne 0$, and $i_{j+1} = i_j + 1$ if $A_{i_j} = 0$.
As the number of acyclic orientations $A\in \AO$ with $A(u, v) = +1$ is equal to the number of acyclic orientations $A\in \AO$ with $A(u, v) = -1$ by \Cref{lem:CompletementIsInvolution}~\eqref{item:CompletementIsInvolution:BetweenNon0Acyclic}, the quantity $\#\{j\in [r] ~;~ A'_j \ne B'_j\}$ of \Cref{thm:ConcatenatingReplacingPatternsWithSharedEndPoints} is even.
Thus, we get a Hamiltonian cycle in $\AO[\oplus_{u, v}^{\ell_1, \ell_2} G]$.
The $0$-even-ness of this cycle is ensured by the $0$-even-ness of each replacing pattern of \Cref{prop:GluingTwoOddPathsReplacingPatterns}.

\medskip

The remaining case is $\ell_1 = 1$ and $\ell_2 \geq 3$ is odd.
For readability, we call $S'_+$ and $S'_-$ the two orientations of $P_1$, where $S'_+(u, v) = +1$ and $S'_-(u, v) = -1$.
Also, we rename $\ell_2$ into simply $\ell$, and $\mu_2$ into $\mu$.
The proof is mainly via the drawings of \Cref{fig:Degenerate00DoubleOddPathGluing} where the conventions are the same as in \Cref{fig:1And-1DoubleODDPathGluing,fig:00DoubleOddPathGluing}.
We cannot directly apply \Cref{thm:ConcatenatingReplacingPatternsWithSharedEndPoints}, but we can concatenate paths joining four families of orientations, namely $A \sqcup S_+' \sqcup S''_{\mu-1}$, and $A \sqcup S_+' \sqcup S''_{\mu+1}$, and $A \sqcup S_-' \sqcup S''_{1}$, and $A \sqcup S_-' \sqcup S''_{2^{\ell}-1}$ for $A \in \AO$.

We again use the splitting subsequences of $\b A$ defined by $i_{j+1} = i_j$ if $A_{i_j}(u, v) \ne 0$, and $i_{j+1} = i_j + 1$ if $A_{i_j}(u, v) = 0$.
Since we are gluing the edge $uv$, we are sure that $uv\notin E(G)$, so we can consider that $A_1(u, v) = +1$ and $A_s(u, v) = 0$ without loss of generality.
By \Cref{fig:Degenerate00DoubleOddPathGluing} (bottom left \& right) and symmetries, if $A_{i_j}(u, v) = A_{i_{j+1}}(u, v) = 0$ there exist replacing patterns for $(A_{i_j}, A_{i_{j+1}})$ that start at $A_{i_j} \sqcup S'_\varepsilon \sqcup S''_{x}$ and ends at $A_{i_{j+1}} \sqcup S'_\eta \sqcup S''_{y}$ for all $\varepsilon,\eta\in \{-, +\}$ and all possible $x, y \in\{\mu-1, \mu+1, 1, 2^\ell-1\}$.
On the other hand, by \Cref{fig:Degenerate00DoubleOddPathGluing} (top left \& right), if $A_{i_j}(u, v) \ne 0$, then there exists a replacing pattern for $(A_{i_j})$ that starts at $A_{i_j} \sqcup S'_\varepsilon \sqcup S''_{x}$ and ends at $A_{i_{j+1}} \sqcup S'_\eta \sqcup S''_{y}$ for all possible $\varepsilon,\eta\in \{-, +\}$ and all possible $x, y \in\{\mu-1, \mu+1, 1, 2^\ell-1\}$.
The concatenation of these paths immediately yields a Hamiltonian path in $\AO$.
Moreover, we can choose the replacing pattern of $A_1$ to start at $A_1 \sqcup S'_+\sqcup S''_{\mu-1}$, and choose the replacing pattern of $A_s$ to end at $A_s \sqcup S'_+\sqcup S''_{\mu-1}$:
Thus, $\AO$ admits a Hamiltonian cycle.
\begin{figure}[h]
    \centering
    \includegraphics[width=\linewidth]{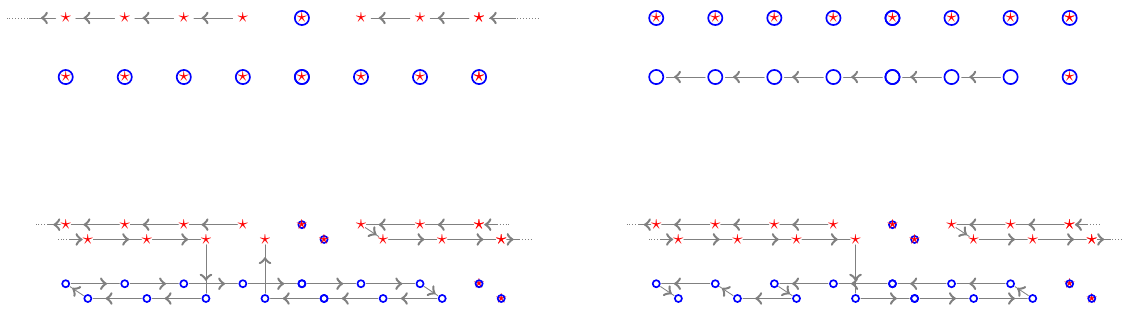}
    \caption{(Top left) Hamiltonian path from $A\sqcup S'_+ \sqcup S''_{\mu-1}$ to $A\sqcup S'_+\sqcup S_{\mu +1}$ in the case $A(u, v) = +1$.
    (Top right) Hamiltonian path from $A\sqcup S'_- \sqcup S''_{2^{\ell}-1}$ to $A\sqcup S'_-\sqcup S''_{1}$ in the case $A(u, v) = -1$.
    (Bottom left) Hamiltonian path from $A_1\sqcup S'_+ \sqcup S''_{\mu-1}$ to $A_2\sqcup S'_+\sqcup S''_{\mu -1}$ in the case $A_1(u, v) = A_2(u, v) = 0$.
    (Bottom right) Hamiltonian path from $A_1\sqcup S'_+ \sqcup S''_{\mu-1}$ to $A_2\sqcup S'_-\sqcup S''_{1}$ in the case $A_1(u, v) = A_2(u, v) = 0$.}
    \label{fig:Degenerate00DoubleOddPathGluing}
\end{figure}
\end{proof}

\subsubsection{Gluing a multipath of even total length}\label{sssec:MultiPathGluingEvenTotalLength}

Using \Cref{thm:PathGluingEvenIsGood,thm:PathGluing2OddIsGood}, we can prove the main result of this section.

\begin{corollary}\label{cor:MultiPathGluingIsGood}
Let $G = (\V, E)$ be a {\good} graph with $\b A = (A_1, \dots, A_s)$ a Hamiltonian cycle in $\AO$.
For $u, v\in \V$ in the same connected component of $G$ and integers $\b\ell = (\ell_1, \dots, \ell_m)$ for $m \geq 1$, if:
\begin{compactitem}
\item $\sum_{i = 1}^m \ell_i$ is even (\ie the number of odd $\ell_i$s is even), and
\item $\b A(u, v)$ is $0$-even, 
\end{compactitem}
then $\oplus_{u, v}^{\b\ell} G$ is {\good}.

Moreover, under these assumptions, there exists a Hamiltonian cycle $\b B = (B_1, \dots, B_t)$ in $\AO[\oplus_{u, v}^{\b\ell} G]$ such that $\b B(u, v)$ is $0$-even.
\end{corollary}

\begin{proof}
Inductive application of \Cref{thm:PathGluingEvenIsGood,thm:PathGluing2OddIsGood}:
since each Theorem yields a Hamiltonian cycle whose sequence of directions on $(u, v)$ is $0$-even, the induction runs.
\end{proof}

\Cref{cor:MultiPathGluingIsGood} has several interesting consequences. First of all, the $0$-even-ness can be dropped if $uv$ is an edge of $G$, since it is impossible that $A(u, v) = 0$ in this case.

\begin{corollary}\label{cor:MultiPathGluingOnArcIsGood}
Let $G = (\V, E)$ be a {\good} graph, and let $uv\in E$.
For numbers $\b\ell = (\ell_1, \dots, \ell_s)$ if $\sum_{p=1}^s \ell_p$ is odd, then $\oplus_{u, v}^{\b\ell} G$ is {\good}.
\end{corollary}

\begin{proof}
As $uv\in E$, for any acyclic orientation $A$ of $G$, there is a directed path from $u$ to $v$ or from $v$ to $u$, \ie $A(u, v) \ne 0$.
Hence, for any Hamiltonian cycle $\b A = (A_1, \dots, A_r)$ of $\AO$, the sequence $\b A(u, v)$ is $0$-even.
Applying \Cref{cor:MultiPathGluingIsGood} yields the claim.
\end{proof}

Secondly, since a single path is {\good}, we can finish our characterization of {\good} multipaths, proving \Cref{thmA}.

\begin{corollary}\label{cor:MultiPathIsGood}
The $(\ell_1, \dots, \ell_s)$-multipath is {\good} if and only if $\sum_p\ell_p$ is odd.
\end{corollary}

\begin{proof}
The ``only if'' direction is proven in \Cref{cor:AOhamiltonianMultiPathImpliesSumLengthsEven}.
For the ``if'' direction, consider the graph $G_\circ = \bigl(\{u, v\}, \emptyset\bigr)$ composed of 2 isolated \nodes: Its $\AO$ is composed of a single \node.
Even if one might not consider this $\AO$ to be Hamiltonian, is it enough for applying \Cref{cor:MultiPathGluingIsGood}: if $\sum_p \ell_p$ is odd, then the $(\ell_1, \dots, \ell_s)$-multipath $\oplus_{u, v}^{\ell_1, \dots, \ell_s} G_\circ$ is {\good}.
\end{proof}

Thirdly, we can characterize {\good} graphs among certain families of graphs:

\begin{example}\label{exm:NewAOhamiltonianGraphs}
As an immediate consequence of \Cref{cor:MultiPathGluingOnArcIsGood,cor:MultiPathGluingIsGood,cor:MultiPathIsGood}, we get:
\begin{enumerate}[(i)]
\item As proven in \cite[Lem.~3.1 \& Thm.~4.2]{SavageSquireWest93-GrayCodesAcyclicOrientations},
the cycle graph $C_n$ is {\good} if and only if $n$ is odd: Indeed $C_n$ is isomorphic to $(n-1,\, 1)$-multipath.
\item Any $2$-sum of two cycle graphs $C_n$ and $C_m$ is {\good} if and only if $n + m$ is even: Indeed $C_n\oplus_2 C_m$ is isomorphic to the $(n-1,\, m-1,\, 1)$-multipath.
\item For a tree $T = ([s], E)$ and numbers $n_1, \dots, n_s \geq 3$, a \defn{$(T,\, n_1, \dots, n_s)$-tree of cycles} is any graph isomorphic to a $2$-sum $H \oplus_2 C_{n_s}$ where $H$ is a $(T\ssm s,\, n_1, \dots, n_{s-1})$-tree of cycles (of course for $s = 1$, a $(T,\, n_1)$-tree of cycle is just a cycle graph $C_{n_1}$).
By consecutive applications of \Cref{cor:MultiPathGluingOnArcIsGood}, one shows that if $n_p$ is odd for all $p\in [s]$, then any $(T,\, n_1, \dots, n_s)$-tree of cycles is {\good}, for any tree $T$.
\item If $uv\in E$ is an edge of $G = (\V, E)$, and $C_n$ is a cycle graph, then the $2$-sum $G\oplus_2 C_n$ is {\good} if and only if $n$ is odd: Indeed this is the case $s = 1$ of \Cref{cor:MultiPathGluingOnArcIsGood}.
\end{enumerate}
\end{example}

% \clearpage

\subsection{Adding a simplicial node and 0-even sequences}\label{ssec:AddingSimpliceNode}

Recall that a node $x$ in a graph $G$ is simplicial if the neighbors of $x$ form a clique, that is if every pair of neighbors of $x$ share an edge in $G$.
If $K$ is a clique in $G$, the \defn{simplicial augmentation} above $K$, denoted $\mathdefn{\odot_K G}$, is the graph obtained by adding a new (simplicial) node to $G$ connected to all the nodes of $K$, \ie $\odot_K G \eqdef G \oplus_n K_{n+1}$ where $n = |K|$ and $K_{n+1}$ is the clique on $[n+1]$.

%The authors \cite{SavageSquireWest93-GrayCodesAcyclicOrientations} proved that chordal graph are {\good}.
%Indeed, it is well-known that chordal graphs can be built from a single node by subsequently adding simplicial nodes.
%The ordering thus constructed is a so-called \defn{perfect elimination ordering} of the chordal graph.
In the details of the proof of \cite[Theorem~4.1]{SavageSquireWest93-GrayCodesAcyclicOrientations} that chordal graphs are {\good}, they prove the first sentence of the upcoming \Cref{thm:AddingSimplicialNodePreservesAOHamiltonian}.
We use the Steinhaus--Johnson--Trotter algorithm (see \Cref{exm:PermutahedralGraphIsGood}) in a similar way but we keep track of $0$-evenness:

\begin{theorem}\label{thm:AddingSimplicialNodePreservesAOHamiltonian}
If $G$ is {\good}, and $K$ is a clique of $G$, then $\odot_K G$ is {\good}.
Moreover, for any nodes $u, v$ of $G$, if $\b A$ is a Hamiltonian cycle of $\AO$ and either $\b A(u, v)$ is $0$-even or if $|K|$ is odd, then $\AO[\odot_K G]$ admits a Hamiltonian cycle $\b B$ such that $\b B(u, v)$ is $0$-even.
\end{theorem}

\begin{proof}
Let $K$ be a clique of $G = (\V, E)$ with $n = |K|$, and let $v_\circ$ be the added node. %\ie $v_\circ \in \V(\odot_K G) \ssm \V$.
Since $K$ is a clique, any acyclic orientation $A$ of $G$ induces a total order on $K$, denoted $\sigma_A : K \to [n]$.
Reciprocally, to a total order $\tau$ on $K\sqcup\{v_\circ\}$, we associate the orientation $\c O(\tau)$ of the neighboring edges of $v_\circ$ where we orient from $v_\circ$ to $u\in K$ if $\tau(v_\circ) < \tau(u)$, and conversely from $u$ to $v_\circ$ if $\tau(u) < \tau(v_\circ)$.

As in \Cref{exm:PermutahedralGraphIsGood}, for $i\in [n+1]$, we denote $\sigma_A^{(i)}$ the total order on $K\sqcup\{v_\circ\}$ defined by
$$\sigma_A^{(i)}(u) = \left\{\begin{array}{ll}
\sigma_A(u) & \text{if } \sigma_A(u) < i \\
i & \text{if } u = v_\circ \\
\sigma_A(u) + 1 & \text{if } \sigma_A(u) > i
\end{array}\right. \qquad \text{ for } u \in K\sqcup\{v_\circ\}.$$

There is a bijection between acyclic orientations $A'$ of $\odot_K G$ and pairs $(A, i)$ with $A\in \AO$ and $i\in [n+1]$ given by $(A, i) \mapsto A\sqcup \c O\bigl(\sigma_A^{(i)}\bigr)$.
Indeed, the edges of $\odot_K G$ are bipartite between the edges of $G$ and the edges adjacent to $v_\circ$, and the orientation $A\sqcup A'$ of $\odot_K G$ (split according to this bipartition) is acyclic if and only if the total order induced by $A'$ on $K\sqcup\{v_\circ\}$ which agrees on $K$ with the total order induced by $A$ on $K$.
Moreover, $A\sqcup \c O\bigl(\sigma_A^{(i)}\bigr)$ is adjacent to $B\sqcup \c O\bigl(\sigma_B^{(j)}\bigr)$ in $\AO[\odot_K G]$ if and only if either $A$ is adjacent to $B$ in $\AO$ and $i = j$, either $A = B$ and $j \in \{i-1, i+1\}$.

Consequently, given a Hamiltonian cycle $\b A = (A_1, \dots, A_s)$ of $\AO$ (recall that $s$ is even by \Cref{lem:CompletementIsInvolution}), the following is mapped to a Hamiltonian cycle $\b B$ of $\AO[\odot_K G]$:
$$\bigl((A_1, n+1), \dots, (A_1, 1), (A_2, 1), \dots, (A_2, n+1), (A_3, n+1), \dots, (A_3, 1), \dots, (A_s, 1), \dots, (A_s, n+1)\bigr) \,.$$

Note that in $(A_k, i)$ there is no directed path between $u$ and $v$ if and only if there is no directed path between $u$ and $v$ in $A_k$.
Hence, by construction, the number of non-zero entries of $\b B(u, v)$ between two non-zero entries is always $n+1$ times the corresponding number for $\b A(u, v)$.
Thus, if $\b A(u, v)$ is $0$-even or if $n$ is odd, then $\b B(u, v)$.
\end{proof}

\begin{example}
Consider the complete graph $K_3$ and let $K = \{2, 3\}$ be a clique of $K_3$.
The graph $\odot_{23} K_3$ is drawn in \Cref{fig:ExampleAddingSimplicialNode}.
A Hamiltonian cycle of $\AO[K_3]$ is given by $\b A = (123, 132, 312, 321,\linebreak 231, 213)$ where $ijk$ correspond to the orientation of $K_3$ with arcs $i\to j$, $j\to k$ and $i\to k$.
The (generalized) Steinhaus--Johnson--Trotter presented in the proof of \Cref{thm:AddingSimplicialNodePreservesAOHamiltonian} yields a Hamiltonian cycle of $\AO[\odot_{23} K_3]$ shown in \Cref{tab:SJTgeneral}.
Note that all the sequences $\b B(u, v)$ are $0$-even for $u, v\in [4]$ except for $\{u, v\} = \{1, 4\}$.
\end{example}

\begin{figure}
    \centering
    \includegraphics[width=0.5\linewidth]{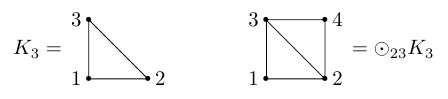}
    \caption[Simplicial augmentation]{(Left) The complete graph $K_3$, and (right) the simplicial augmentation of $K_3$ on its clique $\{2, 3\}$ (abbreviated as $\odot_{23}K_3$).}
    \label{fig:ExampleAddingSimplicialNode}
\end{figure}

\begin{table}[th!] %h!
    \centering 
    \begin{tabular}{cccc|cccc|cccc}
        
        $k$ & $i$ & $B_k^{(i)}$ & $B_k^{(i)}(1, 4)$ & $k$ & $i$ & $B_k^{(i)}$ & $B_k^{(i)}(1, 4)$ & $k$ & $i$ & $B_k^{(i)}$ & $B_k^{(i)}(1, 4)$ \\ 
        \midrule 
        
        1 & 3 & $1\textcolor{blue}{23}\textcolor{red}{\textbf{4}}$ & +1 & 3 & 3 & $\textcolor{blue}{3}1\textcolor{blue}{2}\textcolor{red}{\textbf{4}}$ & +1 & 5 & 3 & $\textcolor{blue}{2}\textcolor{blue}{3}1\textcolor{red}{\textbf{4}}$ & 0 \\
        
        1 & 2 & $1\textcolor{blue}{2}\textcolor{red}{\textbf{4}}\textcolor{blue}{3}$ & +1 & 3 & 2 & $\textcolor{blue}{3}1\textcolor{red}{\textbf{4}}\textcolor{blue}{2}$ & 0 & 5 & 2 & $\textcolor{blue}{2}\textcolor{red}{\textbf{4}}\textcolor{blue}{3}1$ & -1 \\
        
        1 & 1 & $1\textcolor{red}{\textbf{4}}\textcolor{blue}{2}\textcolor{blue}{3}$ & 0 & 3 & 1 & $\textcolor{red}{\textbf{4}}\textcolor{blue}{3}1\textcolor{blue}{2}$ & -1 & 5 & 1 & $\textcolor{red}{\textbf{4}}\textcolor{blue}{2}\textcolor{blue}{3}1$ & -1 \\

        2 & 1 & $1\textcolor{red}{\textbf{4}}\textcolor{blue}{3}\textcolor{blue}{2}$ & 0 & 4 & 1 & $\textcolor{red}{\textbf{4}}\textcolor{blue}{3}\textcolor{blue}{2}1$ & -1 & 6 & 1 & $\textcolor{red}{\textbf{4}}\textcolor{blue}{2}1\textcolor{blue}{3}$ & -1 \\
        
        2 & 2 & $1\textcolor{blue}{3}\textcolor{red}{\textbf{4}}\textcolor{blue}{2}$ & +1 & 4 & 2 & $\textcolor{blue}{3}\textcolor{red}{\textbf{4}}\textcolor{blue}{2}1$ & -1 & 6 & 2 & $\textcolor{blue}{2}1\textcolor{red}{\textbf{4}}\textcolor{blue}{3}$ & 0 \\
        
        2 & 3 & $1\textcolor{blue}{32}\textcolor{red}{\textbf{4}}$ & +1 & 4 & 3 & $\textcolor{blue}{32}1\textcolor{red}{\textbf{4}}$ & 0 & 6 & 3 & $\textcolor{blue}{2}1\textcolor{blue}{3}\textcolor{red}{\textbf{4}}$ & +1 \\
        
        \bottomrule  
    \end{tabular} 
    \caption{
    The Hamiltonian cycle of $\odot_{23}K_3$ obtained using the (generalized) Steinhaus--Johnson--Trotter algorithm where $B_k^{(i)} \eqdef A_k \sqcup \c O(\sigma_{A_k}^{(i)})$.
    There are $\bigl|\AO[K_3]\bigr|\cdot (|K|+1) = 6 \cdot 3 = 18$ acyclic orientations of $\odot_{23} K_3$.
    Each sequence $B_k^{(i)} = a_1a_2a_3a_4$ with $a_i\in [4]$ corresponds to the orientation $a_i \to a_j$ for $i < j$ and $\{a_i, a_j\} \ne \{1, 4\}$.
    The sequence $\b B(1, 4)$ is not $0$-even as $k = 3$, $i = 2$ indicates.
    }
    \label{tab:SJTgeneral} 
\end{table}

\begin{remark}
For $K$ a clique of the graph $G$, and any Hamiltonian cycle $\b B$ of $\AO[\odot_K G]$, note that $\b B(u, x_\circ)$ is $0$-even for all $u \in K$, where $x_\circ$ denotes the added simplicial node.
Indeed, there is an edge between $u$ and $x_\circ$ in $\odot_K G$, so $\b B(u, x_\circ)$ contains no $0$.
\end{remark}

We can combine \Cref{cor:MultiPathGluingIsGood,thm:AddingSimplicialNodePreservesAOHamiltonian} into a single corollary.
To this end, we introduce \defn{simplicial multi-augmentation} of a graph $G$ above a clique-sequence $\b K = (K_1, \dots, K_s)$, defined as $G$ if $s = 0$, otherwise as $\odot_{\b K} G \eqdef \odot_{K_s}\bigl(\odot_{(K_1, \dots, K_{s-1})} G\bigr)$ where $K_s$ is a clique of $\odot_{(K_1, \dots, K_{s-1})} G$.

We are going to construct a sequence of graphs where we alternate between gluing a multipath and adding a simplicial node.
To ease notation, we consider that we start by gluing a multipath, but one can easily start by adding a simplicial node instead.
In each step, we just need to state a condition ensuring that we glue the multipath between two nodes where we can guarantee that the Hamiltonian cycle constructed thus far is $0$-even.
This way we prove (the general version of) \Cref{thmB}.

\begin{corollary}\label{cor:AddingSimplicialNodesAndGluingMultiPath}
Let $G_0, \dots, G_p$ be a sequence of graphs such that $G_{2i+1} = \oplus^{\b \ell_i}_{u_i, v_i} G_{2i}$ for some $u_i, v_i \in \V(G_{2i})$ and some $\b\ell_i\in\N_{\geq 1}^{r_i}$ with $r_i \geq 1$ where $\b\ell_i$ contains at most one $1$, and $G_{2i+2} = \odot_{\b K_i} G_{2i+1}$ for some clique-sequence $\b K_i$ of length $s_i\geq 0$ on $G_{2i+1}$.
If
\begin{compactenum}
\item $\AO[G_0]$ admits a Hamiltonian cycle $\b A$ where $\b A(u_0, v_0)$ is $0$-even; and
\item for all $i$ we have that $\sum_{j=1}^{r_i} \ell_{i, j}$ is even, that $u_i$ and $v_i$ are in the same connected component of $G_{2i}$; and either:
\begin{compactitem}
\item $u_iv_i \in E(G_{2i})$; or
\item $s_i \geq 1$ and there is $j$ such that $|K_{i, j}|$ is odd, and, with $m_i = \max\{j ~;~ |K_{i, j}| \text{ is odd}\}$ and $(x_{i, 1}, \dots, x_{i, s_i})$ the simplicial nodes added in $\odot_{\b K_i} G_{2i+1}$, we have $u_i, v_i \in \V(G_{2i}) \ssm \{x_{i, m_i}, x_{i, m_i+1}, \dots, x_{i, s_i}\}$;
\end{compactitem}
\end{compactenum}
then $G_p$ is {\good}.
\end{corollary}

\begin{proof}
Direct combination of \Cref{cor:MultiPathGluingIsGood,cor:MultiPathGluingOnArcIsGood,thm:AddingSimplicialNodePreservesAOHamiltonian}:
the conditions ensure that $\AO[G_i]$ admits a Hamiltonian cycle $\b A_i$ such that $\b A_{2i}(u_i, v_i)$ is $0$-even for all $i$.
\end{proof}

% \clearpage

\subsection{Disconnected graphs, 1-sums, strongly even and internally even sequences}\label{ssec:StronglyInternallyEven}

It is well-known that if $H$ and $H'$ are Hamiltonian graphs, then the Cartesian product $H\times H'$ is also Hamiltonian:
for instance, it is the first sentence of \cite{TrotterErdos1978-ProductIsHamiltonian}.
Consequently, if $G$ has several connected components, each inducing a {\good} graph, then $G$ itself is also {\good} because it is straightforward to show that $\AO$ is the Cartesian product of the graph of acyclic orientations of its connected components.
Similarly, if $G = H\oplus_1 H'$ is a $1$-sum, then $\AO = \AO[H]\times \AO[H']$, so if $H$ and $H'$ are {\good}, then $G$ is also {\good}.

We can go further and keep track of the $0$-evenness.
For disconnected graphs, the question is trivial by definition:
if $u$ and $v$ are not in the same connected component, then $\b A(u, v) = (0, 0, \dots, 0)$ for any sequence of orientations $\b A$, since no orientation can induce a directed path between $u$ and~$v$.
(Note however that \Cref{cor:MultiPathGluingIsGood} does not apply if $u$ and $v$ are not in the same connected component.)
The case of $1$-sums is, strangely, richer.
We need two notions to further control evenness.

\begin{definition}\label{def:StronglyEven}
A sequence of integers $\b a = (a_1, \dots, a_r)$ is \defn{strongly even} if for all $x \in \{a_1, \dots, a_r\}$ with $x \ne 0$:
\begin{compactitem}
\item if $1 \leq i < j\leq r$ are such that $a_i = a_j \ne 0$ and $a_k \ne a_i$ for all $i < k < j$, then $i$ and $j$ have the same parity; and
\item $\min\{i ~;~ a_i = x\}$ is odd and $\max\{i ~;~ a_i = x\}$ is even.
\end{compactitem}
\end{definition}

\begin{definition}\label{def:InternallyEven}
A sequence of integers $\b a = (a_1, \dots, a_r)$ is \defn{internally even} if:
\begin{compactitem}
\item for any $x \in \{a_1, \dots, a_r\}$ with $x \ne 0$, if $1 \leq i < j\leq r$ are such that $a_i = a_j \ne 0$ and $a_k \ne a_i$ for all $i < k < j$, then $i$ and $j$ have the same parity; and
\item the parity of $\min\{i ~;~ a_i = x\}$ does not depend on $x\ne 0$; and the parity of $\max\{i ~;~ a_i = x\}$ does not depend on $x\ne 0$.
\end{compactitem}
\end{definition}

\begin{remark}
It could be that $\b a$ is strongly even but not $0$-even, \eg $\b a = (1, 0, 2, 1, 0, 2)$.
On the other hand, strongly even implies internally even.
\end{remark}

We first show that $1$-summing a {\good} graph with a strongly even {\good} graph yields a $0$-even {\good} graph.

\begin{proposition}\label{prop:1SumStronglyEven}
Let $G = H\oplus_1 H'$ be a $1$-sum with $w \in \V(H)\cap \V(H')$, and let any $u\in \V(H)\ssm\{w\}$ and $v\in \V(H')\ssm\{w\}$.
If $H'$ is {\good} and $\AO[H]$ admits a Hamiltonian cycle $\b A$ such that $\b A(u, w)$ is strongly even, then $\AO$ admits a Hamiltonian cycle $\b B$ such that $\b B(u, v)$ is both strongly even and $0$-even.
\end{proposition}

\begin{proof}
Let $\b A' = (A'_1, \dots, A'_t)$ be a Hamiltonian cycle of $\AO[H']$ (where $\b A'(w, v)$ is not necessarily $0$-even).
We denote $\b A = (A_1, \dots, A_s)$.
Since $G$ is the $1$-sum $H \oplus_1 H'$, any two acyclic orientations $A_i \in \b A$, $A'_j \in \b A'$ give rise to an acyclic orientation of $G$, so the acyclic orientations of $G$ are $A_i\sqcup A'_j$ for all $i\in [s]$ and $j\in [t]$.
Besides, since the total number of acyclic orientations of a graph is always even (see \Cref{lem:CompletementIsInvolution}), both $s$ and $t$ are even.
Consider $\b B$ defined as:
\begin{equation}\label{eqn:HamiltonianCycleForCartesianProduct}
\begin{array}{rl}
\b B = &\bigl(A_1\sqcup A'_1,\, A_2\sqcup A'_1,\, \dots,\, A_s\sqcup A'_1,\\
&A_s\sqcup A'_2,\, A_{s-1}\sqcup A'_2,\, \dots,\, A_1\sqcup A'_2,\\
&A_1\sqcup A'_3,\, A_2\sqcup A'_3,\, \dots,\, A_s\sqcup A'_3,\\
&\dots,\\
&A_s\sqcup A'_t,\, A_{s-1}\sqcup A'_t,\, \dots,\, A_1\sqcup A'_t\bigr)\,.
\end{array}
\end{equation}

By construction, $\b B$ is a Hamiltonian cycle of $\AO$.
Moreover, $(A_i\sqcup A'_j)(u, v) \ne 0$ if and only if $A_i(u, w) = A'_j(w, v) \ne 0$.
If $A'_j(w, v) = 0$, then $(A_i\sqcup A'_j)(u, v) = 0$, so in evaluating the above definition to get $\b B(u, v)$ the $j^\text{th}$ row is $(0, \dots, 0)$ (and it is of even length).
If $A'_j(w, v) = +1$ (resp. $-1$), then $(A_i\sqcup A'_j)(u, v) = 0$ except if $A_i(u, w) = +1$ (resp. $-1$), so the $j^{\text{th}}$ row in the above is the sequence $\b A(u, w)$ with its $-1$ (resp. $+1$) replaced by $0$:
since $\b A(u, w)$ is strongly even, this sequence is formed by $1$ (resp. $-1$) separated by an even number of zeros with an even number of zeros at the beginning and at the end.
Therefore, by concatenation, the sequence $\b B(u, v)$ is $0$-even.

Furthermore, $\b B(u, v)$ is strongly even.
Indeed, if $(A_i\sqcup A'_j)(u, v) = +1$ and $(A_k\sqcup A'_\ell)(u, v) = +1$ are such that there is no $A_m\sqcup A'_n$ in between in $\b B$ with $(A_m\sqcup A'_n)(u, v) = +1$, then either $j = \ell$ and both orientations are on the same row and separated by an even number of zeros as we have seen; either $j < \ell$ and both orientations are separated by a certain number of rows (each of them of even length) and are an even number of zeros away from the endpoints of their respective rows.
Besides, the second condition for strong evenness of $\b B(u, v)$ straightforwardly follows from the one for $\b A(u, v)$.
\end{proof}

Secondly, $1$-summing a $0$-even {\good} graph with an internally even {\good} graph preserves both $\c{AO}$-Hamiltonicity and $0$-evenness.

\begin{proposition}\label{prop:1Sum0EvenAndInternallyEven}
Let $G = H\oplus_1 H'$ be a $1$-sum with $w \in \V(H)\cap \V(H')$, and let any $u\in \V(H)\ssm\{w\}$ and $v\in \V(H')\ssm\{w\}$.
If $\AO[H']$ admits a Hamiltonian cycle $\b A'$ such that $\b A'(w, v)$ is $0$-even and $\AO[H]$ admits a Hamiltonian cycle $\b A$ such that $\b A(u, w)$ is internally even, then $\AO$ admits a Hamiltonian cycle $\b B$ such that $\b B(u, v)$ is both internally even and $0$-even.
\end{proposition}

\begin{proof}
The proof is very similar to the one of \Cref{prop:1SumStronglyEven}, only the last part of the argument needs to be modified.

We construct $\b B$ as in \Cref{eqn:HamiltonianCycleForCartesianProduct}.
We call the $i^{\text{th}}$ row of $\b B(u, v)$ a $0$-row if $(A_k\sqcup A'_i)(u, v) = 0$ for all $k\in [s]$.
Suppose that the $i^{\text{th}}$ and the $j^{\text{th}}$ rows of $\b B(u, v)$ are non-$0$-rows but that the $k^{\text{th}}$ row is a $0$-row for all $i < k < j$.
Since $\b A(u, w)$ is $0$-even, the $0$-rows come in consecutive sequences of even length: $i$ and $j$ have the same parity.
Let $f_i$ be the number of final $0$s in the $i^{\text{th}}$ row of $\b B(u, v)$ and $\ell_j$ the number of leading $0$s in its $j^{\text{th}}$ row.
As $i$ and $j$ have the same parity, $f_i$ and $\ell_j$ are either both a minimum $k$ such that $A'_k(w, v) \ne 0$ or both a maximum $k$ such that $A'_k(w, v) \ne 0$.
Since $\b A'(w, v)$ is internally even, $f_i$ and $\ell_j$ have the same parity (\ie $f_i + \ell_j$ is even).
Thus, the number of consecutive $0$s between the last non-zero value of the $i^{\text{th}}$ row of $\b B(u, v)$ and the first non-zero value of its $j^{\text{th}}$ row is even.
Consequently, $\b B(u, v)$ is $0$-even.

A similar reasoning as in the proof of \Cref{prop:1SumStronglyEven} for strong evenness shows that $\b B(u, v)$ is also internally even.
\end{proof}

With some case-by-case analysis, the reader can adapt \Cref{prop:1SumStronglyEven,prop:1Sum0EvenAndInternallyEven} for $1$-summing several graphs at once.
We do not state a general theorem here since it is unbearably notation-heavy.
However, we give the following corollary.
To prove it, we need the following notion.

\begin{definition}
For $A\in \AO$ and four nodes $u, v, x, y\in \V$, their \defn{combined direction} is
$$\mathdefn{A(u, v) \combdir A(x, y)} \eqdef \begin{cases}
A(u, v) & \text{if } A(u, v) = A(x, y) \\
0 & \text{otherwise}
\end{cases} \,.$$

As per usual, for a sequence of orientations $\b A = (A_1, \dots, A_r)$, we denote $\b A(u, v)\combdir\b A(x, y) \eqdef \bigl(A_1(u, v)\combdir A_1(x, y),\, \dots,\, A_r(u, v)\combdir A_r(x, y)\bigr)$.
\end{definition}

\begin{corollary}\label{cor:AtADistance0EvennessControl}
Let $G = (\V, E)$ be a {\good} graph with $\b A = (A_1, \dots, A_s)$ a Hamiltonian cycle in $\AO$ and let $u, v \in \V$ such that $\b A(u, v)$ is $0$-even. 
Let $H$ be the $1$-sum at $u$ of $G$ with a path $P_m$ with endpoint $x$, and let $G'$ be the $1$-sum at $v$ of $H$ with a path $P_n$ with endpoint $y$.

Then both $H$ and $G'$ are {\good}.
If, moreover, $m$ and $n$ have the same parity, then there exists a Hamiltonian cycle $\b B = (B_1, \dots, B_t)$ in $\AO[G']$ such that $\b B(x, y)$ is $0$-even.
\end{corollary}

\begin{proof}
Since $H$ is a $1$-sum, the graph $\AO[H]$ is isomorphic to $\AO\times\AO[P_m]$, while the graph $\AO[G']$ is isomorphic to $\AO\times\AO[P_m]\times\AO[P_n]$.
The Cartesian product of Hamiltonian graphs being Hamiltonian, and $\AO[P_k]$ being isomorphic to the (Hamiltonian) cube-graph $Q_k$, we get that both $H$ and $G'$ are {\good}.

Let $\b S$ (resp. $\b S'$) be a Hamiltonian cycle in $\AO[P_m]$ with $S_1(x, u) = 1$ and $S_\mu(x, u) = -1$ (resp. $S'_1(v, y) = 1$ and $S'_{\mu'}(v, y) = -1$).
Since $m$ and $n$ have the same parity, $\mu$ and $\mu'$ also have the same parity (which is the parity of $m+1$).
Consider the usual Hamiltonian cycle $\b T$ in $\AO[P_m]\times\AO[P_n]$ constructed via \Cref{eqn:HamiltonianCycleForCartesianProduct} on $\b S$ and $\b S'$.
% We write $A'\sqcup A''\in \AO[P_m]\times\AO[P_n]$ with $A'\in \AO[P_m]$ and $A''\in \AO[P_n]$.
For an acyclic orientation $A\in \AO$, we have $(A\sqcup S\sqcup S')(x, y) = 1$ (resp. $-1$) if and only if $S(x, u) = A(u, v) = S'(v, y) = 1$ (resp. $-1$).
The sequence $\bigl(S(x, u)\combdir S'(v, y) ~;~ S\sqcup S'\in \b T\bigr)$ is internally even.
Indeed, it has a unique $1$ at the first position, and a unique $-1$ at position either $m\cdot(\mu-1) + \mu'$ if $m$ is even, or $m\cdot(\mu-1) + n-\mu'$ if $m$ is odd: this position is odd in both cases.

It is left to the reader as an exercise to prove that a variant of \Cref{prop:1Sum0EvenAndInternallyEven} implies that the Hamiltonian cycle $\b B$ of $\AO[G']$ constructed using \Cref{eqn:HamiltonianCycleForCartesianProduct} on $\b A$ and $\b T$ satisfies that $\b B(u, v)$ is $0$-even.
This concludes the proof.
\end{proof}

\begin{corollary}\label{cor:AtADistanceMultiPathGluing}
Let $G = (\V, E)$ be a {\good} graph with $\b A = (A_1, \dots, A_s)$ a Hamiltonian cycle in $\AO$ and let $u, v \in \V$ such that $\b A(u, v)$ is $0$-even. 
Let $H$ be the $1$-sum at $u$ of $G$ with a path $P_m$ with endpoint $x$, and let $G'$ be the $1$-sum at $v$ of $H$ with a path $P_n$ with endpoint $y$.

If $m$ and $n$ have the same parity and $\sum_k\ell_k$ is even, then $\oplus_{x, y}^{\b \ell} G'$ is {\good}.
\end{corollary}

\begin{proof}
Direct application of \Cref{cor:MultiPathGluingIsGood,cor:AtADistance0EvennessControl}.
\end{proof}

\Cref{cor:AtADistanceMultiPathGluing} is our third result (with \Cref{cor:MultiPathGluingOnArcIsGood,cor:AddingSimplicialNodesAndGluingMultiPath}) explaining how we can glue a multipath at a pair of nodes $(u, v)$, then glue another multipath at another pair of nodes $(x, y)$.
These three statements are of different nature and suggest the following problem for future research.

\begin{question}
Find conditions (such as $0$-even, strongly even and internally even) that a graph $G$ or (a Hamiltonian cycle in) its graph of acyclic orientations $\AO$ should satisfy to ensure that one can glue multipaths at the different pairs of nodes of $G$.
\end{question}

Note that this question is only an intermediate step.
If one wants to exploit the (open) ear decomposition of $2$-connected graphs in order to inductively certify whether a graph is {\good} or not, then he or she would not only need to solve the above question where the pairs of nodes are taken in the same starting graph $G$, but also to consider pairs of nodes where one or both lie on the glued ears (\ie multipaths) of the previous steps.

\appendix
% Show only sections in the ToC for the appendix
\addtocontents{toc}{\protect\setcounter{tocdepth}{1}}

\section{General statements about pattern lacing}\label{appendix:GeneralPatternLacing}

% \subsectionnotoc{Loose lacing}\label{appendix:LooseLacing}
\subsection{Loose lacing}\label{appendix:LooseLacing}

In this section, we state more general methods of pattern (self-)lacing as compared to \Cref{prop:LacingIsHamiltonianPath,cor:SelfLacingIntoAnAOcycle}. 
For a replacing pattern $\b P$ of $(A_1, \dots, A_s)$ and a replacing pattern $\b Q$ of $(A_s, \dots, A_t)$, consider the graph, illustrated in \Cref{fig:LooseLacing}, whose node set is $\AO[H' |A_s]$ and where there are three types of edges:
\begin{compactitem}
\item \textcolor{red}{red edges} between $A'$ and $B'$ such that $\omega^{\b P}_s(B') = \omega^{\b P}_s(A') + 1$ and $\omega^{\b P}_s(A') \notin \cons_s(\b P)$; and
\item \textcolor{blue}{blue edges} between $A'$ and $B'$ such that $\omega^{\b Q}_s(B') = \omega^{\b Q}_s(A') + 1$ and $\omega^{\b Q}_s(A') \notin \cons_s(\b Q)$; and
\item \textcolor{green}{green edges} between $A'$ and $B'$ such that $\omega^{\b P}_s(B') = \omega^{\b P}_s(A') + 1$ and $\omega^{\b P}_s(A') \in \cons_s(\b P)$ or such that $\omega^{\b Q}_s(B') = \omega^{\b Q}_s(A') + 1$ and $\omega^{\b Q}_s(A') \in \cons_s(\b Q)$.
\end{compactitem}

We say that $\b P$ and $\b Q$ \defn{can loosely lace} if there exists a (directed) path $\b\rho$ from $A'$ with $\omega^{\b P}_s(A') = 1$ to $B'$ with $\omega^{\b Q}_s(B') = m$ such that $\b\rho$ contains all the \textcolor{red}{red edges} and all the \textcolor{blue}{blue edges} (and some, perhaps none, of the \textcolor{green}{green edges}) of this graph.
One has to be careful, since several such paths may exist.
The \defn{loose lace} of $\b P$ and $\b Q$ according to $\b \rho$, denoted $\mathdefn{\b P\wedge_{\b\rho} \b Q}$, is obtained by replacing each \textcolor{red}{red arc} (resp. \textcolor{blue}{blue arc}) of $\b \rho$ by the corresponding directed path in $\b P$ (resp. in $\b Q$), keeping each \textcolor{green}{green arc}, and adding $\b P_0$ and $\b Q_m$ (defined in the usual lacing $\b P\wedge \b Q$).

A proof similar to the proof of \Cref{prop:LacingIsHamiltonianPath} ensures that:
If $\b P$ and $\b Q$ can loosely lace, then any loose lace $\b P\wedge_{\b\rho} \b Q$ is a replacing pattern of $\M[A_1, \dots, A_t]$.

In the case of usual lacing, the graph at stake is a path with two edges between consecutive nodes, one of them being \textcolor{green}{green}.

\begin{figure}[h]
    \centering
    \includegraphics[width=\linewidth]{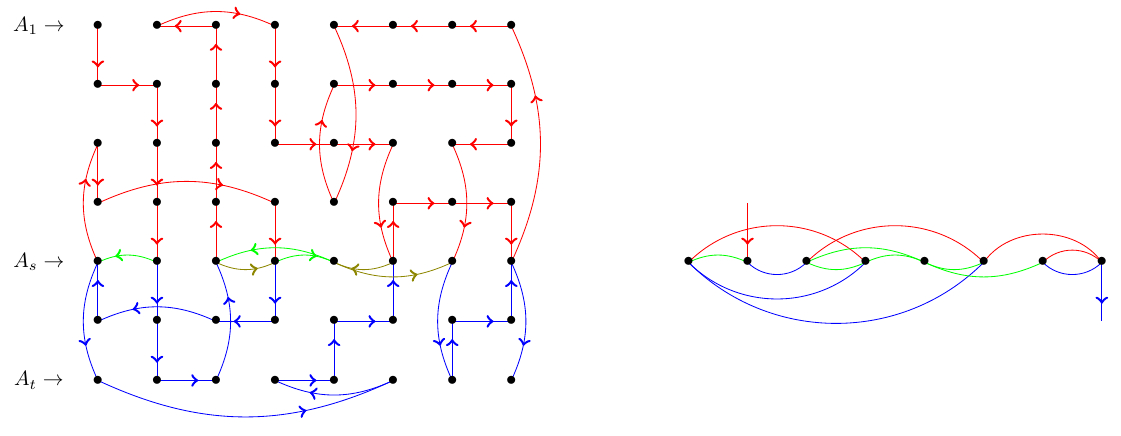}
    \caption{(Left) A replacing pattern $\b P$ of $(A_1, \dots, A_s)$ in \textcolor{red}{red} and a replacing pattern $\b Q$ of $(A_s, \dots, A_t)$ in \textcolor{blue}{blue}: for sake of clarity of the drawing, we have assumed that all $\AO[H'|A_i]$ are the same, but in general many black dots $\bullet$ would not be visited by replacing patterns.
    Note that arcs in these replacing patterns can only be horizontal or vertical (not diagonal) since $\b P$ and $\b Q$ are directed paths inside the product $\AO[H]\times \AO[H']$.
    The \textcolor{green}{green} arcs belong to $\b P$ while the \textcolor{olive}{olive} arcs belong to $\b Q$.
    (Right) The graph on $\AO[H'|A_s]$ obtained by substituting directed paths in $\b P$ and $\b Q$ by single edges in order to decide if $\b P$ and $\b Q$ can loosely lace (here, they cannot).
    \textcolor{green}{Green} arcs and \textcolor{olive}{olive} arcs are \textcolor{green}{green} edges here since we no longer differentiate them.}
    \label{fig:LooseLacing}
\end{figure}

\medskip

Besides, for a replacing pattern $\b P$ of $(A_1, \dots, A_s)$ with $A_1 = A_s$, consider the graph whose node set is $\AO[H' |A_1]$ and where there are four types of edges:
\begin{compactitem}
\item \textcolor{red}{red edges} between $A'$ and $B'$ such that $\omega^{\b P}_s(B') = \omega^{\b P}_s(A') + 1$ and $\omega^{\b P}_s(A') \notin \cons_s(\b P)$; and
\item \textcolor{blue}{blue edges} between $A'$ and $B'$ such that $\omega^{\b P}_1(B') = \omega^{\b P}_1(A') + 1$ and $\omega^{\b P}_1(A') \notin \cons_1(\b P)$; and
\item \textcolor{purple}{purple edges} between $A'$ and $B'$ if there is a path between $A_1 \sqcup A'$ and $A_s \sqcup B'$ in $\b P$ and no $A_1 \sqcup B'$ nor $A_s \sqcup B'$ lie on this path; and
\item \textcolor{green}{green edges} between $A'$ and $B'$ such that $\omega^{\b P}_s(B') = \omega^{\b P}_s(A') + 1$ and $\omega^{\b P}_s(A') \in \cons_s(\b P)$ or such that $\omega^{\b P}_1(B') = \omega^{\b P}_1(A') + 1$ and $\omega^{\b P}_1(A') \in \cons_1(\b P)$.
\end{compactitem}
We say that $\b P$ \defn{can loosely self-lace} if this graph is loop-free and admits a (directed) cycle using all its \textcolor{red}{red}, \textcolor{blue}{blue} and \textcolor{purple}{purple} edges (and some, perhaps none, of its \textcolor{green}{green edges}).
The \defn{loose self-lacing} according to $\b \rho$, denoted $\mathdefn{\wedge_{\b\rho} \b P}$ is defined accordingly.

A proof similar to the proof of \Cref{cor:SelfLacingIntoAnAOcycle} ensures that:
If $\b A$ is a Hamiltonian cycle for $\AO[H]$ and $\b P$ a replacing pattern for $\b A$ that can loosely self-lace, then any loose self-lace $\wedge_{\b \rho} \b P$ is a Hamiltonian cycle for $\AO$.

In the case of usual self-lacing, the graph at stake contains no \textcolor{blue}{blue edges}, but only is a path of \textcolor{green}{green edges} from some $A'$ to some $B'$ together with another path of \textcolor{red}{red} and \textcolor{green}{green} edges from the same $A'$ to $B'$, and a unique \textcolor{purple}{purple edge} between $A'$ and $B'$.

% \begingroup
% \let\addcontentsline\relax
% \subsection{Splitting subsequences with varying endpoints}\label{appendix:SplittingSubsequences}
% \endgroup

\subsection{Splitting subsequences with varying endpoints}\label{appendix:SplittingSubsequences}

One can easily generalize \Cref{thm:ConcatenatingReplacingPatternsWithSharedEndPoints} as follows, in order to weaken its hypotheses.
% As it makes the hypotheses less concrete, we stayed with \Cref{thm:ConcatenatingReplacingPatternsWithSharedEndPoints} as stated.

Let $G = H\sqcup H'$ be a bipartition of the edges of $G$.
Suppose that there exist $X_1, \dots, X_r\in \AO[H']$ and a Hamiltonian path $\b A = (A_1, \dots, A_s)\in \AO[H]$ such that:
\begin{compactitem}
\item one can split the sequence $\b A$ into splitting subsequences $\b A_j = (A_{i_j}, A_{i_j+1}, \dots, A_{i_{j+1}-1})$ for $j \in [r]$ with $i_1 = 1$ and $i_{r}-1 = s$; and
\item $X_j \in \AO[H'|A_{i_j}] \cap \AO[H'|A_{i_j-1}]$ is such that $\b A_j$ admits a replacing pattern which starts at $A_{i_j} \sqcup X_j$ and ends at $A_{i_{j+1}-1} \sqcup X_{j+1}$.
\end{compactitem}
Then there exists a Hamiltonian path in $\AO$ obtained as the concatenation of these replacing patterns.
Moreover, if $\b A$ is a Hamiltonian cycle of $\AO[H]$ and if $X_r = X_1$, then $G$ is {\good}.

This version of \Cref{thm:ConcatenatingReplacingPatternsWithSharedEndPoints} applies in a broader context, but has the drawback of making the pivotal orientations $X_j$ depend on the splitting subsequences $\b A_j$.

% \clearpage 
\section{A note on complete bipartite graphs}\label{app_bipartite}

The \defn{complete bipartite graph} $\mathdefn{K_{m, n}}$ is the graph on nodes $(a_1, \dots, a_m)$ and $(b_1, \dots, b_n)$ with edges $a_ib_j$ for all $i\in [m]$ and $j\in [n]$.
The longstanding question of which $K_{m, n}$ are {\good} already appears in \cite[Sections~3 \& 5]{SavageSquireWest93-GrayCodesAcyclicOrientations}.
In particular, the authors proved that if $K_{m, n}$ is {\good}, then both $m$ and $n$ are odd, see \cite[Theorem~3.3]{SavageSquireWest93-GrayCodesAcyclicOrientations} where they essentially prove that $\sigma(K_{m, n}) = 0$ if and only if both $m$ and $n$ are odd.
We conjecture (as other authors conducting research on the subject) that the reciprocal is true, \ie that
% \begin{conjecture}
the complete bipartite graph $K_{m, n}$ is {\good} if and only if both $m$ and $n$ are odd.
% \end{conjecture}

\begin{example}
As already explained in \cite[below Thm.~3.3]{SavageSquireWest93-GrayCodesAcyclicOrientations}, the graphs $K_{1, n}$ are stars (hence, trees), so the graph of acyclic orientations $\AO[K_{1, n}]$ is isomorphic to the cube-graph $Q_n$.
In particular, by \Cref{exm:CubeGraphIsGood}, the graphs $K_{1, n}$ are {\good}.

The graph $\AO[K_{3, 5}]$ has $4\,718$ nodes and $20\,610$ edges, while the graph $\AO[K_{5, 5}]$ has $329\,462$ nodes and $1\,902\,750$ edges, so a powerful computer and a good implementation of Hamiltonian cycle search might decide whether $K_{3, 5}$ and $K_{5, 5}$ are {\good} or not.

Via a computer experiment (on a normal laptop), we confirmed that $K_{3, 3}$ is {\good}, which was left open in \cite{SavageSquireWest93-GrayCodesAcyclicOrientations}: the graph $\AO[K_{3, 3}]$ has $230$ nodes and $702$ edges.
To showcase the compactness of Gray codes, we give explicitly a Hamiltonian cycle in $\AO[K_{3, 3}]$, namely:

\noindent[3 1 3 4 3 6 1 9 1 2 9 6 5 7 5 8 4 5 3 9 3 1 3 9 2 9 7 9 5 7 8 1 2 7 1 3 6 5 1 8 1 4 1 7 8 5 2 5 7 9 1 7 1\\
2 5 2 1 8 3 5 6 2 9 6 9 7 2 3 2 7 2 1 5 2 6 3 4 1 7 5 3 6 8 1 4 1 9 1 3 2 6 7 1 8 7 2 1 7 8 6 2 8 1 8 7 3\\
4 3 5 3 1 9 1 6 7 1 7 2 1 6 1 9 6 7 2 4 3 9 8 3 8 2 4 6 2 1 3 4 6 8 9 8 6 1 7 1 9 4 6 8 3 2 1 7 1 2 1 6 1\\
3 6 8 1 8 9 4 5 9 2 3 9 7 4 9 6 4 7 4 1 5 1 7 1 2 3 5 2 7 2 3 7 8 3 4 7 1 5 8 5 2 9 7 1 2 4 2 8 4 6 1 5 7\\
8 4 1 7 1 3 1 5 8 1 3 8 1 8 2 4 5 4]

To read this Hamiltonian cycle, one proceeds as follows.
The $i^{\text{th}}$ digit $s_i$ (read from left to right, \ie $s_1 = 3$, $s_2 = 1$, etc.) of the above sequence corresponds to an acyclic orientation $A_i$ defined by:
\begin{compactitem}
\item $A_0$ is the acyclic orientation depicted in \Cref{fig:StartOrientationK33};
\item $A_i$ is obtained from $A_{i-1}$ by changing the orientation of the arc labeled $s_i$.
\end{compactitem}
Besides, to compute directly $A_i$ from $A_0$, one can count for each $j\in [9]$ how many times it appears among $s_1, \dots, s_i$: if $j$ appears an even number of times, then the arcs labeled $j$ are oriented the same way in $A_0$ and in $A_i$, otherwise they are not.
One can check that $A_{230}$ (defined from $A_{229}$ by changing the arc labeled $4$) is indeed $A_0$, so the above sequence is indeed a cycle.
\end{example}

\begin{figure}[h]
    \centering
    \includegraphics[width=0.33\linewidth]{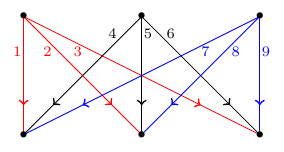}
    \caption{The acyclic orientation $A_0$ of $K_{3, 3}$ at the start of the Hamiltonian cycle of $\AO[K_{3, 3}]$.
    Colors are meaningless, they only help to read the labels of the arcs.}
    \label{fig:StartOrientationK33}
\end{figure}

Let $\c P(i, j)$ be the set of ordered partitions of $[i]$ into $j$ parts.
In \cite{SavageSquireWest93-GrayCodesAcyclicOrientations}, it is established that the acyclic orientations of the complete bipartite graph $K_{m, n}$ are in one-to-one correspondence with triples $(Q, Q', \delta)$ where $\delta \in\{0, 1\}$, $Q\in \c P(m, k)$, $Q'\in \c P(n, k + \varepsilon)$ for some $k \geq 1$ and $\varepsilon  \in \{0,1\}$.
Given an acyclic orientation $A$ of $K_{m, n}$, one obtains such a triple as follows (see illustration in \Cref{fig:Example_graph_to_partition}).
Let $S_1$ be the set of sources of $A$, let $S_2$ be the set of sources of $A\ssm S_1$, and recursively let $S_i$ be the set of sources of $A \ssm (S_1 \cup \dots \cup S_{i-1})$ (where $S_i\ne \emptyset$ since $A$ is acyclic), and let $S_{\kappa}$ be the last (non-empty) source set.
Note that each $S_i$ must be contained exclusively in one of the two parts of $K_{m, n}$ because $K_{m, n}$ is the \emph{complete} bipartite graph.
We set $\varepsilon = 1$ and $k = \frac{\kappa-1}{2}$ if $\kappa$ is odd, and $\varepsilon = 0$ and $k = \frac{\kappa}{2}$ otherwise.
We set $\delta = 0$ if $S_1$ is contained in the part of $K_{m, n}$ of size $m$, and $\delta = 1$ otherwise.
If $\delta = 1$, then $Q = (S_1, S_3, S_5, \dots, S_{2(k+\varepsilon) - 1})$ is our ordered partition of $[m]$, and $Q' = (S_2, S_4, S_6, \dots, S_{2k})$ is our ordered partition of $[n]$; if $\delta = 1$, the opposite holds.

Equivalently, $S_i$ is the set of elements of rank $i+1$ in the poset induced by $A$: this poset is ranked because $K_{m, n}$ is the complete bipartite graph.
The reciprocal bijection can be found in \cite{SavageSquireWest93-GrayCodesAcyclicOrientations}. 

We call the sequence $(S_1, \dots, S_\kappa)$ and its encoding as triple $(Q, Q', \delta)$ a \defn{bicolored $(m,n)$-partition}.
Two elements $p\in [n]$, $q\in[m]$ are \defn{neighboring} in $(S_1, \dots, S_\kappa)$ if there exists an index $i \in [\kappa-1]$, such that $p \in S_i$ and $q \in S_{i+1}$. 

\begin{figure}[h]
    \centering
    \includegraphics[width=0.55\linewidth]{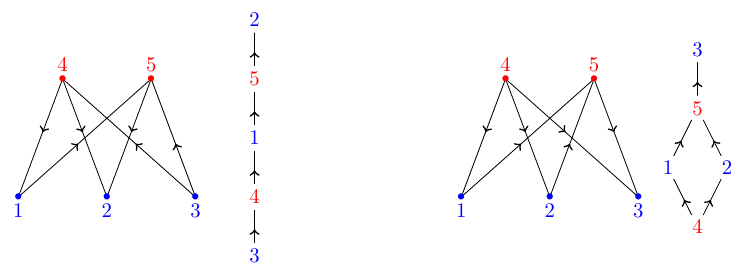}
    \caption{Two acyclic orientations of $K_{2, 3}$ together with their representations as posets.
    The left one corresponds to the sequence of source sets $\textcolor{blue}{3}|\textcolor{red}{4}|\textcolor{blue}{1}|\textcolor{red}{5}|\textcolor{blue}{2}$ yielding the triple $(4|5,\, 3|1|2,\, 1)$ where $\varepsilon = 1$, while the right one corresponds to $\textcolor{red}{4}|\textcolor{blue}{12}|\textcolor{red}{5}|\textcolor{blue}{3}$ yielding the triple $(4|5,\, 12|3,\, 0)$ where $\varepsilon = 0$.}
    \label{fig:Example_graph_to_partition}
\end{figure}

In \cite{SavageSquireWest93-GrayCodesAcyclicOrientations}, the authors only focus on counting the number of bicolored $(m, n)$-partitions because they want to discuss a parity argument.
We go slightly further by describing the graph $\c{AO}(m, n)$ thanks to bicolored $(m, n)$-partitions.

Let $\rho = (Q,Q', \delta)$ be a bicolored partition and $p, q$ be two neighboring elements in $\rho$, with $p \in S_i$ and $q \in S_{i+1}$.
Consider the sequence
$\b S(p, q) = (S_1, \dots, S_{i-1}, S_i \ssm \{p\}, \{q\}, \{p\}, S_{i+1}\ssm \{q\}, S_{i+2}, \dots, S_\kappa)\,.$
This sequence $\b S(p, q)$ may contain some empty sets (possibly $S_i \ssm \{p\}$ or $S_{i+1} \ssm \{q\}$), so we \emph{compress} it by removing all its empty sets and then merging adjacent sets of the same color together, that is by replacing all sub-sequences $(S_{a-1}, S_a, S_{a+1})$ where $S_a = \emptyset$ by a sub-sequence $(S_{a-1}\cup S_{a+1})$.
\begin{definition}
Denoting $\rho'$ the compressed version of $\b S(p, q)$ as described above, we say that $\rho, \rho'$ \defn{differ by a flip} at $(p, q)$.
The \defn{graph of flips} of bicolored $(m, n)$-partitions is the graph whose nodes are the bicolored $(m, n)$-partitions and where two nodes share an edge if they differ by some flip.
\end{definition}

\begin{lemma}
The graph $\c {AO}(m, n)$ is isomorphic to the graph of flips of bicolored $(m, n)$-partitions.
\end{lemma}

\begin{proof}
Let $A\in \c {AO}(m, n)$, and $\rho = (Q, Q', \delta)$ the associated triple.
We first need to show that flipping an edge $(p, q)$ where $p, q$ are neighboring in $\rho$ will not result in an oriented cycle.
Let $S_i$ be the block containing $p$ and $S_{i+1}$ the neighboring block of the partition $\rho$ (Because of the orientation of $(p, q)$, $p$ comes before $q$).
If there were an oriented cycle after flipping this edge, then there would be a directed path from $p$ to $q$ in $A$.
Since the graph is bipartite, the path has length at least $3$ and is thus of the form $(p, r_1, r_2, q)$.
By construction of $A$, this means that the node $r_1$ is in a block between $S_i$ and $S_{i+1}$, which is a contradiction. 

Conversely, consider an edge $(p, q)$ between two non-neighboring blocks.
It follows that there are at least two blocks between $S_i$ containing $p$ and $S_j$ containing $q$.
Take one node of each block, then these $p, q$ and these two nodes form an oriented path in $A$ and would therefore form an oriented $4$-cycle with the edge $(p, q)$.
This finishes the proof.
\end{proof} 

Thus, the problem of finding a Hamiltonian cycle in $\AO[K_{m, n}]$ is equivalent to finding a Gray code for these bicolored $(m,n)$-partitions with respect to the flipping operation described above.

This latter problem has already been explored in the unlabeled case, so bicolored partitions where the elements in each block are not integers in $[m]$ and $[n]$, but rather indistinguishable elements.
Such an unlabeled colored $(m, n)$-partition is a string of $m+n$ elements, namely $m$ zeros and $n$ ones.
Two such strings $s$ and $s'$ differ by a flip if and only if $s'$ is obtained from $s$ by swapping a (consecutive) sub-string $01$ for a (consecutive) sub-string $10$, or the converse\footnote{Although this graph of flips has $\binom{m+n}{m}$ elements, it is not the Johnson graph $J(m+n, m)$ since it is not regular.}.

For example, the two unlabeled partitions corresponding to the bicolored $(m, n)$-partitions in \Cref{fig:Example_graph_to_partition} are respectively $10101$ and $01101$. 

\begin{theorem}[\cite{buckwiedemanngraycodes,eadeshickeyreadhamiltonpaths} independently]
If $n \leq 1$ or $m \leq 1$ or both $n$ and $m$ are odd, then the graph of flips of unlabeled bicolored $(m, n)$-partitions admits a Hamiltonian path starting at $11\dots 1 00\dots0$ and $00\dots 0 11\dots 1$ with $m$ copies of $0$ and $n$ copies of $1$.
\end{theorem}

Note that the above Gray code must start and end at the combinations $11\dots 1 00\dots0$ and $00\dots 0 11\dots 1$ with $m$ copies of $0$ and $n$ copies of $1$.
Consequently, a possible strategy to get a Gray code for the unlabeled case by lacing patterns into this existing Gray code could, at best, produce a Hamiltonian path, not a Hamiltonian cycle.
This strategy in general might not be fruitful, as the cardinality of unlabeled partitions is not easily relatable to the cardinality of the labeled case.

\begin{question}
Is there a Gray code for the labeled bicolored $(m, n)$-partitions described above when both $m, n$ are odd?
Equivalently, is $K_{m, n}$ {\good} when both $m$ and $n$ are odd?
\end{question}

% \clearpage
\section{Parity argument via the count of acyclic orientations for multipaths}\label{app:ParityArgumentMultiPathDirectCount}

We refer to \Cref{ssec:SignatureMultiPath} for the definitions concerning multipaths, no result from \Cref{ssec:SignatureMultiPath} is needed.
We give an alternative proof that if a multipath is {\good} then its number of edges is even.
The first step is to count the number of (a)cyclic orientations of a multipath according to their rank with respect to the usual orientation $D^{\b\ell}_{u\to v}$ from $u$ to $v$.

\begin{proposition}\label{prop:NbrAOmultiPath}
For $\b \ell = (\ell_1, \dots, \ell_s)$, the number of acyclic orientations $A$ of $P^{\b \ell}$ such that $\rank_{D^{\b \ell}_{u\to v}} A = k$ is:
$$|\AODk[{D^{\b \ell}_{u\to v}}]| = \binom{\sum_{p=1}^s \ell_p}{k} - \sum_{\substack{S, T\subseteq[s]\\ S\cap T = \emptyset,~ S\ne \emptyset,~ T\ne \emptyset}} (-1)^{|S\sqcup T|}\cdot \binom{\sum_{p\notin S\sqcup T} \ell_p}{k - \sum_{p\in S} \ell_p}$$
\end{proposition}

\begin{proof}
Recall that $\AODk = \c O_{k}(D) \ssm \CODk$, and that $|\c O_{k}(D)| = \binom{|E|}{k}$ for all digraphs $D = (\V, E)$.
Hence, it is enough to count $\CODk[{D^{\b \ell}_{u\to v}}]$, remembering that $D^{\b\ell}_{u\to v}$ has $\sum_{p=1}^s \ell_p$ arcs.

We denote by $E_p$ the set of edges of the path $P^{\ell_p} = (u, i_{p, 2}, \dots, i_{p, \ell_p}, v)$, for $p\in [s]$.
For a subset $X$ of edges of $D^{\b\ell}_{u\to v}$, the orientation $(D^{\b\ell}_{u\to v})^X$ is cyclic if and only if $E_p\subseteq X$ and $E_q\cap X = \emptyset$ for some $p\ne q$, \ie if and only if one of the paths in the multipath is directed from $u$ to $v$ while another one is directed from $v$ to $u$.
We count such sets $X$ according to their size, using inclusion-exclusion.

For $k\geq 0$, and $S, T \subseteq[s]$ with $S\cap T = \emptyset$, and $S\ne \emptyset$ and $T\ne \emptyset$, let $\c X_{k, S, T}$ be the set of subsets $X$ of edges of $D^{\b\ell}_{u\to v}$ with $|X| = k$ such that $E_p\subseteq X$ for $p\in S$ and $E_q\cap X = \emptyset$ for $q\in T$.
Such a subset corresponds to picking $\sum_{p\in S}\ell_p$ prescribed edges, then choosing the remaining $k - \sum_{p\in S}\ell_p$ among the edges that are not in $\bigsqcup_{p\in {S\sqcup T}} E_p$.
Thus (by convention, if $a < 0$ or $b < 0$ or $a < b$, then $\binom{a}{b} = 0$): $|\c X_{k, S, T}| = \binom{\sum_{p \notin S\sqcup T} \ell_p}{k - \sum_{p\in S}\ell_p}$. 
% The only important property is that $|\c X_{k S, T}|$ is a binomial coefficient, its value does not matter.

The set of cyclic orientations of rank $k$, namely $\CODk[D^{\b\ell}_{u\to v}]$, is in bijection with $\bigcup_{S, T\subseteq[s]} \c X_{k, S, T}$ where the union is on all $S, T\subseteq[s]$ with $S\cap T = \emptyset$, $S\ne\emptyset$ and $T\ne \emptyset$.
The latter union is not disjoint, but according to the inclusion-exclusion principle, we get the following, proving the claim:
$$\Bigl|\CODk[D^{\b\ell}_{u\to v}]\Bigr| = \left|\bigcup_{\substack{S, T\subseteq[s] \\ S\cap T = \emptyset,~ S, T\ne\emptyset}} \c X_{k, S, T}\right| = \sum_{\substack{S, T\subseteq[s] \\ S\cap T = \emptyset,~ S, T\ne\emptyset}} (-1)^{|S\sqcup T|}\cdot|\c X_{k, S, T}|\hspace{0.5cm} \qedhere$$
\end{proof}

\begin{proposition}\label{prop:MultiPathIsNotGood}
If $\sum_p \ell_p$ is even, then the $(\ell_1, \dots, \ell_s)$-multipath is not {\good}.
\end{proposition}

\begin{proof}
Let $\b\ell = (\ell_1, \dots, \ell_s)$.
We will apply \Cref{lem:ParityArgument} (in the form discussed in \Cref{rmk:PartiyArgumentForCyclicOrientations}):
if $P^{\b\ell}$ is {\good}, then $\sum_{k} (-1)^k\cdot |\CODk[D^{\b\ell}_{u\to v}]| = 0$.
By \Cref{prop:NbrAOmultiPath}, we get:

\begin{align*}
\sum_{k} (-1)^k\cdot |\CODk[D^{\b\ell}_{u\to v}]| &= \sum_{\substack{S, T\subseteq[s] \\ S\cap T = \emptyset,~ S, T \ne\emptyset}} (-1)^{|S\sqcup T|}\cdot\left(\sum_{k} (-1)^k\cdot\binom{\sum_{p\notin S\sqcup T} \ell_p}{k - \sum_{p\in S} \ell_p}\right) \\
&= \sum_{\substack{S, T\subseteq[s] \\ S\cap T = \emptyset,~ S, T \ne\emptyset}} (-1)^{|S\sqcup T|}\cdot\left(\sum_{k} (-1)^{k + \sum_{p\in S}\ell_p}\cdot\binom{\sum_{p\notin S\sqcup T} \ell_p}{k}\right) 
\end{align*}

Recall that $\sum_k (-1)^k\cdot \binom{m}{k} = 0$ if $m\geq 1$ and equals $1$ if $m = 0$.
Hence, the rightmost sum is non-zero if and only if $S\sqcup T = [s]$.
We obtain

\[\sum_{k} (-1)^k\cdot |\CODk[D^{\b\ell}_{u\to v}]|
= \sum_{\substack{S\sqcup T = [s] \\ S\cap T = \emptyset,~ S, T \ne\emptyset}} (-1)^{|S\sqcup T|}\cdot (-1)^{\sum_{p\in S} \ell_p} = (-1)^s \cdot \sum_{\substack{S\subseteq [s]\\ S\ne\emptyset,~ S\ne[s]}} (-1)^{\sum_{p\in S} \ell_p}\] 

Note that $(-1)^{\sum_{p\in S}\ell_p} = 1$ if and only if $S$ contains an even number of odd values $\ell_p$.
Hence, denoting $A \eqdef \{p\in [s] ~;~ \ell_p \text{ is odd}\}$ and $B \eqdef \{p\in [s] ~;~ \ell_p \text{ is even}\}$, we get:
$\sum_{S\subseteq [s]} (-1)^{\sum_{p\in S} \ell_p} = \sum_{\substack{X\subseteq A \\ Y\subseteq B}} (-1)^{|X|} = 2^{|B|}\cdot\sum_{j=0}^{|A|}\binom{|A|}{j}(-1)^j = 2^{|B|}\cdot 0 = 0$.
Substituting in the above, we obtain:
\begin{align*}
\sum_{k} (-1)^k\cdot |\CODk[D^{\b\ell}_{u\to v}]| &= (-1)^s\cdot\left(\sum_{S\subseteq [s]} (-1)^{\sum_{p\in S}\ell_p} ~-~ (-1)^{0} - (-1)^{\sum_{p=1}^s \ell_p}\right) \\
&= (-1)^{s+1}\cdot\left(1 + (-1)^{\sum_{p=1}^s \ell_p}\right)
\end{align*}

By \Cref{lem:ParityArgument} (see \Cref{rmk:PartiyArgumentForCyclicOrientations}), if $D^{\b\ell}_{u\to v}$ is {\good}, then $\sum_{k} (-1)^k\cdot |\CODk[D^{\b\ell}_{u\to v}]| = 0$, implying that $(-1)^{\sum_{p=1}^s \ell_p} = -1$, \ie implying that $\sum_{p=1}^s \ell_p$ is odd.
\end{proof}

\begin{remark}
Note that, contrarily to \Cref{ssec:SignatureMultiPath}, the present reasoning does not allow computing the acyclic signature of multipaths, only determining whether it is $0$ or not.
\end{remark}

\section{Acyclic polynomials and \texorpdfstring{$\c{AO}$}{AO}-Hamiltonicity of some small graphs}\label{app:GraphExamples}

To motivate further research, for $22$ graphs (from the named-graph database of SageMath~\cite{Sage}), we computed the number of acyclic orientations $\psi(G)$, the acyclic signature $\sigma(G)$, and the acyclic polynomial $\Psi(D; t)$.
The chosen orientation $D$ of the graph $G$ is obtained by orienting each edge $ij$ of $G$ from $i$ to $j$ if $i < j$.
% We list the coefficients of $\Psi(D; t)$, \eg ``$1, 3, 5, 5, 3, 1$'' stands for $1 + 3t + 5t^2 + 5t^3 + 3t^4 + t^5$.
% Since these coefficients form a palindromic sequence, the reader cannot be mistaken.
For each graph, we address $\c{AO}$-Hamiltonicity.
If $\sigma(G) \ne 0$, then $G$ is not {\good}; otherwise we need to construct the graph of acyclic orientations $\AO$ and search for a Hamiltonian cycle in it.
The latter is a \textbf{NP}-hard problem.
Our computations took minutes for the Moser spindle, and did not finish within 24 hours for the Wagner graph.
% \clearpage
\input{graph_table_landscape}

% Restore normal ToC depth (if there is anything after the appendix)
\addtocontents{toc}{\protect\setcounter{tocdepth}{2}}

% \clearpage
\printbibliography

% \newpage
%\input{CubePrismETC}
\end{document}